\documentclass[aos]{imsart} % generate file witthout "Submitted to \ldots" header
\RequirePackage[authoryear]{natbib}%% uncomment this for author-year citations
\RequirePackage[colorlinks,citecolor=blue,urlcolor=blue]{hyperref}
\RequirePackage{graphicx}
\usepackage[reqno]{amsmath}
\RequirePackage{amsfonts,amssymb,amsthm, mathtools}
\makeatletter
\usepackage{xcolor}
\usepackage[final]{pdfpages}
\usepackage{ragged2e}
\usepackage{array}
\usepackage{tabularx}
\usepackage{multirow}
\usepackage{rotating}
\usepackage{eurosym}
\usepackage{dsfont}
\usepackage{lmodern}
\usepackage{ulem}
\usepackage{multicol}
\usepackage{hyperref}
\usepackage{etoolbox}
\usepackage{textcomp}
\usepackage{subcaption}
\usepackage{tikz}
\usepackage{tikz-cd}
\usetikzlibrary {positioning}
\usepackage{shadethm}
\usepackage{pdfpages}
\usepackage[utf8]{inputenc}
\usepackage{csquotes}
\usepackage[T1]{fontenc}
\usepackage[german, english]{babel}
\usepackage{latexsym}
\usepackage{nccmath}
\usepackage{color}
\usepackage{comment}
\usepackage{mathrsfs}
\usepackage{clipboard}
\usepackage{autobreak}
\usepackage{bm}
\usepackage{stackengine,scalerel}
\usepackage{cleveref}
\usepackage{stmaryrd}
\makeatother

\renewcommand{\thetable}{\arabic{section}.\arabic{table}}
\numberwithin{equation}{section}

\theoremstyle{plain} %default (text in italic)
\newtheorem{theorem}{Theorem}[section]
\newtheorem{lemma}[theorem]{Lemma}
\newtheorem{proposition}[theorem]{Proposition}
\newtheorem{corollary}[theorem]{Corollary}

\newtheorem{assumption}{Assumption}

\theoremstyle{remark}
\newtheorem{definition}[theorem]{Definition}

\newtheorem{example}[theorem]{Example}

\newtheorem{remark}[theorem]{Remark}

\def\tablename{\footnotesize Table}

\def\an{\text{an}}
\def\pa{\text{pa}}
\def\ne{\text{ne}}
\def\dis{\text{dis}}
\def\ch{\text{ch}}

\newcommand{\C}{\mathbb{C}} % komplexe
\newcommand{\R}{\mathbb{R}} % reelle
\newcommand{\Z}{\mathbb{Z}} % ganze
\newcommand{\N}{\mathbb{N}} % natuerliche
\newcommand{\BE}{\mathbb{E}}% Erwartungswert
\newcommand{\CX}{{X}}

\newcommand{\CY}{{Y}}

\newcommand{\BA}{\mathbf{A}}
\newcommand{\BB}{\mathbf{B}}
\newcommand{\BFC}{\mathbf{C}}

\newcommand{\by}{\mathbf{y}}
\newcommand{\BFE}{\mathbf{E}}
\newcommand{\BS}{\Sigma}

\newcommand{\HY}{{Y}}
\newcommand{\HZ}{{Z}}

\newcommand{\msep}{\bowtie_m}

\newcommand{\beao}{\begin{eqnarray*}}
\newcommand{\eeao}{\end{eqnarray*}\noindent}
\newcommand{\limm}{\underset{n \rightarrow \infty}{\text{l.i.m.}}}
\newcommand{\limhh}{\underset{h \rightarrow 0}{\text{lim\:}}}
\newcommand{\limh}{\underset{h \rightarrow 0}{\text{l.i.m.\:}}}
\newcommand{\limk}{\underset{k \rightarrow \infty}{\text{l.i.m.\:}}}

\newcommand{\uint}{\int_{-\infty}^{\infty}}

\newcommand{\dAB}{d_{AB}} %Dichteannahme
\newcommand{\inst}{\:\raisebox{2pt}{\tikz{\draw[-,densely dashed,line width = 0.5 pt](0,0) -- (5mm,0);}}\:}
\newcommand{\rarrow}{\:\raisebox{0pt}{\tikz{\draw[->,solid,line width = 0.5 pt](0,0) -- (5mm,0);}}\:}
\newcommand{\larrow}{\:\raisebox{0pt}{\tikz{\draw[->,solid,line width = 0.5 pt](5mm,0) -- (0,0);}}\:}

\newcommand{\nrarrow}{\:\raisebox{0pt}{\tikz{\draw[->,solid,line width = 0.5 pt](0,0) -- (5mm,0);
            \draw[-,solid,line width = 0.5 pt](1.5mm,-1mm) -- (2.5mm,1mm);}}\:}
\newcommand{\nrarrowzwei}{\:\raisebox{0pt}{\tikz{\draw[->,solid,line width = 0.5 pt](0,0) -- (5mm,0);
            \draw[-,solid,line width = 0.5 pt](1.5mm,-1mm) -- (2.5mm,1mm);}}}
\newcommand{\nrarrownull}{\nrarrowzwei_{0\:}}
\newcommand{\nrarrowinf}{\nrarrowzwei_{\infty \:}}
\newcommand{\nsimnull}{\nsim_{0\:}}
\newcommand{\nsiminf}{\nsim_{\infty\:}}

\allowdisplaybreaks %% display breaks for align
\definecolor{darkgreen}{RGB}{0,139,0}
\newcommand{\LS}[1]{{\color{purple} #1}}
\newcommand{\VF}[1]{{\color{darkgreen} #1}}

\begin{document}

\begin{frontmatter}
\title{Mixed orthogonality graphs for  \vspace*{0.2cm}  \\ \vspace*{0.2cm} continuous-time stationary  processes} %models with applications to   \\ \vspace*{0.2cm} MCAR processes}
\runtitle{Orthogonality graphs for continuous-time models}

\begin{aug}
{   \author{\fnms{Vicky} \snm{Fasen-Hartmann}\ead[label=e1]{vicky.fasen@kit.edu}\orcid{0000-0002-5758-1999}}
   \and
   \author{\fnms{Lea} \snm{Schenk}\ead[label=e2]{lea.schenk@kit.edu}\ead[label=e3]{}\orcid{0009-0009-6682-6597}}
}
%\thanksref{t1}\thankstext{t1}
\address{Institute of Stochastics, Karlsruhe Institute of Technology\\[2mm] \printead[presep={\ }]{e1,e2}}

% \address[A]{Institute of Stochastics, Karlsruhe Institute of Technology\printead[presep={,\ }]{e1}}
% \address[B]{Institute of Stochastics, Karlsruhe Institute of Technology\printead[presep={,\ }]{e2}}

%\thanksref{e3}
\thankstext{e3}{\textit{Funding:} This work is supported by the project “digiMINT”, which is a part of the “Qualitätsoffensive Lehrerbildung”, a joint initiative of the Federal Government and the Länder which aims to improve the quality of teacher training. The program is funded by the Federal Ministry of Education and Research. The authors are responsible for the content of this publication.}

%\textbf{Funding. } This project is part of the “Qualitätsoffensive Lehrerbildung”, a joint initiative of the Federal Government and the Länder which aims to improve the quality of teacher training. The programme is funded by the Federal Ministry of Education and Research. The authors are responsible for the content of this publication.

\runauthor{V. Fasen-Hartmann and L. Schenk}

\end{aug}

\begin{abstract}
In this paper, we introduce different concepts of Granger causality and contemporaneous correlation for multivariate stationary continuous-time processes to model different dependencies between the component processes. Several equivalent characterisations are given for the different definitions, in particular by orthogonal projections.
 %To visualise and analyse the different dependencies,  graphs are used where the components of the continuous-time process are represented by vertices, directed edges between the vertices indicate causal influences and undirected edges indicate contemporaneous uncorrelation between the component processes.
We then define two mixed graphs based on different definitions of Granger causality and contemporaneous correlation, the (mixed) orthogonality graph and the local (mixed) orthogonality graph. %to visualise and analyse the different dependencies in stationary continuous-time processes.
In these graphs, the components of the process are represented by vertices, directed edges between the vertices visualise Granger causal influences and undirected edges visualise contemporaneous correlation between the component processes. Further, we introduce various notions of Markov properties in analogy to \cite{EI10}, which relate paths in the graphs to different dependence structures of subprocesses, and we derive sufficient criteria for the (local) orthogonality graph to satisfy them. Finally, as an example,  for the popular multivariate continuous-time AR (MCAR) processes, we explicitly characterise the edges in the (local) orthogonality graph by the model parameters.
% The popular multivariate continuous-time AR (MCAR) processes satisfy our assumptions. For MCAR processes we show that the (local) orthogonality graphs can be characterised explicitly by the \mbox{model parameters.}
\end{abstract}



\end{frontmatter}
%\maketitle

\makeatletter
\renewcommand\l@subsection{\@dottedtocline{1}{1.5em}{2.9em}}
\makeatother
%{\rmfamily \tableofcontents}

%%%%%%%%%%%%%%%%%%%%%%%%%%%%%%%%%%%%%%%%%%%%%%%%%%%%%%%%%%%%%%%%%%%%%%%%%%%%%%%%%%

\section{Introduction}\label{sec:intro}
%Überschriften nur zur gedanklichen Struktur, kommen natürlich am Ende weg.

%%%%%%%%%%%%%%%%%%%%%%%%%%%%%%%%%%%%%%%%%%%%%%%%%%%%%
%\subsection*{Prozesse in stetiger Zeit und warum für stetige Zeit wichtig?}
%%%%%%%%%%%%%%%%%%%%%%%%%%%%%%%%%%%%%%%%%%%%%%%%%%%%%

In this paper, we define new notions of Granger causality and contemporaneous correlation specifically for multivariate stochastic processes in continuous time and visualise them in mixed graphs.
With the increasing interest in complex multivariate data sets and networks in diversified fields, the interest in graphical models develops rapidly, although the attempt to use graphical models for the visualisation and analysis of causal structures in stochastic models is quite old  \citep{Wright1921, Wright1934}. % in the context of linear structural equation systems.
The key advantage of graphical models is the simple and clear way to display the dependencies of stochastic processes.
%and the easy implementation on computers. This can be exploited in the study of causal structures of high-dimensional time series in \cite{EI07}.
We refer to the nice overview in \cite{Handbook:graphical} for the state of the art on the mathematical and statistical aspects of graphical models. % Graphical models are popular because they can represent the often complex dependence structure of the components of the stochastic process in a simple, clear graphical structure, and thus graphs provide a concise way to communicate the dependence structure.
In our graphical models, vertices represent the different component series  $\CY_v=(Y_v(t))_{t\in \R}$, $v\in V:=\{1,\ldots,k\}$, of an underlying continuous-time %multivariate
stochastic process $\CY_V=(Y_V(t))_{t\in \R}$. The vertices are connected with directed and undirected edges, which % visualise dynamic relations between the components and 
represent  Granger causalities and contemporaneous correlations, respectively.

%The goal of Granger causality is to understand the causal relationship between two time series $\CY_a$ and $\CY_b$.
The mathematical notion of causality was popularised by Clive W.~J.~Granger and Christopher A.~Sims. In his original work, \cite{GR69} used a linear vector autoregressive (VAR) model, whereas \cite{SI72} used a moving average (MA) model to understand the causal effects in a bivariate model; a detailed discussion of the relationships between Granger and Sims causality is given in  \cite{Kuersteiner2018}, see also \cite{DU98, Eichler:2013a}. Since then, their ideas have been extended in various ways and have been applied in diversified fields, such as neuroscience \citep{Bergmann21}, econometrics \citep{
Imbens:2022}, environmental science \citep{Cox:Popken}, genomics \citep{Heerah} and social systems \citep{Kuzma}. The recent publication of \cite{Shojaie:Fox} is an excellent review of Granger causality with its advances.

However, not every interesting relationship between two component series $\CY_a$ and $\CY_b$ is necessarily a causal relation and directed. But this does not diminish the importance of modelling such relationships. Some well-known examples are the correlation between the aggressive behaviour and the amount of time spent playing computer games each day \citep{LE11}
%the correlation between the increase in the stork population and the increase in out-of-hospital births \citep{HO04},
and the correlation between the number of infants who sleep with the light on and the number of people who develop myopia in later life \citep{ZA00}. %A widely discussed example in psychology without an underlying common cause can be found in \cite{JU69}.
To model such undirected relationships, we use contemporaneous correlation, a symmetric relation between $\CY_a$ and $\CY_b$. %In discrete time the idea is that $\CY_a$ and $\CY_b$ are contemporaneously uncorrelated if, given the amount of information provided by the past of $\CY_V$ up to time $t$, $\CY_a(t+1)$ and $\CY_b(t+1)$ are uncorrelated.

Our novel approach is to define concepts of Granger causality and contemporaneous correlation for continuous-time multivariate processes by orthogonal projections onto linear spaces generated by subprocesses, resulting in \textsl{conditional orthogonality} relations. For processes in discrete time, this attempt was already studied in \cite{FL85, DU98, EI07}.  In contrast to the other papers, \cite{EI07} even represents the conditional orthogonality relations of a discrete-time VAR process in a graph, where Granger causality models the directed influences and contemporaneous correlation the undirected influences.  An alternative approach is to use \textsl{conditional independence} relations using conditional expectations given $\sigma$-fields generated by subprocesses, see \cite{CH82, FL82, EI10} for discrete-time processes and \cite{CO96, FL96, PE12} for continuous-time processes and especially for semimartingales. \cite{CO96} propose to model undirected influences by global instantaneous causality and local instantaneous causality in continuous time, however, the results are not related to graphical models. Again, \cite{EI10} defines a graphical model for time series in discrete time representing the conditional independence relations using Granger causality for directed influences and contemporaneous conditional dependence for the undirected influences. 
%\LS{Hier sollte man beim ungerichteten vielleicht doch nochmal Eichler erwähnen}
For Gaussian random vectors, conditional independence and conditional orthogonality are equivalent, and the standard literature on graphical models for random vectors is based on conditional independence  \citep{LA04}. In non-Gaussian time series models, however, conditional expectations are much more difficult to compute than linear predictions, so we use conditional orthogonality. This is also reflected in the fact that the assumptions in \cite{EI10} to receive the Markov properties of the graphical time series models based on conditional independence are much more technical and difficult to verify than those in \cite{EI07} based on conditional orthogonality.

An extension of conditional independence is the concept of \textsl{local independence} for composable finite Markov processes of \cite{Schweder} which was generalised to semimartingales by \cite{Aalen87}. This concept has been applied to define and analyse the \textit{local independence graph}, e.g., in the context of composable finite Markov processes, point processes and physical systems in \cite{DI06, DI07, DI08, Eichler:Dahlhaus:Dueck, CO07, roysland2024graphical}. These definitions were recently taken up by \cite{Mogensen:Hansen:2020, Mogensen:Hansen:2022} who study (canonical) local independence graphs for It$\hat{\text{o}}$ processes. However, the results rely on the semimartingale property of such processes, but semimartingales do not seem to be the right tool for stationary time series models, especially for non-Gaussian models. Additionally,  \cite{Mogensen:Hansen:2022} assume continuous sample paths, which excludes Lévy-driven stochastic processes with jumps.

This paper is the first paper developing graphical models for conditional orthogonality relations of general stationary stochastic process in continuous-time. %  without any additional assumptions on the process. %In particular, we extend the
%definitions of Granger causality and contemporaneous correlation from \cite{EI07} for VAR models to continuous time.
%Therefore, the linear spaces generated by the underlying process $\CY_V$ and subprocesses must be chosen appropriately and adapted to our setting.
We also present  several equivalent characterisations of our concepts of Granger causality and contemporaneous correlation and relate them to other definitions in the literature. These definitions do not require the stationarity of $\CY_V$. Importantly, we define local versions of Granger causality and contemporaneous correlation, which are less strong.
Based on the different definitions of Granger causality and contemporaneous correlation, we then introduce two mixed graphs, the \textsl{(mixed) orthogonality graph} and the \textsl{local (mixed) orthogonality graph} for such multivariate stochastic processes in continuous time. For example, for an Ornstein-Uhlenbeck process, the two graphs may look like in \Cref{fig: Two graphical models b}. We can already see from this picture that the edges of the local orthogonality graph are also edges in the orthogonality graph.

\begin{figure}[ht]
  \begin{subfigure}{0.49\textwidth}
  \centering
    \begin{tikzpicture}[align=center,node distance=2cm and 2cm, semithick ,
		state/.style ={circle, draw,  text=black , minimum width =0.7 cm}, arrow/.style={-latex}, every loop/.style={}]

  		\node[state] (2) {2};
  		\node[state] (1) [below left of=2] {1};
  		\node[state] (3) [below right of=2] {3};

  		\path
   		(1) edge [densely dashed] node {} (3)
   		(2) edge [densely dashed] node {} (3)
            (1) edge [densely dashed] node {} (2)
            
   		(1) edge [arrow, bend left=17] node {} (3)
   		(1) edge [arrow, bend left=17] node {} (2)
            (2) edge [arrow, bend left=17] node {} (3)
   		(3) edge [arrow, bend left=17] node {} (2)
		;
 	\end{tikzpicture}  
  \caption{Orthogonality graph}
  \label{fig: Orthogonality graph a}
  \end{subfigure}
  \begin{subfigure}{0.49\textwidth}
  \centering
	\begin{tikzpicture}[align=center,node distance=2cm and 2cm, semithick ,
		state/.style ={circle, draw,  text=black , minimum width =0.7 cm}, arrow/.style={-latex}, every loop/.style={}]

  		\node[state] (2) {2};
  		\node[state] (1) [below left of=2] {1};
  		\node[state] (3) [below right of=2] {3};

  		\path
   		(1) edge [densely dashed] node {} (3)
     
   		(1) edge [arrow, bend left=17] node {} (3)
   		(2) edge [arrow, bend left=17] node {} (3)
            (3) edge [arrow, bend left=17] node {} (2)
		;
 	\end{tikzpicture}  
  \caption{Local orthogonality graph}
  \label{fig: Partial correlation graph b}
  \end{subfigure}
  \caption{In the left figure is the orthogonality graph and in right figure the local orthogonality graph of the Ornstein-Uhlenbeck process defined in \Cref{Example:Ornstein-Uhlenbeck process}.}
  \label{fig: Two graphical models b}
\end{figure}
\vspace{-0.6cm}
%certain directional and non-directional influences between them.
%%%%%%%%%%%%%%%%%%%%%%%%%%%%%%%%%%%%%%%%%%%%%%%%%%%%%
%\subsection*{Markoveigenschaften}
%%%%%%%%%%%%%%%%%%%%%%%%%%%%%%%%%%%%%%%%%%%%%%%%%%%%%
%\LS{The results of this paper allow us to derive general Granger non-causality relations between variables given only partial information, which is central to the discussion of spurious causality.} \marginpar{\LS{Satz passt meiner Meinung nach hier nicht mehr}}

%Here we already talk about the global Markov properties of graphs.

The causality structure of a graph is usually described by Markov properties. \cite{EI07, EI10} discusses Markov properties for mixed graphical models, namely the pairwise, local, block-recursive and two global Markov properties, using $m$-separation \citep{RI03} and $p$-separation \citep{Levitz}, respectively, for the global ones. For an asymmetric graph, \cite{DI08} develops and investigates an asymmetric notion of separation and discusses different levels of Markov properties. %Generalisations of these definitions are in \cite{Mogensen:Hansen:2022}. %, who study local independence graphs for It$\hat{\text{o}}$ processes.
In addition, \cite{Mogensen:Hansen:2022} show that the multivariate Ornstein-Uhlenbeck process driven by a Brownian motion is the only process that satisfies their global Markov property. As the above literature shows, the derivation of global Markov properties might be quite challenging and often it is only valid under additional or even restrictive assumptions.

In our (local) orthogonality graph, we show the pairwise, local and block-recursive Markov property 
%under additional assumptions,
and then discuss global Markov properties in both graphs. Importantly,  the orthogonality graph satisfies the global Andersson, Madigan and Perlman (AMP) Markov property \citep{AN00}, which is a sufficient criterion for conditional orthogonality. The assumptions on our orthogonality graph are quite general. We only require a stationary mean-square continuous stochastic process in continuous time {with expectation zero}, which is purely non-deterministic, with some restriction on the spectral density, which is, e.g., satisfied for Ornstein-Uhlenbeck and, more general, for continuous-time moving average (MCAR) processes.
Since the notion of $m$-separation in the AMP Markov property is strong, we present less restrictive alternatives and discuss the global Markov property of the orthogonality graph.
Although the local orthogonality graph also satisfies the pairwise, local and block-recursive Markov properties, not surprisingly stronger assumptions are required for global Markov properties.

%%%%%%%%%%%%%%%%%%%%%%%%%%%%%%%%%%%%%%%%%%%%%%%%%%%%%
%\subsection*{MCAR}
%%%%%%%%%%%%%%%%%%%%%%%%%%%%%%%%%%%%%%%%%%%%%%%%%%%%%

Finally, we derive the graphical structure of the popular multivariate continuous-time autoregressive (MCAR) processes driven by a general centred Lévy process with finite second moments, which are important extensions of their discrete-time counterparts. Different choices of the driving Lévy process and the model parameters, i.e., the parameters of the autoregressive polynomial and the covariance matrix of the driving Lévy process, allow quite flexible modelling of the margins, so MCAR processes form a broad class of processes. Special cases are the Gaussian MCAR processes, where the Brownian motion is the driving Lévy process and Ornstein-Uhlenbeck processes, which are MCAR(1) processes.
For general MCAR models, we derive that the (local) orthogonality graph is well defined and we explicitly characterise the different types of edges by the model parameters. These characterisations differ for the orthogonality and local orthogonality graph. Finally, we find analogues to the edge characterisations for vector autoregressive processes in \cite{EI07}.

Remarkably, in the case of Gaussian MCAR processes, our characterisations of
 local Granger causality and local contemporaneous correlation given by the model parameters, respectively, coincide with the characterisations of local Granger causality and local instantaneous causality in \cite{CO96}. However, our approach has several advantages. On the one hand, their theory is developed for semimartingales and several characterisations even assume continuous sample paths. But non-Gaussian Lévy-driven MCAR models have jumps and can therefore not be covered by their theory.  On the other hand, modelling the dependencies of the MCAR process in the local orthogonality graph allows to encode local Granger causalities and local contemporaneous correlations between multivariate subprocesses through the derived Markov properties. This is not content of \cite{CO96}.  Similarly, for Gaussian Ornstein-Uhlenbeck models, the local independence graph of \cite{Mogensen:Hansen:2022} coincides with our local causality graph. But their approach is based on Brownian motion driven It$\hat{\text{o}}$ processes, again excluding Lévy driven models or MCAR$(p)$ processes with $p\geq 2$.
 %But in contrast to these papers, our on (local) Granger causality and (local) contemporaneous correlation, also hold for very general non-Gaussian models which are, e.g., driven by the huge class of Lévy processes. Lévy driven MCAR processes which are not Gaussian  and hence, have jumps can not be investigated with the approaches in \cite{CO96}  and \cite{Mogensen:Hansen:2022} because both assume continuous sample paths.
 %\cite{CO96}  assume in several characterisations continuous sample paths, at least of the martingale part of their semimartingale, similarly \cite{Mogensen:Hansen:2022} assumes a classical It$\hat{\text{o}}$ processes where the driving process is a Brownian motion.
%Note that the graph Ornstein-Uhlenbeck process (GrOU) in \cite{Courgeau:Veraart:2022} is a network without defining the direction of the edges and is therefore not comparable to our model. 
To the best of our knowledge, our paper is the first on graphical properties of Lévy-driven MCAR models. It provides a generalisation of the results known from the literature to non-Gaussian processes. In \cite{VF23preb} we even develop extensions to the more general class of multivariate state space models based on the present paper, and in \cite{VF24} we present an undirected graphical model and relate it to the (local) orthogonality graph.

%%%%%%%%%%%%%%%%%%%%%%%%%%%%%%%%%%%%%%%%%%%%%%%%%%%%%%%%%%%%%%%%%%%%%%%%%%%%%%%%%%
\subsection*{Structure of the paper}
The paper is structured as follows. In \Cref{sec:prelim}, we first lay the foundation by introducing the conditional orthogonality relation as well as appropriate linear spaces generated by multivariate stochastic processes in continuous time and their properties which are important for this paper. We conclude the preliminaries with properties on mean-square differentiable stationary processes {with expectation zero}.
In Sections \ref{sec:influence} and \ref{sec:undirected_influences}, we then define, discuss, and relate different directed and undirected interactions between the component series of continuous-time stationary processes, i.e.,  Granger causality and contemporaneous correlation. This groundwork culminates in the definition of the orthogonality graph and the local orthogonality graph in \Cref{sec:path_diagrams}. For these orthogonality graphs, we prove several Markov properties. Finally, in \Cref{sec:CGMCAR}, we characterise the different graphical models for MCAR processes. The proofs of the paper are moved to the appendix.

%%%%%%%%%%%%%%%%%%%%%%%%%%%%%%%%%%%%%%%%%%%%%%%%%%%%%%%%%%%%%%%%%%%%%%%%%%%%%%%%%%
\subsection*{Notation}
Throughout the paper, $V=\{1,\ldots k\}$  and $\CY_V=(Y_V(t))_{t\in \R}$ denotes a $k$-dimensional (weakly) stationary stochastic process with expectation zero that is continuous in mean square. From now on we call the space of all real or complex $(k\times k)$-dimensional matrices $M_k(\R)$ and $M_k(\C)$, respectively. Similarly, $M_{k,d}(\R)$ and $M_{k,d}(\C)$ denote real and complex $(k\times d)$-dimensional matrices. We write $I_{k}$ for the $k$-dimensional identity matrix and $0_k$ for the $k$-dimensional zero matrix ($k\in \N$). With $\|\cdot\|$ we denote some matrix norm. The vector $e_v \in \R^k$ is the $v$-th unit vector and $\BFE_j^\top : = (0_{k\times k(j-1) }, I_{k},  0_{k\times k(p-j)}) \in M_{k \times kp}(\R)$, $j=1,\ldots,p$.
%\begin{align*}
%\BFE_j  =
%\begin{pmatrix}
%     0_{k(j-1) \times k} \\
%     I_{k} \\
%     0_{k(p-j) \times k}
%\end{pmatrix} \in M_{kp \times k}(\R), \quad j=1,\ldots,p.
%\end{align*}
For hermitian matrices \mbox{$A,B \in M_{k}(\C)$,} we write  $A \geq_L B$ if and only if $B-A$ is positive semi-definite, i.e.,~$B-A\geq 0$. Similarly, we write $A>0$ if $A$ is positive definite. Furthermore, $\sigma(A)$ are the eigenvalues of $A$. Finally, by $\text{l.i.m.}$ we denote the mean square limit.

%%%%%%%%%%%%%%%%%%%%%%%%%%%%%%%%%%%%%%%%%%%%%%%%%%%%%%%%%%%%%%%%%%%%%%%%%%%%%%%%%%
\section{Preliminaries}\label{sec:prelim}
In these preliminaries, we present some basics about the conditional orthogonality relation, such as the semi-graphoid property. Furthermore, we define the important linear spaces of this paper and give properties of mean-square differentiable stationary processes with expectation zero, which we use throughout the paper.
We start with some fundamentals on linear spaces in $L^2=L^2(\Omega, \mathcal{F}, \mathbb{P})$, the Hilbert space of square-integrable complex-valued random variables on a common probability space $(\Omega, \mathcal{F}, \mathbb{P})$. As usual, the inner product is $\langle X, Y \rangle_{L^2} = \BE[X \overline{Y}]$ for $X, Y\in L^2$ and orthogonality with respect to this inner product is denoted by $X \perp Y$. We set $\Vert X \Vert_{L^2}:=\sqrt{\langle X,X \rangle_{L^2}}$ for $X\in L^2$ and  identify random variables that are equal $\mathbb{P}$-a.s. Note that if $X_n \to_{L^2} X$ and $Y\in L^2$, then
\begin{align} \label{limit E}
    \lim_{n\to\infty}\BE(X_nY)=\BE(XY),
\end{align}
which can be shown by Cauchy-Schwarz inequality.
Further, suppose $\mathcal{L}_1$ and $\mathcal{L}_2$ are closed linear subspaces of $L^2$, where the closure is formed in the mean square. Then
\begin{align*}
\mathcal{L}_1^\bot=\{ X \in L^2: \langle X,Y \rangle_{L^2}=0 \text{ for all } Y \in \mathcal{L}_1\}
\end{align*}
is the orthogonal complement of $\mathcal{L}_1$. The sum of $\mathcal{L}_1$ and $\mathcal{L}_2$ is the linear vector space %denoted by
\begin{align*}
\mathcal{L}_1 + \mathcal{L}_2= \{ X+Y: X \in \mathcal{L}_1, Y \in \mathcal{L}_2 \}.
\end{align*}
Even when  $\mathcal{L}_1$ and $\mathcal{L}_2$ are closed subspaces, this sum may fail to be closed if both are infinite-dimensional. A classic example of this can be found in \cite{HA51}, p.~28. Hence, the closed direct sum is denoted by
\begin{align*}
\mathcal{L}_1 \vee \mathcal{L}_2=\overline{ \{ X+Y: X \in \mathcal{L}_1, Y \in \mathcal{L}_2 \}}.
\end{align*}
We further denote the orthogonal projection of $X\in L^2$ on $\mathcal{L}_1$ by $P_{\mathcal{L}_1}(X)=P_{\mathcal{L}_1}X$.
A review of the properties of orthogonal projections can be found, e.g., in \cite{WE80, BR91, LI15}.

%%%%%%%%%%%%%%%%%%%%%%%%%%%%%%%%%%%%%%%%%%%%%%%%%%%%%%%%%%%%%%%%%%%%%%%%%%%%%%%%%%
\subsection{Conditional orthogonality}
With those notations in mind, we define the conditional orthogonality relation as in \cite{EI07}, p.~347.

\begin{definition}
Let $\mathcal{L}_i$, $i=1,2,3$, be closed linear subspaces of $L^2$. Then $\mathcal{L}_1$ and $\mathcal{L}_2$ are \textsl{conditionally orthogonal} given $\mathcal{L}_3$ if 
\begin{align*}
X-P_{\mathcal{L}_3}X \perp Y-P_{\mathcal{L}_3}Y \quad \forall \:  X \in \mathcal{L}_1,\, Y \in \mathcal{L}_2.
\end{align*}
The conditional orthogonality relation is denoted by $\mathcal{L}_1 \perp \mathcal{L}_2 \: \vert \: \mathcal{L}_3$.
\end{definition}

Moreover, we summarise properties of the conditional orthogonality relation as given in \cite{EI07}, Proposition A.1.

\begin{lemma}\label{properties of conditional orthogonality}
Let $\mathcal{L}_i$, $i=1,\ldots,4$, be closed linear subspaces of $L^2$.  Then the conditional orthogonality relation defines a \textsl{semi-graphoid}, i.e., it satisfies the following properties:
\begin{enumerate}
\item[(C1)] Symmetry: $\mathcal{L}_1 \perp \mathcal{L}_2 \: \vert \: \mathcal{L}_3$ $\Rightarrow$ $\mathcal{L}_2 \perp \mathcal{L}_1 \: \vert \: \mathcal{L}_3$.
  \item[(C2)] (De-) Composition:
  $\mathcal{L}_1 \perp \mathcal{L}_2 \: \vert \: \mathcal{L}_4$ and $\mathcal{L}_1 \perp \mathcal{L}_3 \: \vert \: \mathcal{L}_4$
  $\Leftrightarrow$
  $\mathcal{L}_1 \perp \mathcal{L}_2 \vee \mathcal{L}_3 \: \vert \: \mathcal{L}_4$.
  \item[(C3)] Weak union:
   $\mathcal{L}_1 \perp  \mathcal{L}_2 \vee \mathcal{L}_3 \: \vert \: \mathcal{L}_4$
  $\Rightarrow$
  $\mathcal{L}_1 \perp  \mathcal{L}_2 \: \vert \:  \mathcal{L}_3 \vee \mathcal{L}_4$.
  \item[(C4)] Contraction:
    $\mathcal{L}_1 \perp  \mathcal{L}_2 \: \vert \: \mathcal{L}_4$ and  $\mathcal{L}_1 \perp   \mathcal{L}_3 \: \vert \: \mathcal{L}_2 \vee \mathcal{L}_4$
  $\Rightarrow$
  $\mathcal{L}_1 \perp  \mathcal{L}_2 \vee \mathcal{L}_3 \: \vert \:  \mathcal{L}_4$.
\end{enumerate}
If $(\mathcal{L}_2 \vee \mathcal{L}_4) \cap (\mathcal{L}_3 \vee \mathcal{L}_4) =  \mathcal{L}_4$ holds and $\mathcal{L}_2 \vee \mathcal{L}_3$ is separable, then the conditional orthogonality relation defines a \textsl{graphoid}, i.e., additionally we have:
\begin{itemize}
\item[(C5)] Intersection: $\mathcal{L}_1 \perp   \mathcal{L}_2 \: \vert \: \mathcal{L}_3 \vee \mathcal{L}_4$ and  $\mathcal{L}_1 \perp  \mathcal{L}_3 \: \vert \: \mathcal{L}_2 \vee \mathcal{L}_4$
  $\Rightarrow$
  $\mathcal{L}_1 \perp  \mathcal{L}_2 \vee \mathcal{L}_3 \: \vert \: \mathcal{L}_4$.
\end{itemize}
\end{lemma}

%The conditional orthogonality relation allows the analysis of linear relations between the processes $\CY_A$ and $\CY_B$ \textit{given} another process $\CY_C$. 
Note that the definition of conditional orthogonality reduces to the usual orthogonality when $\mathcal{L}_3=\{0\}$.
For a more detailed discussion of the conditional orthogonality relation, we refer to \cite{FL85}, who give the above results in terms of a general Hilbert space.

\begin{remark}
If  $(\mathcal{L}_2 \vee \mathcal{L}_4) \cap (\mathcal{L}_3 \vee \mathcal{L}_4) =  \mathcal{L}_4$ holds, we say that $\mathcal{L}_2$ are $\mathcal{L}_3$ \textsl{conditionally linearly separated} by $\mathcal{L}_4$ (cf.~\citealp{EI07}, p.~348).
\end{remark}

%%%%%%%%%%%%%%%%%%%%%%%%%%%%%%%%%%%%%%%%%%%%%%%%%%%%%%%%%%%%%%%%%%%%%%%%%%%%%%%%%%
\subsection{Linear subspaces}\label{subsec:linsub}
To apply the concept of conditional orthogonality to a multivariate stochastic process $\CY_V$, where $V =\{1,\ldots, k\}$, we define suitable closed linear subspaces. 
Let $A\subseteq V$, $s,t\in [-\infty,\infty]$ and $s\leq t$. Then we define the closed linear space 
\begin{align*}
\mathcal{L}_{Y_A}(s,t) &\coloneqq \overline{\text{span}}\left\{Y_a(u):\, a\in A,\, u\in[s,t]\cap\R\right\}%=\overline{\ell_{Y_A}(s,t)},
\end{align*}
with $\mathcal{L}_{Y_A}(-\infty,-\infty) \coloneqq \mathcal{L}_{Y_A}(\infty,\infty) \coloneqq \{0\}$ 
and use the shorthands 
\begin{align*}
 \mathcal{L}_{Y_A}(t) \coloneqq \mathcal{L}_{Y_A}(-\infty,t),\quad \quad \quad
\mathcal{L}_{Y_A}(-\infty) \coloneqq \bigcap\limits_{t\in \R}  \mathcal{L}_{Y_A}(t), \quad \quad \quad  \mathcal{L}_{Y_A}  \coloneqq {\mathcal{L}_{Y_A}(-\infty,\infty).} 
\end{align*}
Sometimes we use as well the linear space
\begin{align*}
\ell_{Y_A}(s,t) &\coloneqq \text{span}\left\{Y_a(u):\, a\in A,\, u\in[s,t]\cap\R\right\},
\end{align*}
whose closure is $\mathcal{L}_{Y_A}(s,t)$.
For further discussion and properties of such linear spaces, we refer to the early works of \cite{CR61, CR64, CR71}, but also to \cite{RO67, LI15, BrockwellLindner2024}. Furthermore, in \Cref{subsec:condorth} we derive sufficient criteria for conditional linear separation and separability of these linear spaces. The next lemma provides the basic properties of these linear spaces, which we use throughout the paper. The proof is given in the Supplementary Material \ref{Sec:Suppl}.

\begin{lemma}\label{additivity in time domain}\label{additivity in index set}
Let $A,B\subseteq V$, $s,t\in \R$, $s\leq t$. Then the following statements hold:
\begin{itemize}
    \item[(a)] %\begin{align*}
$\mathcal{L}_{Y_A}(s) \vee \mathcal{L}_{Y_A}(s,t) = \mathcal{L}_{Y_A}(t) $  $\mathbb{P}$-a.s.
    \item[(b)] $\mathcal{L}_{Y_A}(s,t) \vee \mathcal{L}_{Y_B}(s,t) = \mathcal{L}_{Y_{A \cup B}}(s,t)$  $\mathbb{P}$-a.s.
    \item[(c)] $\mathcal{L}_{Y_A}(t) \vee \mathcal{L}_{Y_B}(t)= \mathcal{L}_{Y_{A \cup B}}(t)$  $\mathbb{P}$-a.s.
    \item[(d)] $\overline{\bigcup_{n\in \N} \mathcal{L}_{Y_A}(n)} = \mathcal{L}_{Y_A}$ $\mathbb{P}$-a.s.
\end{itemize}
\end{lemma}

\subsection{Mean-square differentiable stationary processes}
To compute the mean-square derivative of a stationary continuous-time process $\CY_V$ with expectation zero, the following result of \cite{GI04}, IV.~§3, Corollary 2 is useful; see as well \cite{BrockwellLindner2024}, Example 5.17 and  \cite{DO60}, XI.~§9, Example 1. \newpage

%The proof can be done step by step as in \cite{DO60}, XI.~§9, Example 1.

%For more details on stationary processes we refer to \cite{ CR399, CR39, CR61, CR64, CR71, KO41, DO60, RO67, BR91, MA61, GL58}. \LS{Weniger?}
%For the theory of multivariate stationary processes we refer to
%\cite{CR399, CR39, CR61, CR64, CR67, CR71, KO41, MA61}.

\begin{proposition} \label{Proposition 3.7}
Let $\CY_V$ be a %$k$-dimensional
stationary process with expectation zero, spectral density $f_{Y_VY_V}(\lambda)$, $\lambda \in \R$, and  \textsl{spectral representation}
\begin{align}\label{spectral representation of stationary process}
Y_V(t) = \uint e^{i \lambda t} \Phi_V(d\lambda), \quad t\in\R,
\end{align}
where  $\Phi_V(\lambda)=(\Phi_1(\lambda),\ldots,\Phi_k(\lambda))^\top $ is a  random measure with
\begin{align*}
\BE[\Phi_V(d\lambda)] = 0_k \in \R^k \quad \text{ and } \quad
\BE[\Phi_V(d\lambda) \overline{\Phi_V(d\mu)}^\top ] =
%= \delta_{\lambda=\mu} F_{ Y_V Y_V}(d\lambda)
\delta_{\lambda=\mu} f_{Y_VY_V}(\lambda) d\lambda.
\end{align*}
 Then
 \begin{align*}
    \limh \frac{Y_V(t)-Y_{V}(t-h)}{h}
\end{align*}
exists if and only if $\int_{-\infty}^{\infty}\lambda^2\|f_{Y_VY_V}(\lambda)\|\,d\lambda<\infty$. In this case,
\begin{align*}
    D^{(1)}Y_V(t)\coloneqq\limh \frac{Y_V(t)-Y_{V}(t-h)}{h}=\uint i\lambda e^{i \lambda t} \Phi_V(d\lambda), \quad t\in\R.
\end{align*}
%Furthermore, the limit holds also $\mathbb{P}$-a.s., i.e.,
%\begin{align*}
%     D^{(1)}Y_V(t)=\lim_{h\to 0} \frac{Y_V(t)-Y_{V}(t-h)}{h} \quad \mathbb{P}\text{-a.s.}
%\end{align*}
\end{proposition}
Obviously, by recursion, we receive as well higher derivatives. Note that for a one-dimensional process $Y=(Y(t))_{t\in \R}$, the condition $\int_{-\infty}^{\infty}\lambda^2|f_{YY}(\lambda)|\,d\lambda<\infty$ is equivalent to the existence of $c_{YY}''(0)$, where $c_{YY}(t)$, $t\in \R$, is the autocovariance function of $Y$.

\begin{remark} \label{Remark 3.6}
Suppose $\CY_v$ is mean-square differentiable for some $v\in V$. Then
\begin{align*}
    D^{(1)}Y_v(t)=\underset{h \searrow 0}{\text{l.i.m.\:}}
    \frac{Y_v(t)-Y_v(t-h)}{h}\in \mathcal{L}_{Y_v}(t).
\end{align*}
Similarly, we are able to show by induction that if $\CY_v$ is $j_v$-times mean-square differentiable, then
%\begin{align*} %\label{eq1d}
    $D^{(j_v)}Y_v(t)\in \mathcal{L}_{Y_v}(t).$
%\end{align*}
\end{remark}
For further details on stationary processes, we refer to % the numerous publications of Cramér (see, e.g., \citealp{CR39, CR61}) as well as to 
the comprehensive works of \cite{DO60, RO67, LI15, BrockwellLindner2024}.

%%%%%%%%%%%%%%%%%%%%%%%%%%%%%%%%%%%%%%%%%%%%%%%%%%%%%%%%%%%%%%%%%%%%%%%%%%%%%%%%%%

\section{Directed influences: Granger causality for stationary continuous-time processes}\label{sec:influence}
%\subsection{Different concepts of Granger causality}\label{subsec:Granger causality}
In this section, we introduce and characterise directed influences between the component series of $\CY_V$ using different concepts of causality: \textit{local Granger causality}, \textit{Granger causality} and\textit{ global Granger causality}, where global Granger non-causality implies Granger non-causality which in turn implies local Granger non-causality. In  \Cref{sec:Proof:Granger_non-causal:project}, we present the proofs of the present section.

The idea of a Granger causal influence of one component series $\CY_a$ on another component series $\CY_b$ goes back to \cite{GR69}. In discrete time, the general idea that one process $\CY_a$ is Granger non-causal for another process $\CY_b$ is based on the question of whether the prediction of $Y_b(t+1)$ based on the information available at time $t$ provided by the past and present values of $\CY_V$ is diminished by removing the information provided by the past and present values of $\CY_a$. To transfer this approach to the continuous-time setting, we need to ask what it means to predict a time step into the future. As there is no obvious approach, we present the aforementioned three different concepts, motivated by other definitions of Granger causality in the literature. The first approach is the direct generalisation of \cite{EI07}, \mbox{Definition 2.2,} to continuous-time processes, considering one time step in the future.

%\LS{This time interval $[t,t+1]$ represents any context-dependent time interval $[t,t+c]$, $c>0$. For MCAR processes, we recognize in \Cref{sec:charMCAR} that such definitions are equivalent.}

\begin{definition}\label{(linear) Granger non-causal}
Let $A,B \subseteq S \subseteq V$ and $A\cap B=\emptyset$. Then $\CY_A$ is \textsl{Granger non-causal for $\CY_B$ with respect to $\CY_S$} if, for all $t\in \R$,
\begin{align*}
\mathcal{L}_{Y_B}(t,t+1) \perp \mathcal{L}_{Y_A}(t) \: \vert \: \mathcal{L}_{Y_{S \setminus A}}(t).
\end{align*}
 We write $\CY_A \nrarrow \CY_B \: \vert \: \CY_S$.
\end{definition}

\begin{remark} \label{Remark 3.2}
        In the definition of Granger causality, we use the time step $h=1$ because this is also done for discrete-time processes in  \cite{EI07} and it is the natural choice. Of course, it is also plausible to take some step size $h>0$ and define that $\CY_A$ is Granger non-causal for $\CY_B$ with respect to $\CY_S$ by 
        \begin{align} \label{eq:Remark:3.2}
            \mathcal{L}_{Y_B}(t,t+h) \perp \mathcal{L}_{Y_A}(t) \: \vert \: \mathcal{L}_{Y_{S \setminus A}}(t) \quad \forall \, t\in\R.
        \end{align}
        The results of this paper are straightforwardly transferable to this definition, but for ease of notation we stick to $h=1$. For popular examples such as the MCAR processes, see \Cref{Remark:step size}, and state space models \citep{VF23preb}, we recognise that for different $h$ these definitions are even equivalent. 
    %\end{itemize}
\end{remark}

In the next lemma, we present some equivalent characterisations of Granger causality, for completeness the proof is given in the \mbox{Supplementary Material \ref{Sec:Suppl}.} % whose proof is given in \Cref{sec:Proof:Granger_non-causal:project}.

\begin{lemma}\label{Charakterisisierung linear granger non-causal}\label{Charakterisisierung linear granger noncausa 2l}
Let $A,B \subseteq S \subseteq V$ and $A\cap B=\emptyset$.
Then the following statements are equivalent:
\begin{itemize}
    \item[(a)]    $\CY_A \nrarrow \CY_B \: \vert \: \CY_S$
    \item[(b)] $\mathcal{L}_{Y_B}(t+1) \perp \mathcal{L}_{Y_A}(t) \: \vert \: \mathcal{L}_{Y_{S \setminus A}}(t)$ $\forall \: t \in \R.$
    \item[(c)] $\ell_{Y_B}(t,t+1) \perp \ell_{Y_A}(-\infty,t) \: \vert \: \mathcal{L}_{Y_{S \setminus A}}(t)$ $\forall \: t \in \R.$
    \item[(d)] $\ell_{Y_b}(s,s) \perp \ell_{Y_a}(s',s') \: \vert \: \mathcal{L}_{Y_{S \setminus A}}(t)$
$\forall \; a \in A$, $b \in B$, $s\in [t,t+1]$, $s'\leq t$,  $t \in \R$.
\end{itemize}
\end{lemma}

%\VF{For the definition of Granger non-causality and its characterisations the assumption of stationarity is not needed and can be neglected. Only for the derivation of the Markov properties of the mixed orthogonality graphs and therefore, for its definition, in \Cref{sec:path_diagrams} we use it.}

The stationarity assumption is not necessary for the definition of Granger causality and its characterisations and can be neglected here. We first need it in \Cref{sec:path_diagrams}, e.g., for the intersection property (C5). %of conditional orthogonality to hold. % and the derivation of the Markov properties of mixed orthogonality graphs.

\begin{remark}
 The characterisation in \Cref{Charakterisisierung linear granger noncausa 2l} (b) is analogous to \cite{EI07}, Definition 2.2. The other characterisations are useful for checking Granger non-causality. In particular, we have  shown implicitly in \Cref{Charakterisisierung linear granger noncausa 2l} (d) that
\begin{align}\label{eq:correspondenceXBXbcausality}
    \CY_A \nrarrow \CY_B \: \vert \: \CY_S
    \quad \Leftrightarrow \quad
    \CY_A \nrarrow \CY_b \: \vert \: \CY_S
    \quad \forall\; b\in B.
\end{align}
\end{remark}

 From the characterisations in \Cref{Charakterisisierung linear granger non-causal}, the idea of Granger non-causality as equality of two predictions, as given, e.g., in \cite{DU98} for discrete-time processes, is not yet clear. Therefore, we provide another characterisation of Granger non-causality using orthogonal projections. % the proof of which is also given in \Cref{sec:Proof:Granger_non-causal:project}.

\begin{theorem}\label{Charakterisierung als Gleichheit der linearen Vorhersage}
Let $A,B \subseteq S \subseteq V$ and $A\cap B=\emptyset$. Then $\CY_A$ is Granger non-causal for $\CY_B$ with respect to $\CY_S$ if  for all $h\in[0,1]$, $t\in \R$, and $b\in B$,
\begin{align*}
P_{\mathcal{L}_{Y_S}(t)}Y_b(t+h) = P_{\mathcal{L}_{Y_{S \setminus A}}(t)}Y_b(t+h) \quad \mathbb{P}\text{-a.s.}
\end{align*}
\end{theorem}

In other words, the information given by the past process $(Y_A(s), s\leq t)$ can be forgotten without any consequences for the optimal linear prediction of $\CY_B(t+h)$ for $h\in[0,1]$. In particular, since $\mathcal{L}_{Y_{S \setminus A}}(t) \subseteq \mathcal{L}_{Y_{S \setminus \{a\}}}(t)\subseteq \mathcal{L}_{Y_{S}}(t)$ for any $a\in A$, we receive

\begin{align} \label{eq3b}
    \CY_A \nrarrow \CY_B \: \vert \: \CY_S
    \quad \Rightarrow \quad
    \CY_a \nrarrow \CY_b \: \vert \: \CY_S
    \quad \forall\;a\in A,\, b\in B.
\end{align}

Under some additional model assumptions %, in particular stationarity,
the opposite direction is also true. However, this is the topic of \Cref{sec:path_diagrams}.

\begin{remark}
\cite{FL96}, Definition 2.1, and \cite{CO96}, Definition 1, take a different approach to define Granger non-causality in continuous-time, using the equality of conditional expectations instead of orthogonal projections, and generated $\sigma$-fields instead of generated linear spaces. %as information sets.
\cite{CO96}, Definition 2, also defines a local version of Granger causality, called local instantaneous causality, in the context of semimartingales. In Proposition 1 they further relate it to the definition of \cite{Renault:Szafarz}, who study first-order stochastic differential equations. Instead of looking at the entire prediction time interval $[t,t+1]$, \cite{CO96} examine $[t,t+h]$ as $h \rightarrow 0$ and, to get non-trivial limits, they use difference quotients. They also note that the highest existing derivative of the process must always be examined to obtain a non-trivial criterion. Therefore, in the style of their characterisation of local Granger causality and our \Cref{Charakterisierung als Gleichheit der linearen Vorhersage}, we define the following version of local Granger causality which is, as we derive in \Cref{Lemma:Granger_causality_relations}, weaker as Granger causality.
\end{remark}

\begin{comment}
\cite{FL96}, Definition 2.1, and \cite{CO96}, Definition 1, take a different approach to generalise the one time step into the future to continuous-time. Instead of looking at the entire time interval $0 \leq h \leq 1$ they examine the limit $h \rightarrow 0$. They also note that the highest existing derivative of the process must always be examined to obtain a non-trivial criterion and one has to consider increments per unit of time.
However, they discuss conditional expectations as opposed to orthogonal projections as we use them in the present paper. We, therefore, define a local version of Granger causality in terms of orthogonal projections motivated from \Cref{Charakterisierung als Gleichheit der linearen Vorhersage} as follows.
\end{comment}

\begin{definition}\label{definition (linear) local Granger non-causal}
Suppose  $\CY_v=(Y_v(t))_{t\in\R}$ is $j_v$-times mean-square differentiable but the $(j_v+1)$-derivative does not exist for  $v\in V$. The $j_v$-derivative is denoted by $D^{(j_v)}\CY_v$, where  for $j_v=0$ we define $D^{(0)} Y_v=Y_v$. Let $A,B \subseteq S \subseteq V$ and $A\cap B=\emptyset$.
Then $\CY_A$ is \textsl{local Granger non-causal} for $\CY_B$ with respect to $\CY_S$ if,  for all $t\in \R$ and $b\in B$,
\begin{align*}
&\limh P_{\mathcal{L}_{Y_S}(t)} \left(\frac{D^{(j_b)} Y_b(t+h)- D^{(j_b)} Y_b(t)}{h}\right)\\
&\quad =\limh P_{\mathcal{L}_{Y_{S \setminus A}}(t)} \left(\frac{D^{(j_b)} Y_b(t+h)- D^{(j_b)} Y_b(t)}{h}\right)
 \quad \mathbb{P}\text{-a.s}.
\end{align*}
 We write $\CY_A \nrarrownull \CY_B \: \vert \: \CY_S$.
\end{definition}

\begin{remark} ~ \label{Remark:3.7}
\begin{itemize} 
    \item[(a)]
Since $Y_b$ is by assumption not $(j_b+1)$-times mean-square differentiable, the $L^2$-limit of $(D^{(j_b)} Y_b(t+h)- D^{(j_b)} Y_b(t))/h$ does not exist. However, it is still possible that the $L^2$-limit of $$P_{\mathcal{L}_{Y_S}(t)} \left(\frac{D^{(j_b)} Y_b(t+h)- D^{(j_b)} Y_b(t)}{h} \right)
\quad \text{and}\quad
P_{\mathcal{L}_{Y_{S\backslash A}}(t)} \left(\frac{D^{(j_b)} Y_b(t+h)- D^{(j_b)} Y_b(t)}{h}\right)$$ exist and only then local Granger non-causality is possible. 
\item[(b)] Typical examples of stochastic processes satisfying the assumptions of \Cref{definition (linear) local Granger non-causal} are MCAR processes (\Cref{sec:CGMCAR})
and the more general class of state space models (\citealp{VF23preb}) but as well fractional MCAR processes (\citealp{Marquardt,CO96}).
%In, we develop that for MCAR processes the limits indeed exist. \LS{In \cite{VF23preb}, we develop that the limit also exists for more general state space models. For, such limits exist for suitable non-integer derivatives. We refer to \cite{CO96}, who deal with the latter topic in the nonlinear context.}
\end{itemize}
\end{remark}

\begin{remark}
By definition we receive
\begin{align}\label{eq:correspondenceXBXbcausality local}
    \CY_A \nrarrownull \CY_B \: \vert \: \CY_S
    \quad \Leftrightarrow \quad
\CY_A \nrarrownull \CY_b \: \vert \: \CY_S \quad \forall\; b\in B.
\end{align}
Moreover, for $a\in A$, the subset relation $\mathcal{L}_{Y_{S \setminus A}}(t) \subseteq \mathcal{L}_{Y_{S \setminus \{a\}}}(t) \subseteq \mathcal{L}_{Y_{S}}(t)$ implies
\begin{align}\label{eq3.4}
    \CY_A \nrarrownull \CY_B \: \vert \: \CY_S
    \quad \Rightarrow \quad
    \CY_a \nrarrownull \CY_b \: \vert \: \CY_S
    \quad \forall\; a\in A,b\in B.
\end{align}
Again, the opposite direction is valid under some additional assumption, see \Cref{sec:path_diagrams}.
\end{remark}

Local Granger causality implies a kind of local version of conditional orthogonality.  %; see \Cref{sec:Proof:Granger_non-causal:project} for the proof.

\begin{theorem}\label{Charakterisierung local Granger non-causal}
Suppose  $\CY_v=(Y_v(t))_{t\in\R}$ is $j_v$-times mean-square differentiable but the $(j_v+1)$-derivative does not exist for $v\in V$.  %with $j_v$-derivative denoted by $D^{(j_v)}\CY_v$ for all $v\in V$, where  for $j_v=0$ we define $D^{(0)} Y_v(t)=Y_v(t)$.
Further, let $A,B \subseteq S \subseteq V$ and $A\cap B=\emptyset$. Then $\CY_A \nrarrownull \CY_B \: \vert \: \CY_S$ implies that, for all $Y^A \in \mathcal{L}_{Y_{A}}(t)$ and $t\in \R$,
\begin{align*}
\lim_{h \rightarrow 0} \frac{1}{h} \BE \left[
\left(D^{(j_b)} Y_b(t+h) - P_{\mathcal{L}_{Y_{S \setminus A}}(t)} D^{(j_b)} Y_b(t+h)\right)
\overline{\left(Y^A - P_{\mathcal{L}_{Y_{S \setminus A}}(t)} Y^A \right)}
\right]=0.
\end{align*}
\end{theorem}

 %Note that in \Cref{subsec:condorth} we discuss  sufficient criteria for $\mathcal{L}_{Y_S}(t)= \mathcal{L}_{Y_A}(t)+\mathcal{L}_{Y_{S\setminus A}}(t)$, where we use as well that our process $\CY_V$ is stationary.
%\LS{Hier kann man nicht zu Lindquist und muss den Beweis komplett angeben. Kürzt man also vorne, muss hier ausführlich sein. Ist vorne ausführlich, dann hier kürzbar und nur auf Grenzwert verweisen. Beweis im Anhang Korrektur lesen.}

A third concept of directed influence is to consider causality up to an arbitrary horizon. In discrete time, the concept of causality at any horizon goes back to the seminal work of \cite{SI72} and is also called Sims causality. We introduce the following definition as a generalisation of \cite{EI07}, Definition 4.4, to continuous-time processes.

\begin{definition}\label{(linear) global Granger non-causal at different horizons}
Let $A,B \subseteq S \subseteq V$ and $A\cap B=\emptyset$.
%Then $\CY_A$ is \textsl{(linear) Granger non-causal for $\CY_B$ with respect to $\CY_S$ up to horizon $h$}, $h\in \R$, if and only if
%\begin{align*}
%\mathcal{L}_{Y_B}(t,t+h) \perp \mathcal{L}_{Y_A}(t) \: \vert \: \mathcal{L}_{Y_{S \setminus A}}(t),
%\end{align*}
%for $t\in \R$. We write $\CY_A  \stackrel{h}{\nrarrow} \CY_B \: \vert \: \CY_S$.
Then $\CY_A$ is \textsl{global Granger non-causal for $\CY_B$ with respect to $\CY_S$} if, for all $h\geq 0$ and $t\in \R$,
\begin{align*}
\mathcal{L}_{Y_B}(t,t+h) \perp \mathcal{L}_{Y_A}(t) \: \vert \: \mathcal{L}_{Y_{S \setminus A}}(t).
\end{align*}
 We write $\CY_A \nrarrowinf \CY_B \: \vert \: \CY_S$.
\end{definition}
%\marginpar{Globale Definition ggf. mit Raum bis unendlich abändern}

The study of such long-run effects is a useful complement to understanding the relationship between the component series and allows us to distinguish between short-run and long-run causality.

\begin{remark}
The characterisations are similar to those for Granger causality. %In particular, $\CY_A$ is (linear) Granger non-causal for $\CY_B$ with respect to $\CY_S$ up to horizon $h$, $h \in \R$, if and only if
%\begin{align*}
%P_{\mathcal{L}_{Y_S}(t)}Y_b(t+h') = P_{\mathcal{L}_{Y_{S \setminus A}}(t)}Y_b(t+h'),
%\end{align*}
%$\mathbb{P}$-a.s. for all $0 \leq h' \leq h$, $t\in \R$ and $b\in B$.
In particular, $\CY_A$ is global Granger non-causal for $\CY_B$ with respect to $\CY_S$, if and only if, for all $h \geq 0$, $t\in \R$ and $b\in B$,
\begin{align}\label{Char xy}
P_{\mathcal{L}_{Y_S}(t)}Y_b(t+h) = P_{\mathcal{L}_{Y_{S \setminus A}}(t)}Y_b(t+h) \quad \mathbb{P}\text{-a.s.}
\end{align}
On the one hand, note that the proof is similar to the proof of \Cref{Charakterisierung als Gleichheit der linearen Vorhersage} and on the other hand, that analogue relationships as in \eqref{eq:correspondenceXBXbcausality} and \eqref{eq3b} hold. The characterisation \eqref{Char xy} is again consistent with the characerisation in \cite{DU98} for discrete-time processes and with the definition of global Granger causality in \cite{CO96}, who use generated $\sigma$-fields instead of linear spaces and conditional expectations instead of orthogonal projections. Of course, for Gaussian processes, the two definitions coincide.
\end{remark}

%%%%%%%%%%%%%%%%%%%%%%%%%%%%%%%%%%%%%%%%%%%%%%%%%%%%%%%%%%%%%%%%%%%%%%%%%%%%%%%%%%
%\subsection{Relations of the different concepts of Granger causality} \label{sec:relations_Granger_Causality}

In the following lemma, we state relations between Granger non-causality, local Granger non-causality and global Granger non-causality. %We refer to  \Cref{sec:Proof:Granger_non-causal:project} for the proof.
See again \cite{Kuersteiner2018, DU98, Eichler:2013a} for the relations between the different definitions for discrete-time processes.
%also discuss the relations between the three definitions of causality in discrete-time. Yet, their proof is based on the equality of orthogonal projections instead of  conditional orthogonality. \VF{Aber wir haben auch eine Formulierung mit Projektionen. Vermutlich hätte man, wenn man gewollt hätte, den Beweis auch mit Projektionen führen können. So finde ich den Satz missverständlich}

\begin{lemma}\label{Lemma:Granger_causality_relations}
%Let $\CY_V=(Y_V(t))_{t\in \R}$ be a $k$-dimensional process that satisfies \Cref{Assumption 1 und 2} and
Let $A,B \subseteq S \subseteq V$ and $A\cap B=\emptyset$. %, $h \in \R$.
Then the following implications hold:
\begin{itemize}
    \item[(a)]  \makebox[3.17cm][l]{$\CY_A  \nrarrowinf \CY_B \: \vert \: \CY_{S }$}
    $\Rightarrow \quad \CY_A  \nrarrow \CY_B \: \vert \: \CY_{S}.$
    %\CY_A  \stackrel{h}{\nrarrow} \CY_B \: \vert \: \CY_S \quad \text{and} \quad
\item[(b)] \makebox[3.16cm][l]{$\CY_A  \nrarrowinf \CY_{S\setminus A} \: \vert \: \CY_{S}$}
$\Leftrightarrow \quad \CY_A  \nrarrow \CY_{S\setminus A} \: \vert \: \CY_S.$
%$\CY_A  \stackrel{\infty}{\nrarrow} \CY_B \: \vert \: \CY_{S \setminus A}
%\CY_A  \stackrel{h}{\nrarrow} \CY_B \: \vert \: \CY_S \quad \text{and} \quad
%\CY_A  \nrarrow \CY_B \: \vert \: \CY_{S \setminus A}$.
%\CY_A  \nrarrow \CY_{S\setminus A} \: \vert \: \CY_S \quad \Rightarrow \quad
\item[(c)] \makebox[3.195cm][l]{$\CY_A  \nrarrow \CY_{S\setminus A} \: \vert \: \CY_S$}
$\Rightarrow \quad \CY_A  \nrarrowinf \CY_{B} \: \vert \: \CY_S.$
\item[(d)] % Suppose $\CY_V$ is not differentiable. Then
\makebox[3.15cm][l]{$\CY_A  \nrarrow \CY_B \: \vert \: \CY_S$}
$\Rightarrow \quad \CY_A  \nrarrownull \CY_B \: \vert \: \CY_S.$
\end{itemize}
\end{lemma}

\begin{remark} The opposite direction in \Cref{Lemma:Granger_causality_relations} (a) does not hold in general. \cite{DU98}, p.~1106,  present a counterexample in discrete time and explain the lack of equivalence between Granger non-causality and global Granger non-causality as follows. If there are auxiliary components, $\CY_A$ might not help to predict $\CY_B$ given $\CY_S$ one step ahead but $\CY_A$ might help to predict $\CY_B$ given $\CY_S$ several periods ahead. For example, the values  of $\CY_A$ up to time $t$ may help to predict $\mathcal{L}_{Y_B}(t+1,t+2)$, even though they are useless to predict $\mathcal{L}_{Y_B}(t,t+1)$, because $\CY_A$ may help to predict the environment one period ahead, which in turn influences $\CY_A$ at a subsequent period.
%However, $\CY_A  \nrarrow \CY_{S\setminus A} \: \vert \: \CY_S$ implies $\CY_A  \stackrel{\infty}{\nrarrow} \CY_{S\setminus A} \: \vert \: \CY_S$ and
Therefore, it is also not surprising that we have equivalence in the case without environment in \Cref{Lemma:Granger_causality_relations} (b).
%\begin{align*}
%\CY_A  \nrarrow \CY_{S\setminus A} \: \vert \: \CY_S \quad \Leftrightarrow \quad \CY_A  \nrarrowinf \CY_{S\setminus A} \: \vert \: \CY_S.
%\end{align*}
This holds in particular for every bivariate process, i.e.,
\begin{align*}
\CY_a  \nrarrow \CY_{b} \: \vert \: \CY_{\{a,b\}} \quad \Leftrightarrow \quad \CY_a  \nrarrowinf \CY_{b} \: \vert \: \CY_{\{a,b\}}.
\end{align*}
\end{remark}

The similarities and differences between the various definitions of Granger causality can also be seen in examples, so we examine Ornstein-Uhlenbeck processes. In particular, we see that the opposite direction of \Cref{Lemma:Granger_causality_relations} (d) does not generally hold.

%that from $\CY_a \stackrel{0}{\nrarrow} \CY_b \: \vert \: \CY_V$ we can in general not follow that $\CY_a \nrarrow \CY_b \: \vert \: \CY_V$.

\begin{example} \label{Example:Ornstein-Uhlenbeck process} 
Suppose $Y_V=(Y_V(t))_{t\in \R}$ is an Ornstein-Uhlenbeck process driven by a two-sided $k$-dimensional Lévy process $(L(t))_{t\in \R}$. An one-sided Lévy process $(L(t))_{t\geq 0}$ is an $\R^k$-valued stochastic process with $L(0)=0_k$ $\mathbb{P}$-a.s., stationary and independent increments and c\`adl\`ag sample paths. Now, $L=(L(t))_{t\in \R}$ is obtained from two independent copies $(L_1(t))_{t\geq 0}$ and $(L_2(t))_{t\geq 0}$ of a one-sided Lévy process
via
%\begin{align*}
$L(t) =
L_1(t)$ if $ t \geq 0$ and
$L(t)=-\lim_{s\nearrow -t} L_2(s) $ if  $t<0.$
We assume that the Lévy process has a finite second moment with $\Sigma_L \coloneqq \BE[L(1)L(1)^\top]$ and expectation zero. % For more details on Lévy processes, see \cite{SA07}.
Suppose further that $\BA \in M_k(\R)$ with $\sigma(\BA)\subseteq (-\infty, 0) + i \R$. Then the stochastic differential equation
\begin{align*}
dY_V(t) = \BA Y_V(t) dt +dL(t)
\end{align*}
has the unique stationary %$k$-dimensional
solution $\CY_V$ given by
\begin{align*}
Y_V(t) = \int_{- \infty}^{t} e^{\BA (t-u)} dL(u), \quad t\in\R.
\end{align*}
The process $\CY_V$ is called (causal) \textsl{Ornstein-Uhlenbeck process} (cf.~\citealp{Masuda:2004}). For the Ornstein-Uhlenbeck process, we derive in \Cref{sec:CGMCAR}, in the more general context of (causal) MCAR processes, that
\begin{align*}
\CY_a \nrarrowinf \CY_b \: \vert \: \CY_V
\quad \Leftrightarrow \quad \CY_a \nrarrow \CY_b \: \vert \: \CY_V
\quad &\Leftrightarrow \quad
[\BA^{\alpha}]_{ab}=0, \quad \alpha=1,\ldots,k-1,
\\
\CY_a \nrarrownull \CY_b \: \vert \: \CY_V
\quad &\Leftrightarrow \quad
[\BA]_{ab}=0.
%\CY_a \stackrel{h}{\nrarrow} \CY_b \: \vert \: \CY_V
%\quad & \Leftrightarrow \quad
%[A^{\alpha}]_{ab}=0, \quad \alpha=1,\ldots,k-1,
%\\
\end{align*}
Of course,
\begin{align*}
   \CY_a \nrarrow \CY_b \: \vert \: \CY_V  \hspace*{0.3cm} \Rightarrow \hspace*{0.3cm} [\BA^{\alpha}]_{ab}=0, \; \alpha=1,\ldots,k-1
    \hspace*{0.3cm} \Rightarrow \hspace*{0.3cm} [\BA]_{ab}=0
    \hspace*{0.3cm} \Rightarrow \hspace*{0.3cm} \CY_a \nrarrownull \CY_b \: \vert \: \CY_V,
\end{align*}
but the opposite direction does not generally hold, an exception is the case where $\BA$ is a diagonal matrix. A specific counterexample is the Ornstein-Uhlenbeck process with 
\begin{align} \label{MM}  
   \BA=\begin{pmatrix} 
   \makebox[0.6cm]{-2} & \makebox[0.6cm]{0} & \makebox[0.6cm]{0} \\ 
   0 & -2 & 1 
   \\ 1 & 1 & -2\end{pmatrix}
   \quad \text{ and } \quad
   \Sigma_L=\begin{pmatrix} 
   \makebox[0.6cm]{1} & \makebox[0.6cm]{0} & \makebox[0.6cm]{1/2} \\ 
   0 & 1 & 0 \\ 
   1/2 & 0 & 1\end{pmatrix},
   %\BA^2=\begin{pmatrix} 
   %\makebox[0.6cm]{4} & \makebox[0.6cm]{0} & \makebox[0.6cm]{0} \\ 
   %1 & 5 & -4 \\ 
   %-4 & -4 & 5\end{pmatrix},
\end{align}
which is the underlying stochastic process of  \Cref{fig: Two graphical models b}. Here, $\CY_1 \nrarrownull \CY_2 \: \vert \: \CY_{\{1,2,3\}}$ but $\CY_1 \rarrow \CY_2 \: \vert \: \CY_{\{1,2,3\}}$.  It is clear from the example that Granger non-causality is much stronger than local Granger non-causality, and that in general there is no equivalence.
Note that the special structure of $\Sigma_L$ does not play a role in these directed influences, but the covariance structure has an impact on the undirected influences which we will define in the next section.
\end{example}

%\begin{remark}
 %   Apart from Eichler's concepts in discrete time,  there have already been other suggestions to define Granger non-causality in continuous time at any horizon.  introduce the concept of weak global Granger non-causality for real-valued stochastic processes; call it \textsl{global Granger non-causality}. However, their definition is based on conditional expectations instead of orthogonal projections and therefore a comparison of the different definitions of Granger causality is only possible in cases where the conditional expectation coincides with the projection on the linear space as, e.g., for Gaussian processes. In this case, it is possible to show that weak global Granger non-causality and Granger causality at any horizon are equivalent; for details see \cite{Schenk:PhD}.
%\end{remark}

%%%%%%%%%%%%%%%%%%%%%%%%%%%%%%%%%%%%%%%%%%%%%%%%%%%%%%%%%%%%%%%%%%%%%%%%%%%%%%%%%%

\section{Undirected influences: Contemporaneous correlation for stationary continuous-time processes} \label{sec:undirected_influences}
%%%%%%%%%%%%%%%%%%%%%%%%%%%%%%%%%%%%%%%%%%%%%%%%%%%%%%%%%%%%%%%%%%%%%%%%%%%%%%%%%%

%%%%%%%%%%%%%%%%%%%%%%%%%%%%%%%%%%%%%%%%%%%%%%%%%%%%%%%%%%%%%%%%%%%%%%%%%%%%%%%%%%
In this section, we introduce and characterise undirected influences between the component series of $\CY_V$ using different concepts of contemporaneous correlation. % and restricting ourselves again to linear influences. % although the main interest is usually in Granger causality.
The idea is simple: There is no undirected influence between $\CY_a$ and $\CY_b$, if and only if, given the amount of information provided by the past of $\CY_V$ up to time $t$, $\CY_a$ and $\CY_b$ are uncorrelated in the future.
Again, we need to specify what we mean by the future in continuous time. The first definition is a generalisation of \cite{EI07}, Definition 2.2, in discrete time, to continuous time, looking at the entire time interval $[t,t+1]$.

\begin{definition}\label{(linear) global contemporaneously uncorrelated}
Let $A,B \subseteq S \subseteq V$ and $A\cap B=\emptyset$. Then $\CY_A$ and $\CY_B$ are \textsl{contemporaneously uncorrelated with respect to $\CY_S$} if, for all $t\in \R$,
\begin{align*}
\mathcal{L}_{Y_A}(t,t+1) \perp \mathcal{L}_{Y_B}(t,t+1) \: \vert \: \mathcal{L}_{Y_{S}}(t).
\end{align*}
 We write $\CY_A \nsim \CY_B \: \vert \: \CY_S$.
\end{definition}

\begin{remark}\label{Remark 4.2}
 Similarly, as for the definition of Granger causality, we defined contemporaneous uncorrelation by using the step size $h=1$. However, it is also possible to use some arbitrary but fixed step size $h>0$ and define it via
 \begin{align} \label{eqref:Remark 4.2}
\mathcal{L}_{Y_A}(t,t+h) \perp \mathcal{L}_{Y_B}(t,t+h) \: \vert \: \mathcal{L}_{Y_{S}}(t) \quad \forall \, t\in\R.
\end{align} 
The choice of $h$ has no effect on the characterisation of the undirected influences in certain models; see \Cref{Remark:step size} for MCAR processes and \cite{VF23preb} for state space models. 
\end{remark}

Unlike Granger causality, contemporaneous correlation is symmetric, reflecting an undirected influence. By analogy with \Cref{Charakterisisierung linear granger non-causal}, we obtain the following equivalent characterisations of contemporaneous uncorrelation. Since the proof is very similar, it is not given here. Again, the stationarity assumption is not necessary for the definition of contemporaneous uncorrelation and its characterisations, it can be neglected.

\begin{lemma}\label{Charakterisisierung contemporaneously uncorrelated}
%Let $\CY_V=(Y_V(t))_{t\in \R}$ be a $k$-dimensional process that satisfies \Cref{Assumption 1 und 2} and
Let $A,B \subseteq S \subseteq V$ and $A\cap B=\emptyset$.
Then the following characterisations are equivalent:
\begin{itemize}
    \item[(a)]  $\CY_A \nsim \CY_B \: \vert \: \CY_S$.
    \item[(b)] $\mathcal{L}_{Y_A}(t+1) \perp \mathcal{L}_{Y_B}(t+1) \: \vert \: \mathcal{L}_{Y_S}(t)$  $\forall \: t \in \R.$
    \item[(c)] $\ell_{Y_A}(t,t+1) \perp \ell_{Y_B}(t,t+1) \: \vert \: \mathcal{L}_{Y_S}(t)$  $\forall \: t \in \R.$
    \item[(d)] $\ell_{Y_a}(s,s) \perp \ell_{Y_b}(s',s') \: \vert \: \mathcal{L}_{Y_S}(t)$
 $\forall\,a \in A$, $b \in B$, $s, s'\in [t,t+1]$, $t \in \R$.
\end{itemize}
\end{lemma}

\begin{remark}\label{eq:correspondenceXAXa}
In the following, we make some remarks about \Cref{Charakterisisierung contemporaneously uncorrelated} (d).
\begin{itemize}
    \item[(a)] In \Cref{Charakterisisierung contemporaneously uncorrelated} (d), we have implicitly shown that
\begin{align*} %\label{eq:correspondenceXAXa}
    \CY_A \nsim \CY_B \: \vert \: \CY_S
    \quad \Leftrightarrow \quad
    \CY_a \nsim \CY_b \: \vert \: \CY_S
    \quad
   \forall\; a\in A,\, b\in B,
\end{align*}
which is useful for the verification of contemporaneous uncorrelation.
\item[(b)] Given our \Cref{Charakterisisierung contemporaneously uncorrelated} (d) and \cite{EI07}, Definition 2.2, it would also be plausible to define contemporaneous uncorrelation by
$\ell_{Y_a}(s,s) \perp \ell_{Y_b}(s,s) \: \vert \: \mathcal{L}_{Y_S}(t)$ \linebreak
 $\forall\,a \in A$, $b \in B$, $s\in[t,t+1]$, $t \in \R$. In this case, however, no global Markov property can be shown in the associated orthogonality graph (cf.~\Cref{sec:path_diagrams}), since the evidences rely heavily on \Cref{(linear) global contemporaneously uncorrelated} and \Cref{properties of conditional orthogonality}.
\end{itemize}
\end{remark}

Similar to Granger non-causality, a characterisation of contemporaneous uncorrelation can be given, which allows for an interpretation as the correspondence of two linear predictions.

\begin{theorem}\label{Charakterisierung als Gleichheit der linearen Vorhersage 2}
Let $A,B \subseteq S \subseteq V$ and $A\cap B=\emptyset$. Then $\CY_A$ and $\CY_B$ are contemporaneously uncorrelated with respect to $\CY_S$, if and only if, for all $b\in B$, $h\in[0,1]$, \mbox{and $t \in \R$,}
\begin{align*}
P_{\mathcal{L}_{Y_{S}}(t)\vee \mathcal{L}_{Y_{A}}(t,t+1)} Y_b(t+h)  = P_{\mathcal{L}_{Y_{S}}(t)} Y_b(t+h) \quad \mathbb{P}\text{-a.s.}
\end{align*}
 %Analogous $\CY_A$ and $\CY_B$ are contemporaneously uncorrelated with respect to $\CY_S$, if and only if
%\begin{align*}
%P_{\mathcal{L}_{Y_{S}}(t)\vee \mathcal{L}_{Y_{A}}(t,t+1)} Y_b(t+h) = P_{\mathcal{L}_{Y_{S}}(t)}Y_b(t+h)
%\end{align*}
%$\mathbb{P}$-a.s for all $b\in B$, $0\leq h \leq 1$ and $t \in \R$.
\end{theorem}
 In words, the linear prediction of the information about $\CY_B$ in the near future based on $\mathcal{L}_{Y_S}(t)$ can not be improved by adding further information about $\CY_A$ in the near future (and vice versa).
The proof is again similar to the proof of \Cref{Charakterisierung als Gleichheit der linearen Vorhersage} and we therefore skip the details.

To define a local version of contemporaneous uncorrelation, note that the characterisation in \Cref{Charakterisisierung contemporaneously uncorrelated} (b) means that for any $\HY^A\in\mathcal{L}_{Y_A}(t+1)$ and $\HY^B\in\mathcal{L}_{Y_B}(t+1)$ %we have
\begin{align} \label{MR}
    \BE\left[\left( \HY^A- P_{\mathcal{L}_{Y_{S}}(t)}\HY^A  \right) \overline{\left(\HY^B- P_{\mathcal{L}_{Y_{S}}(t)} \HY^B \right)} \right]=0.
\end{align}
So the motivation for the local version is that instead of taking all  $\HY^A\in\mathcal{L}_{Y_A}(t+1)$, we use only the highest derivative $D^{(j_a)} Y_a(t+h)$ for each $a\in A$ and consider $h\rightarrow 0$, similarly for $\mathcal{L}_{Y_B}(t+1)$. To get non-trivial limits we also have to divide by $h$.

\begin{definition}\label{Def: local contemporaneously uncorrelated}
Suppose  $Y_v=(Y_v(t))_{t\in\R}$ is $j_v$-times mean-square differentiable  but the $(j_v+1)$-derivative does not exist for  $v\in V$. %with $j_v$-derivative denoted by $D^{(j_v)}Y_v$ for all $v\in V$, where  for $j_v=0$ we define $D^{(0)} Y_v=Y_v$.
Let $A,B \subseteq S \subseteq V$ and $A\cap B=\emptyset$. Then $\CY_A$ and $\CY_B$ are \textsl{locally contemporaneously uncorrelated} with respect to $\CY_S$ if, for all $t\in \R$, $a\in A$, $b\in B$,
\begin{align*}
\limhh \frac{1}{h} \: \BE & \left[\left( D^{(j_a)} Y_a(t+h)- P_{\mathcal{L}_{Y_{S}}(t)} D^{(j_a)} Y_a(t+h) \right) \right. \\
		  &\times \left. \overline{\left( D^{(j_b)} Y_b(t+h)- P_{\mathcal{L}_{Y_{S}}(t)} D^{(j_b)} Y_b(t+h) \right)} \right]=0.
\end{align*}
We write $\CY_A \nsimnull \CY_B \: \vert \: \CY_S$.
\end{definition}

\begin{remark} \label{eq:correspondenceXAXa local}
$\mbox{}$
\begin{itemize}
\item[(a)] Due to the definition, we receive directly
\begin{align*}%\label{eq:correspondenceXAXa local}
    \CY_A \nsimnull \CY_B \: \vert \: \CY_S
    \quad \Leftrightarrow \quad
    \CY_a \nsimnull \CY_b \: \vert \: \CY_S
    \quad
   \forall\; a\in A,\, b\in B,
\end{align*}
which is useful for verifying local contemporaneous uncorrelation.
    \item[(b)] \Cref{Def: local contemporaneously uncorrelated} is similar to the characterisation of local contemporaneous uncorrelation for semimartingales in \cite{CO96}, Proposition 3, using linear predictions instead of conditional expectations and $\sigma$-fields instead of linear spaces. But \cite{CO96} assume additionally that the martingale part of the semimartingale is continuous, excluding Lévy-It$\hat{\mbox{o}}$ processes that are not Brownian motion driven, such as Lévy-driven Ornstein-Uhlenbeck processes.
    \item[(c)] To give an equivalent characterisation as an equality of projections, restrictions on the linear derivative spaces are necessary. Thus,  we do not include these characterisations here. Sufficient, however, is in any case that for all $t\in \R$, $a\in A$, $b\in B$,
    \begin{align*}
        &\limh P_{\mathcal{L}_{Y_{S}}(t)} \left( \frac{D^{(j_b)} Y_b(t+h)- D^{(j_b)} Y_b(t)}{\sqrt{h} } \right) \\
        &\quad = \limh P_{\mathcal{L}_{Y_{S}}(t) \vee \mathcal{L}_a(t,t+h)} \left( \frac{D^{(j_b)} Y_b(t+h)- D^{(j_b)} Y_b(t)}{\sqrt{h}} \right)  \quad \mathbb{P}\text{-a.s}.
    \end{align*}
\end{itemize}
\end{remark}

Finally, we introduce a global concept of contemporaneous correlation, in analogy to global Granger causality, to discuss short-run vs. long-run effects.

\begin{definition}
Let $A,B \subseteq S \subseteq V$ and $A\cap B=\emptyset$.
%Then $\CY_A$ and $\CY_B$ are \textsl{{(linear) contemporaneously uncorrelated} with respect to $\CY_S$ {up to horizon $h$}}, $h \in \R$, if and only if
%\begin{align*}
%\mathcal{L}_{Y_A}(t,t+h) \perp \mathcal{L}_{Y_B}(t,t+h) \: \vert \: \mathcal{L}_{Y_{S}}(t),
%\end{align*}
%for $t\in \R$. We write $\CY_A \stackrel{h}{\nsim} \CY_B \: \vert \: \CY_S$.
Then $\CY_A$ and $\CY_B$ are \textsl{globally contemporaneously uncorrelated} with respect to $\CY_S$ if, for $h\geq 0$ and $t\in \R$,
\begin{align*}
\mathcal{L}_{Y_A}(t,t+h) \perp \mathcal{L}_{Y_B}(t,t+h) \: \vert \: \mathcal{L}_{Y_{S}}(t).
\end{align*}
 We write $\CY_A \nsiminf \CY_B \: \vert \: \CY_S$.
\end{definition}

%\marginpar{Globale Definition ggf. mit Raum bis unendlich abändern}

\begin{remark}
Again, projections can be used to characterise the global contemporaneous uncorrelation.
%$\CY_A$ and $\CY_B$ are contemporaneously uncorrelated with respect to $\CY_S$ up to horizon $h$, $h \in \R$, if and only if
%\begin{align*}
%P_{\mathcal{L}_{Y_{S}}(t)\vee \mathcal{L}_{Y_{B}}(t,t+h)} Y_a(t+h')  = P_{\mathcal{L}_{Y_{S}}(t)} %Y_a(t+h')
%\end{align*}
%$\mathbb{P}$-a.s for all $a\in A$, $0\leq h' \leq h$ and $t \in \R$.
%$\CY_A$ and $\CY_B$ are global contemporaneously uncorrelated with respect to %$\CY_S$.
%Let $A,B \subseteq S \subseteq V$, $A\cap B=\emptyset$.
Precisely, $\CY_A$ and $\CY_B$ are {globally contemporaneously uncorrelated} with respect to $\CY_S$, if and only if,
%\begin{align*}
%\mathcal{L}_{Y_A}(t,t+h) \perp \mathcal{L}_{Y_B}(t,t+h) \: \vert \: \mathcal{L}_{Y_{S}}(t),
%\end{align*}
%for $t\in \R$. We write $\CY_A \stackrel{h}{\nsim} \CY_B \: \vert \: \CY_S$.
%Then $\CY_A$ and $\CY_B$ are \textsl{global contemporaneously uncorrelated} with respect to $\CY_S$, if and only if,
for all $b\in B$, $0\leq h' \leq h$, $h\geq 0$, and $t \in \R$
\begin{align*}
P_{\mathcal{L}_{Y_{S}}(t)\vee \mathcal{L}_{Y_{A}}(t,t+h)} Y_b(t+h')  = P_{\mathcal{L}_{Y_{S}}(t)} Y_b(t+h') \quad \mathbb{P}\text{-a.s.}
\end{align*}
The proof is similar to the proof of \Cref{Charakterisierung als Gleichheit der linearen Vorhersage 2} %only based at \Cref{Charakterisisierung contemporaneously uncorrelated}
and is therefore not included in the paper. Also, the analogue statements to \Cref{eq:correspondenceXAXa} hold.
\end{remark}

\begin{comment}
Note that the definition of contemporaneous uncorrelation in \Cref{(linear) global contemporaneously uncorrelated}
states that for any $\HY^A \in \mathcal{L}_{Y_A}(t,t+1)$ and $\HY^B \in \mathcal{L}_{Y_B}(t,t+1)$
\begin{align*}
\BE \left[\left( \HY^A- P_{\mathcal{L}_{Y_{S}}(t)} \HY^A \right)
		  \overline{\left( \HY^B- P_{\mathcal{L}_{Y_{S}}(t)} \HY^B \right)} \right]=0.
\end{align*}
\end{comment}
%%%%%%%%%%%%%%%%%%%%%%%%%%%%%%%%%%%%%%%%%%%%%%%%%%%%%%%%%%%%%%%%%%%%%%%%%%%%%%%%%%
%\subsection{Relations of the different concepts of contemporaneous uncorrelation}
It is obvious that, by definition and due to \Cref{Remark 3.6} and \eqref{MR}, the following relations between the three definitions of contemporaneous uncorrelation are valid.

\begin{lemma}\label{lemma:Beziehungen cont uncor}
Let $A,B \subseteq S \subseteq V$ and $A\cap B=\emptyset$. %, $h \in \R$.
Then the following implications hold:
\begin{itemize}
\item[(a)] \makebox[2,8cm][l]{$\CY_A \nsiminf \CY_B \: \vert \: \CY_S$} $\Rightarrow \quad \CY_A  \nsim \CY_B \: \vert \: \CY_S$.
%\text{and} \quad \CY_A  \stackrel{h}{\nsim} \CY_B \: \vert \: \CY_S,
\item[(b)] \makebox[2,78cm][l]{$\CY_A  \nsim \CY_B \: \vert \: \CY_S$} $ \Rightarrow \quad  \CY_A  \nsimnull \CY_B \: \vert \: \CY_S.$
\end{itemize}
\end{lemma}

%\VF{Was ist mit der umgekehrten Richtung in (a)?}
%\VF{fraktionale CARMA Prozesse in Erwägung ziehen, wo nicht aller Äquivalenzen vermutlich funktionieren, oder was für ein Gegenbeispiel haben Dufour und Renault?}

The similarities and differences between the various definitions again become apparent when looking at examples. In particular, we derive that the opposite direction in \Cref{lemma:Beziehungen cont uncor} (b) does not hold in general.

\begin{example}
Suppose $\CY_V$  is the Ornstein-Uhlenbeck process as defined in \Cref{Example:Ornstein-Uhlenbeck process} with $\BA$ and $\Sigma_L$ as in \eqref{MM}. Then we derive in \Cref{sec:CGMCAR} that
\begin{align*}
\CY_a \nsiminf \CY_b \: \vert \: \CY_V
\quad  \Leftrightarrow \quad
\CY_a \nsim \CY_b \: \vert \: \CY_V
\quad & \Leftrightarrow \quad
[\BA^{\alpha} \BS_L (\BA^\top )^\beta]_{ab}=0, \quad \alpha, \beta=0,\ldots,k-1, \\
%\intertext{and}
\CY_a \nsimnull \CY_b \: \vert \: \CY_V
\quad & \Leftrightarrow \quad
[\BS_L]_{ab}=0.
%\CY_a \stackrel{h}{\ninst} \CY_b \: \vert \: \CY_V
%\quad & \Leftrightarrow \quad
%[A^{\alpha} \BS_L (A^\top )^\beta]_{ab}=0, \quad \alpha, \beta=0,\ldots,k-1,
%\\
\end{align*}
Of course, we obtain
\begin{align*}
\CY_a \nsim \CY_b \: \vert \: \CY_V
\quad  &\Rightarrow \quad
[\BA^{\alpha} \BS_L (\BA^\top )^\beta]_{ab}=0, \quad \alpha, \beta=0,\ldots,k-1, 
\quad  \Rightarrow \quad
[\BS_L]_{ab}=0 \\ \quad  &\Rightarrow \quad
\CY_a \nsimnull \CY_b \: \vert \: \CY_V,
\end{align*}
but the opposite direction does not generally hold, in turn, an exception is the case where $\BA$ is a diagonal matrix. A specific counterexample is again the Ornstein-Uhlenbeck process from \Cref{Example:Ornstein-Uhlenbeck process}, which we see in \Cref{fig: Two graphical models b}. Here, $\CY_1 \nsimnull \CY_2 \: \vert \: \CY_{\{1,2,3\}}$ but $\CY_1 \sim \CY_2 \: \vert \: \CY_{\{1,2,3\}}$.
\end{example}

%%%%%%%%%%%%%%%%%%%%%%%%%%%%%%%%%%%%%%%%%%%%%%%%%%%%%%%%%%%%%%%%%%%%%%%%%%%%%%%%%%

\section{Orthogonality graphs for stationary  continuous-time processes}\label{sec:path_diagrams}
In this section, we introduce graphical models for stationary, mean-square continuous processes $\CY_V=(Y_V(t))_{t\in \R}$. These graphical models visualise directed as well as undirected relations between the different component series $\CY_v=(Y_v(t))_{t\in \R}$, $v=1,\ldots,k$. The vertices represent the different component series $\CY_v$, $v=1,\ldots,k$, of the process. Furthermore, they are connected by directed and undirected edges, which represent certain directional and non-directional influences between them. The arising graphical models are then called (mixed) orthogonality graphs.

%%%%%%%%%%%%%%%%%%%%%%%%%%%%%%%%%%%%%%%%%%%%%%%%%%%%%%%%%%%%%%%%%%%%%%%%%%%%%%%%%%
\subsection{Separability and conditional linear separation}\label{subsec:condorth}

For the definition of the graphical models, we first ensure that the conditional orthogonality relation satisfies the property of intersection (C5) in \Cref{properties of conditional orthogonality} for suitable linear subspaces and second, we show that the missing relations in \eqref{eq3b} and \eqref{eq3.4} hold. Therefore, we investigate separability and conditional linear separation of linear spaces. %First of all, note that since $\CY_V$ is stationary and continuous in mean square, all linear spaces used in this paper are separable.
The proofs of the lemmata of this subsection are the subject of the Supplementary Material \ref{suppl:proofs_section_5}, and the proofs of the propositions and theorems are content of \Cref{Sec:Proofs:conditional orthogonality}.

\begin{lemma}\label{Separabilität}
%Let $\CY_V=(Y_V(t))_{t\in \R}$ be a $k$-dimensional process that satisfies \Cref{Assumption 1 und 2}.
Let $A \subseteq V$ and $s,t\in \R$ with $s<t$. Then $\mathcal{L}_{Y_A}$, $\mathcal{L}_{Y_A}(t)$ and $\mathcal{L}_{Y_A}(s,t)$ are separable.
\end{lemma}
 Furthermore, we require that $\mathcal{L}_{Y_A}(t)$ and $\mathcal{L}_{Y_B}(t)$ are conditionally linearly separated by $\mathcal{L}_{Y_C}(t)$ if $t\in \R$ and $A,B,C \subseteq V$ are disjoint.
 This assumption is a lot more intricate because it is a very abstract definition and difficult to verify.
\begin{remark}
Unlike us, \cite{EI10} uses conditional independence instead of conditional orthogonality. For the associated intersection property (C5) %, going back to \cite{PE94},
\textsl{measurable conditional separation} is required, corresponding to our \textsl{conditional linear separation} assumption. There,  measurable conditional separation is also generally not valid, and sufficient assumptions are given.
\end{remark}

To better understand conditional linear separation, we introduce a sufficient criterion.

\begin{lemma} \label{Lemma lineare Separabilität}
Let $t\in \R$. Suppose that for all $A,B \subseteq V$ with $A\cap B=\emptyset$ we have
\begin{align*}
\mathcal{L}_{Y_{A}}(t) \cap \mathcal{L}_{Y_{B }}(t) =  \{0\} \quad \text{and} \quad
\mathcal{L}_{Y_{A}}(t) + \mathcal{L}_{Y_{B}}(t) = \mathcal{L}_{Y_{A}}(t) \vee \mathcal{L}_{Y_{B}}(t)\quad \mathbb{P}\text{-a.s.}
\end{align*}
 Then, for all disjoint subsets $A,B,C \subseteq V$, we get % for all $t\in \R$ and disjoint subsets $A,B,C \subseteq V$,
\begin{align*}
\mathcal{L}_{Y_{A\cup C}}(t) \cap \mathcal{L}_{Y_{B \cup C}}(t) =  \mathcal{L}_{Y_C}(t) \quad \mathbb{P}\text{-a.s.}
\end{align*}
\end{lemma}
The first assumption is the linear independence of the two linear spaces, the second assumption is the closedness of the sum. %The detailed proof of the assertion is given in\Cref{Sec:Proofs:conditional orthogonality}.
It makes little sense to formulate these two properties as assumptions on $\CY_V$, as they are still too abstract and difficult to verify. Therefore, we provide an easy-to-use criterion. %The problem of the closedness of the sum of two subspaces is the subject of numerous publications, see, e.g., %\cite{KO40},

\begin{assumption}\label{Assumption an Dichte}
Suppose $\CY_V$ has a %positive definite
spectral density matrix $f_{Y_V Y_V}(\cdot)>0$ and that there exists an $0<\varepsilon<1$, such that
\begin{align*}
\dAB(\lambda) \coloneqq f_{Y_AY_A}(\lambda)^{-1/2}f_{Y_AY_B}(\lambda)f_{Y_BY_B}(\lambda)^{-1}f_{Y_BY_A}(\lambda)f_{Y_AY_A}(\lambda)^{-1/2} \leq_L
(1-\varepsilon)I_{\alpha},
\end{align*}
for almost all $\lambda \in \R$ and for all disjoint subsets $A,B\subseteq V$ with $\# A=\alpha$.
\end{assumption}

For $A=\{a\}$ the function $\dAB(\lambda)$, $\lambda \in \R$, is called multiple coherence; we refer to \cite{Priestley} and \cite{BR01} for further reading.
\Cref{Assumption an Dichte} is satisfied, e.g., for stationary causal MCAR processes and in particular Ornstein-Uhlenbeck processes, for details see \Cref{sec:CGMCAR},  and for the more general family of state space models see \cite{VF23preb}. In our opinion, even fractional MCAR processes satisfy this assumption. Furthermore, the assumption is indeed sufficient for conditional linear separability.

\begin{proposition} \label{Eigenschaften der linearen Räume}
Let $\CY_V$ satisfy \Cref{Assumption an Dichte}. Then for all $t\in \R$ and disjoint subsets $A,B,C \subseteq V$, we have
 \begin{align*}
\mathcal{L}_{Y_{A}}(t) \cap \mathcal{L}_{Y_{B }}(t) & =  \{0\}, \quad
\mathcal{L}_{Y_{A}}(t) + \mathcal{L}_{Y_{B}}(t) = \mathcal{L}_{Y_{A}}(t) \vee \mathcal{L}_{Y_{B}}(t),%, \quad \mathbb{P}\text{-a.s.}
\quad \text{ and } \\
& \mathcal{L}_{Y_{A\cup C}}(t) \cap \mathcal{L}_{Y_{B \cup C}}(t) =  \mathcal{L}_{Y_C}(t) \quad \mathbb{P}\text{-a.s.}
\end{align*}
\end{proposition}

%The proof of \Cref{Eigenschaften der linearen Räume} is given in \Cref{Sec:Proofs:conditional orthogonality}.% shows that this assumption is satisfied based on spectral representations of elements in the linear spaces.
%At this point we just make some comments on \Cref{Assumption an Dichte}.

 Recall that in \Cref{Charakterisierung local Granger non-causal} we already assume the closedness of the sum, and now \Cref{Eigenschaften der linearen Räume} gives a sufficient criterion for this property.

%\VF{To recall, the conditional linear separation that we have already assumed in \Cref{Charakterisierung local Granger non-causal} and now \Cref{Eigenschaften der linearen Räume} gives a sufficient criterion for this theorem.}

\begin{remark}
First of all, $\dAB(\lambda) \leq_{L} I_{\alpha \times \alpha}$ holds even without \Cref{Assumption an Dichte}. Indeed, suppose $\Phi_B(\cdot)$ is the random spectral measure from the spectral representation  of $\CY_B$ in \eqref{spectral representation of stationary process}, then  the spectral density matrix of
\begin{align*}
\varepsilon_{A \vert B}(t) = Y_A(t) - \uint e^{i\lambda t} f_{Y_AY_B}(\lambda)f_{Y_BY_B}(\lambda)^{-1} \Phi_B(d\lambda)
\end{align*}
is
\begin{align*}
f_{\varepsilon_{A\vert B}\varepsilon_{A\vert B}}(\lambda) =  f_{Y_AY_A}(\lambda) - f_{Y_AY_B}(\lambda) f_{Y_BY_B}(\lambda)^{-1}f_{Y_BY_A}(\lambda),
\end{align*}
and it is non-negative definite according to \cite{BR91}, p.~436. Furthermore, \Cref{Assumption an Dichte} especially forbids some purely linear relationships between the components, which can be seen as follows. Assume that $\dAB(\lambda) =I_{\alpha}$ for almost all $\lambda \in \R$. Then $f_{\varepsilon_{A\vert B}\varepsilon_{A\vert B}}(\lambda) = 0_{\alpha}$ for almost all $t\in \R$ and thus, $c_{\varepsilon_{A\vert B}\varepsilon_{A\vert B}}(t)= 0_{\alpha}$ for all $t\in \R$. Therefore, $\varepsilon_{A \vert B}(t) = 0_{\alpha}$
$\mathbb{P}$-a.s. and $Y_A(t)$ is already a linear transformation of $Y_B(t)$. Somewhat loosely, one could say that \Cref{Assumption an Dichte} not only forbids a purely linear relationship between $\CY_A$ and $\CY_B$ but already requires some kind of distance between the subprocesses due to the uniform boundedness. This also fits with  \cite{BR01}, eq.~(8.3.10), who calls the matrix function $\dAB(\lambda)$ in discrete-time a measure of the linear association of $\CY_A$ and $\CY_B$ at frequency $\lambda$.
\end{remark}

\begin{remark}
Let us compare \Cref{Assumption an Dichte} with \cite{EI07}, equation (2.1), who proposes a comparable assumption on the spectral density matrix in discrete time, also with the aim that the property of intersection (C5) is valid. \cite{EI07} demands the existence of a constant $c>1$, such that the spectral density matrix satisfies
\begin{align}\label{eq: Eichlers Dichteannahme}
\frac{1}{c}I_k \leq_L f_{Y_VY_V}(\lambda) \leq_L c I_k,
\end{align}
for all $\lambda \in [-\pi, \pi]$.
If this assumption is fulfilled, some matrix algebra calculations as in the proof of \Cref{Schranke Intervall} give that for any disjoint subsets $A, B \subseteq V$,
\begin{align*}
f_{Y_AY_A}(\lambda) - f_{Y_AY_B}(\lambda)f_{Y_BY_B}(\lambda)^{-1}f_{Y_BY_A}(\lambda) \geq_L \frac{1}{c}I_{\alpha} \geq_L \frac{1}{c^2} f_{Y_AY_A}(\lambda).
\end{align*}
Thus, on the interval $[-\pi, \pi]$ \Cref{Assumption an Dichte} is satisfied with $\varepsilon= 1/c^2$. However, \cite{EI07}'s assumption is stricter than ours since one must be able to place a diagonal matrix between $1/c^2 f_{Y_AY_A}(\lambda)$ and $f_{Y_AY_A}(\lambda) - f_{Y_AY_B}(\lambda)f_{Y_BY_B}(\lambda)^{-1}f_{Y_BY_A}(\lambda)$. We further point out that we cannot generalise \cite{EI07}'s assumption directly to continuous-time processes by assuming \eqref{eq: Eichlers Dichteannahme} for almost all $\lambda \in \R$. This requirement is too strict and, e.g., not satisfied for Ornstein-Uhlenbeck processes.
\end{remark}

\begin{comment}
\begin{remark}
From the inclusion property
\begin{align*}
\mathcal{L}_{Y_A}(- \infty) \subseteq \mathcal{L}_{Y_A}(t) \subseteq \mathcal{L}_{Y_A}, \quad t \in \R,
\end{align*}
it happens that two extreme cases may occur.
\begin{itemize}
    \item[(i)] If  $\mathcal{L}_{Y_A}(- \infty)=\mathcal{L}_{Y_A}$, the process is called \textit{deterministic}, since the remote past already contains all available information concerning the process. If the remote past is known, the future development of the process can be predicted precisely \cite[III, eq.~(2.1)]{RO67}. %that is
%\begin{align*}
%P_{\mathcal{L}_{Y_V}(t)} Y_v(t+h) = Y_v(t+h)
%\end{align*}
%$\mathbb{P}$-a.s. for all $v\in V$, $t\in \R$ and $h\geq 0$ (\cite{RO67}, III, (2.1)).
    \item[(ii)] If $\mathcal{L}_{Y_A}(- \infty)=\{ 0\}$ the process is called \textit{purely non-deterministic}. %The remote past is of no use for prediction purposes.
Loosely speaking,  any information of the process, must have entered as a new impulse at some instant in the past (\cite{CR61}, p.~251, \cite{CR71}, p.~7). From the viewpoint of  prediction, purely non-deterministic  is  equivalent to
$\underset{h \rightarrow \infty}{\text{l.i.m.\:}} P_{\mathcal{L}_{Y_V}(t)} Y_v(t+h) = 0$ $\mathbb{P}\text{-a.s.} $ for all $v\in V$ and $t\in \R$
  (\cite{RO67}, III, Theorem 2.1.).
  \end{itemize}
\end{remark}
\end{comment}

\Cref{Assumption an Dichte} now ensures, as desired, that the conditional orthogonality relation satisfies the property of intersection (C5) in \Cref{properties of conditional orthogonality} for suitable linear subspaces. \Cref{Assumption an Dichte} further provides us with the missing relations of the causality concepts in \eqref{eq3b} and \eqref{eq3.4}. %the proof  is moved to \Cref{Sec:Proofs:conditional orthogonality}.

%\LS{Mit der Differenzdefinition der lokalen Kausalität kann man Beweis in (b) nicht mehr führen}

\begin{proposition} \label{Proposition 5.6}
Let $\CY_V$ satisfy \Cref{Assumption an Dichte}. Let $A,B \subseteq S \subseteq V$ and $A\cap B=\emptyset$. Then %the following relations hold:
\begin{itemize}
    \item[(a)] \makebox[2,78cm][l]{$\CY_{A} \nrarrow \CY_B\: \vert \: \CY_S$}
    \makebox[3,5cm][l]{$\Leftrightarrow \quad \CY_{a} \nrarrow \CY_b\: \vert \: \CY_S$} $ \forall \: a\in A, \: b\in B$.
    \item[(b)] \makebox[2,76cm][l]{$\CY_{A} \nrarrownull \CY_B\: \vert \: \CY_S$}
    \makebox[3,5cm][l]{$\Leftrightarrow \quad
    \CY_{a} \nrarrownull \CY_b\: \vert \: \CY_S$} $ \forall \: a\in A, \: b\in B$.
    \item[(c)] \makebox[2,8cm][l]{$\CY_{A} \nrarrowinf \CY_B\: \vert \: \CY_S$}
    \makebox[3,5cm][l]{$ \Leftrightarrow \quad
    \CY_{a} \nrarrowinf \CY_b\: \vert \: \CY_S$} $\forall \: a\in A, \: b\in B$.
\end{itemize}
\end{proposition}

However, for the proof of the global Andersson, Madigan and Perlman (AMP) Markov property in our orthogonality graph, we require further assumptions. Any process that is wide sense stationary can be uniquely decomposed in a deterministic and a purely non-deterministic process that are mutually orthogonal (\citealp{GL58}, Theorem 1). From the point of view of applications, deterministic processes are not important. % (\cite{CR71}, p.~7).
Therefore, we assume that the given process is purely non-deterministic.

\begin{assumption}\label{Assumption purely nondeterministic of full rank}
Let $\CY_V$ be purely non-deterministic, that is $\mathcal{L}_{Y_V}(- \infty) = \{0 \}$ $\mathbb{P}$-a.s.
\end{assumption}

Necessary and sufficient conditions for processes being purely non-deterministic can be found, e.g., in \cite{GL58}, Theorem 3, \cite{RO67}, III, Theorem 2.4, \cite{MA61}, Theorem 1. Typical examples are MCAR processes and the more general class of state space models whose driving Lévy process has expectation zero.

Finally, we can deduce the following property from Assumptions \ref{Assumption an Dichte} and \ref{Assumption purely nondeterministic of full rank}, which we require for the proof of the global AMP Markov property. The property further stands in analogy to assumption (M) on $\sigma$-fields in \cite{EI10} and equation (2.4) in \cite{EI01}. Note that these assumptions are stronger than our Assumptions \ref{Assumption an Dichte} and \ref{Assumption purely nondeterministic of full rank} and quite difficult to verify.
 %of one of our orthogonality graphs. %The proof of this lemma is given in \Cref{Sec:Proofs:conditional orthogonality}.
%The proof is based on the same idea as the proof of conditional linear separability.

\begin{lemma}\label{Eigenschaft für Hilfslemma B3}
Let $\CY_V$ %be a $k$-dimensional process that satisfies
satisfy Assumptions \ref{Assumption an Dichte} and \ref{Assumption purely nondeterministic of full rank}.
Let $A \subseteq V$ and $t\in\R$.  Then
\begin{align} \label{eqasd}
\bigcap_{k \in \N} \left( \mathcal{L}_{Y_A}(t-k) \vee \mathcal{L}_{Y_{V\setminus A}}(t) \right) =\mathcal{L}_{Y_{V\setminus A}}(t) \quad \mathbb{P}\text{-a.s.}
\end{align}
\end{lemma}

Note that Assumptions \ref{Assumption an Dichte} and \ref{Assumption purely nondeterministic of full rank} are not necessary assumptions for the following Markov properties to hold. Sufficient and weaker assumptions are the conditional linear separation and \eqref{eqasd}, both are satisfied under Assumptions \ref{Assumption an Dichte} and \ref{Assumption purely nondeterministic of full rank}.

%%%%%%%%%%%%%%%%%%%%%%%%%%%%%%%%%%%%%%%%%%%%%%%%%%%%%%%%%%%%%%%%%%%%%%%%%%%%%%%%%%%%%%%%%%%%%%%%%%
\subsection{Introduction to (local) orthogonality graphs}\label{subsect: definition of the orthogonality graph}
Let us now visualise suitable concepts of directed and undirected influences in graphical models.  In principle, it is possible to define a graph with any of the three definitions of Granger causality and contemporaneous correlation.
However, our goal is to define a graph with concepts that are as strong as necessary, but as weak as possible, so that the usual Markov properties for mixed graphs hold. For MCAR processes, global Granger causality and Granger causality as well as global contemporaneous uncorrelation and contemporaneous uncorrelation coincide (see \Cref{sec:CGMCAR}) and therefore we do not discuss a global graph.

\begin{definition}\label{Definition orthogonality graph}
Let $\CY_V$ %be a $k$-dimensional process that satisfies
satisfy Assumptions \ref{Assumption an Dichte} and \ref{Assumption purely nondeterministic of full rank}.
\begin{itemize}
    \item[(a)]
If we define $V=\{1,\ldots,k\}$ as the vertices and the edges $E_{OG}$ via
\begin{itemize}
\item[(i)\phantom{i}] \makebox[2,5cm][l]{$a \rarrow b \notin E_{OG}$}  $\Leftrightarrow \quad \CY_a \nrarrow \CY_b \: \vert \: \CY_V$,
\item[(ii)]  \makebox[2,5cm][l]{$a \inst b \notin E_{OG}$} $\Leftrightarrow \quad \CY_a \nsim \CY_b \: \vert \: \CY_V$,
\end{itemize}
for $a,b\in V$, $a\neq b$, then $G_{OG}=(V,E_{OG})$ is called \textsl{(mixed) orthogonality graph} for $\CY_V$. \\ $\mbox{}$ \vspace*{-0.2cm}
\item[(b)] If we define $V=\{1,\ldots,k\}$ as the vertices and the edges $E_{OG}^{0}$ via
\begin{itemize}
\item[(i)\phantom{i}] \makebox[2,5cm][l]{$a \rarrow b \notin E_{OG}^{0}$} $ \Leftrightarrow \quad \CY_a \nrarrownull \CY_b \: \vert \: \CY_V$,
\item[(ii)]  \makebox[2,5cm][l]{$a \inst b \notin E_{OG}^{0}$} $ \Leftrightarrow \quad \CY_a \nsimnull \CY_b \: \vert \: \CY_V$,
\end{itemize}
for $a,b\in V$, $a\neq b$, then $G_{OG}^{0}=(V,E_{OG}^{0})$ is called \textsl{local (mixed) orthogonality graph} for $\CY_V$.
\end{itemize}
\end{definition}

In words, in both graphs each vertex $v\in V$ represents one component series $\CY_v$. Two vertices $a$ and $b$ are joined by a directed edge $a \rarrow b$ whenever $\CY_a$ is (local) Granger causal for $\CY_b$ and by an undirected edge $a \inst b$ whenever $\CY_a$ and $\CY_b$ are (locally) contemporaneously correlated given $\CY_V$. We %want to
make some remarks on those graphical models.

\begin{remark} $\mbox{}$
\begin{itemize}
    \item[(a)] The motivation for the name (local) orthogonality graph arises from the fact that both the directed and undirected edges are defined by specific (local) conditional orthogonality relations. For a concise notation, we omit the word conditional. Furthermore, the name (local) orthogonality graph is also analogous to the local independence graph (\citealp{DI06, DI07, DI08, Mogensen:Hansen:2020, Mogensen:Hansen:2022}).
    %The motivation for the name orthogonality graph comes, of course, from the  Granger causality relations in the directed edges, and also from the fact that the contemporaneous correlations used for the undirected edges are sometimes called local causalities. 
    The graphical models are further named \textsl{mixed} orthogonality graphs because they  contain two types of edges. Since we do not usually consider purely directed or undirected graphs, we omit the prefix mixed for ease of notation. Note that the orientation of the directed edge makes a difference and multiple edges of the same type and orientation are not allowed. Thus, two vertices $a$ and $b$ can be connected by up to three edges, namely $a \rarrow b$, $a \larrow b$ and $a \inst b$, as can also be seen in \Cref{fig: Two graphical models b}.
   %\VF{ Analogue to \cite{EI07}, we use a dashed line $\inst$, to connect contemporaneously correlated processes, since undirected edges $\edge$ are usually associated with non-zero entries in the inverse of the covariance matrix.}\marginpar{würde ich streichen}
    \item[(b)] The Assumptions \ref{Assumption an Dichte} and \ref{Assumption purely nondeterministic of full rank} as well as the stationarity and the mean square continuity are not necessary for the definition of the graphs, but they are essential for the usual Markov properties to hold.  Wide sense stationarity is a basic requirement, otherwise, e.g., \Cref{Assumption an Dichte} is not well-defined, which is a sufficient assumption for conditional linear separation.  %, as we have already discussed in the definition of the linear spaces.
    The mean square continuity and \Cref{Assumption an Dichte} will already be used for the first time in the proof of the local Markov property.
    %Man verwendet \Cref{Proposition 5.6} und darin conditional linear separability.
    %when we apply the property of intersection (C5) from \Cref{properties of conditional orthogonality}.
    \Cref{Assumption purely nondeterministic of full rank} is only required in the proof of the global AMP Markov property. Since we show global Markov properties for the local orthogonality graph only in special cases, \Cref{Assumption purely nondeterministic of full rank} is not necessary there.
\item[(c)] We already know that $a \rarrow b \notin E_{OG}$
directly implies $a \rarrow b \notin E_{OG}^0$
and similarly $a \inst b \notin E_{OG}$ also gives  $a \inst b \notin E_{OG}^0$. In summary, $E_{OG}^0\subseteq E_{OG}$, the graph defined by the local versions of Granger causality and contemporaneous correlation has fewer edges than the graph  $G_{OG}$ based on the classical Granger causality and contemporaneous correlation, and in general the graphs are not equal. {Again, this can be seen in \Cref{fig: Two graphical models b}.} The advantage of the graph $G_{OG}^0$ based on the local version is that it allows to model more general graphs than $G_{OG}$. 
    \item[(d)] In \Cref{Definition orthogonality graph}, we have defined the orthogonality graph and the local orthogonality graph. Of course, it is also possible to define the \textsl{global orthogonality graph} based on global Granger causality and global contemporaneous correlation, but this is not part of this work. There are various reasons for this.
    On the one hand, the sparsity structure of the global orthogonality graph is very weak. The global orthogonality graph has even more edges than the orthogonality graph and the local orthogonality graph. Moreover, the orthogonality graph already satisfies the global AMP Markov and the global Markov property, as we are going to derive later in \Cref{sec:global}. These Markov properties can easily be transferred to the global orthogonality graph, the proofs are even easier.   
    On the other hand, in specific models such as MCAR processes and state space models, Granger causality corresponds to global Granger causality, and contemporaneous correlation corresponds to contemporaneous correlation (cf.~\Cref{Remark:step size}), so that the global orthogonality graph is equal to the orthogonality graph and does not give any additional information.
    %However, whereas $G_{OG}$ satisfies the global AMP and the global Markov property it is not possible to show these for $G_{OG}^0$.
\end{itemize}
\end{remark}
%\LS{Erkennt man an den neuen Beweisen ohne (C5) noch, warum man die Separabilität braucht? Eine kurze Suche zeigt, dass wir C5 gar nicht mehr verwenden.}
%\VF{Aber Proposition 5.3 und die conditionally linearly separation geht einige mal ein. C5 direkt muss auch nicht sein (aber natürlich indirekt).}
%\LS{Hatte ich übersehen. Die continuity in mean sqare sollte man aber noch anders formulierne, da man die Separabilität von linearen Räumen nirgends mehr so explizit sieht wie es bei C5 der Fall war.}

%%%%%%%%%%%%%%%%%%%%%%%%%%%%%%%%%%%%%%%%%%%%%%%%%%%%%%%%%%%%%%%%%%%%%%%%%%%%%%%%%%%%%%%%%%%%%%%%%%
\subsection{Markov properties of (local) orthogonality graphs}\label{sec:Markprop}
The (local) orthogonality graph decodes directed and undirected relations between component series of the process $\CY_V$. Conversely, a mixed graph can be associated with a set of constraints imposed on the stochastic process $\CY_V$. Such a set of causal relations encoded by a graph is commonly known as a Markov property of the graph (cf.~\citealp{LA04, WI08}). In this section, we introduce various levels of Markov properties. %Since, in the context of time series, graphs can encode different types of relationships between the component series, we therefore speak of \textit{causal} Markov properties when talking about Granger non-\textit{causal}ity (cf.~\cite{EI10}, p.~236). 
We start with the pairwise, local and block-recursive Markov properties. We then move on to two global Markov properties, namely the global AMP Markov property and the global Markov property.

%%%%%%%%%%%%%%%%%%%%%%%%%%%%%%%%%%%%%%%%%%%%%%%%%%%%%%%%%%%%%%%%%%%%%%%%%%%%%%%%%%%%%%%%%%%%%%%%%%
\subsubsection{Pairwise, local and block-recursive Markov property}\label{subsec:pwlcbc}

Let us start with a few simple Markov properties that we expect from a graph. First of all, the \mbox{(local)}  orthogonality graph visualises pairwise relationships between the components of a process $\CY_V$ by definition, that is the pairwise Markov property.

\begin{proposition}\label{pairwise Markov property} $\mbox{}$
\begin{itemize}
    \item[(a)] Let $G_{OG}=(V,E_{OG})$ be the orthogonality graph for $\CY_V$. Then $\CY_V$ satisfies the \textsl{pairwise Markov property} with respect to $G_{OG}$, i.e., for all $a,b\in V$, $a\neq b$:
\begin{itemize}
\item[(i)\phantom{i}] \makebox[2,5cm][l]{$a \rarrow b \notin E_{OG}$}  $\Rightarrow \quad \CY_a \nrarrow \CY_b \: \vert \: \CY_V$,
\item[(ii)] \makebox[2,5cm][l]{$a \inst b \notin E_{OG}$}   $\Rightarrow \quad \CY_a \nsim \CY_b\: \vert \: \CY_V$.
 \end{itemize}  $\mbox{}$ \vspace*{-0.2cm}
    \item[(b)] Let $G^{0}_{OG}=(V,E_{OG}^{0})$ be the local orthogonality graph for $\CY_V$. Then $\CY_V$ satisfies the \textsl{pairwise Markov property} with respect to $G_{OG}^{0}$.
\end{itemize}
\end{proposition}

Further, define  %As in \cite{EI09}, p.~6f, let
$\pa(a)=\{ v\in V \: \vert \: v \rarrow a \in E\}$
and $
\ne(a)=\{ v\in V \: \vert \: v \inst a \in E\}$
as the set of parents and neighbours of $a\in V$, respectively.
If we consider a vertex $a\in V$, then all vertices $b\in V\setminus(\pa(a)\cup\{a\})$ %that are not parents of $a$
are Granger non-causal for $a$, i.e., $\CY_{b} \nrarrow \CY_a\: \vert \: \CY_V$.  A direct consequence of \Cref{Proposition 5.6} (a) is then that $\CY_{V\setminus(\pa(a)\cup\{a\})} \nrarrow \CY_a\: \vert \: \CY_V$ holds. The same applies to neighbours of $a$ and the components being contemporaneously uncorrelated. Let $a\in V$ and $b \in V\setminus(\ne(a)\cup\{a\})$, then $a \inst b \notin E_{OG}$ and $\CY_{b} \nsim \CY_a\: \vert \: \CY_V$. \Cref{eq:correspondenceXAXa} yields $\CY_{V\setminus(\ne(a)\cup\{a\})} \nsim  \CY_a\: \vert \: \CY_V$.   This is the local Markov property. The same arguments work for the local orthogonality graph using  \Cref{Proposition 5.6} (b) and \Cref{eq:correspondenceXAXa local}, respectively.

\begin{proposition}\label{local Markov property} $\mbox{}$
\begin{itemize}
    \item[(a)] Let $G_{OG}=(V,E_{OG})$ be the orthogonality graph for $\CY_V$. Then $\CY_V$ satisfies the \textsl{local Markov property} with respect to $G_{OG}$, i.e., for all $a\in V$:
\begin{itemize}
\item[(i)\phantom{i}] $\CY_{V\setminus(\pa(a)\cup\{a\})} \nrarrow \CY_a\: \vert \: \CY_V$,
\item[(ii)] $\CY_{V\setminus(\ne(a)\cup\{a\})} \nsim \CY_a\: \vert \: \CY_V$.
\end{itemize}  $\mbox{}$ \vspace*{-0.2cm}
    \item[(b)] Let $G_{OG}^{0}=(V,E_{OG}^{0})$ be the local orthogonality graph for $\CY_V$. Then $\CY_V$ satisfies the \textsl{local Markov property} with respect to $G_{OG}^{0}$.
\end{itemize}
\end{proposition}

Furthermore, let
$
\pa(A) = \bigcup_{a\in A} \pa(a)$ and
$\ne(A) = \bigcup_{a\in A} \ne(a)$
denote the set of all parents and neighbours of vertices in $A \subseteq V$. Again, we expect components that are not parents of $A$ to be Granger non-causal for $A$ and components that are not neighbours of $A$ to be contemporaneously uncorrelated to $A$. This is the block-recursive Markov property and it also follows directly from \Cref{Proposition 5.6}, \Cref{eq:correspondenceXAXa} and   \Cref{eq:correspondenceXAXa local}.
%Again, the proof can be found in \Cref{subsec:proofsMarkprop}.

\begin{proposition}\label{block-recursive Markov property}
$\mbox{}$
\begin{itemize}
    \item[(a)] Let $G_{OG}=(V,E_{OG})$ be the orthogonality graph for $\CY_V$. Then $\CY_V$ satisfies the \textsl{block-recursive Markov property} with respect to $G_{OG}$, i.e., for all $A\subseteq V$:
\begin{itemize}
\item[(i)\phantom{i}] $\CY_{V\setminus(\pa(A)\cup A)} \nrarrow \CY_A\: \vert \: \CY_V$,
\item[(ii)] $\CY_{V\setminus(\ne(A)\cup A)} \nsim \CY_A\: \vert \: \CY_V$.
\end{itemize}  $\mbox{}$ \vspace*{-0.2cm}
    \item[(b)] Let $G_{OG}^{0}=(V,E_{OG}^{0})$ be the local orthogonality graph for $\CY_V$. Then $\CY_V$ satisfies the \textsl{block-recursive Markov property} with respect to $G_{OG}^0$.
\end{itemize}
\end{proposition}

%\begin{proof}[Proof of \Cref{block-recursive Markov property}] $\mbox{}$\\
%(a) (i) Let $a\in A$ and $b \in V\setminus(\pa(A)\cup A)$. Then $b \rarrow a \notin E_{OG}$ and $\CY_{b} \stackrel{}{\nrarrow} \CY_a\: \vert \: \CY_V$. A conclusion of \Cref{Proposition 5.6} (a) is then
%$\CY_{V\setminus(\pa(A)\cup A)} \stackrel{}{\nrarrow} \CY_A\: \vert \: \CY_V$.\\
%(ii) For the second part, let $a\in A$ and $b \in V\setminus(\ne(A)\cup A)$. Then again $a \inst b \notin E_{OG}^{}$ and $\CY_{b} \stackrel{}{\nsim} \CY_a\: \vert \: \CY_V$. \Cref{eq:correspondenceXAXa local} yields $\CY_{V\setminus(\ne(A)\cup A)} \stackrel{}{\nsim}  \CY_A\: \vert \: \CY_V$.  \\
%(b) The proof goes on the same way as in (a).
%\end{proof}

% The following composition and decomposition properties apply
%\begin{align}\label{remark markov properties}
%\begin{split}
%\CY_{b} \nrarrow \CY_a\: \vert \: \CY_V \quad \forall \: a\in A, \: b\in B
%\quad  &\Leftrightarrow \quad
%\CY_{B} \nrarrow \CY_A\: \vert \: \CY_V, \\
%\CY_{a} \nsim \CY_b\: \vert \: \CY_V \quad \forall \: a\in A, \: b\in B
%\quad  &\Leftrightarrow \quad
%\CY_{A} \nsim \CY_B\: \vert \: \CY_V.
%\end{split}
%\end{align}

In our (local) orthogonality graph all three Markov properties are fulfilled. {Thus, for example, using the local Markov property, we can infer from \Cref{fig: Two graphical models b} that $Y_{\{2,3\}}\nrarrow Y_1 \vert Y_{\{1,2,3\}}$ and  $Y_{\{2,3\}}\nrarrownull Y_1 \vert Y_{\{1,2,3\}}$.} However, the validity of Markov properties is not self-evident. For more information, see \cite{EI10}, Theorem 2.1 and Definition 2.3, who proposes to specify graphical time series models that satisfy the block-recursive Markov property as graphical time series models.
%\LS{Ich bin hier gerade total verwirrt\ldots Bei uns ist doch die Herausforderung in der ersten Komponente von einelementigen zu mehrelementigen Mengen zu kommen. Bei \cite{EI10}, Theorem 2.1 ist es aber die zweite Komponente also die composition property. Die ist bei uns ja nach Definition im localen orthogonality graph bzw im orthogonality graph wegen (d) erfüllt. In dem Fall macht dann der Literaturverweis wenig Sinn.}
For the visualisation of the various Markov properties at more complex examples {than the one in \Cref{fig: Two graphical models b}}, we also refer to \cite{EI10}, Example 2.1.

\subsubsection{Global Markov properties for the orthogonality graph  $G_{OG}=(V,E_{OG})$} \label{sec:global}

The three Markov properties we have discussed so far only encode relations with respect to $\CY_V$. However, for a better understanding of the causal structure, we are interested in relations with respect to partial information. An intuitive analysis of orthogonality graphs suggests that paths between vertices may be associated with relations between corresponding components given only the information provided by a subprocess.
 To this end, we first introduce the global AMP Markov property of \cite{AN00}, Definition 6, which relates paths in a graph to conditional orthogonality relations between variables. We then introduce the global Markov property, which provides sufficient criteria for Granger non-causality and contemporaneous uncorrelation. As we have to make additional assumptions for the local orthogonality graph, the results for the local model are presented in the next subsection, and here we only consider the \mbox{orthogonality graph.}

Let us start with the global AMP Markov property, where for $A, B, C\subseteq V$  disjoint, the fact that $A$ and $B$ are separated given $S$ implies that $\mathcal{L}_A$ and $\mathcal{L}_B$ are conditionally orthogonal given $\mathcal{L}_S$. %Here, we want to find a separation criterion that implies conditional orthogonality relations between the component series. 
%For mixed graphs, several separation criteria combine the notion of conditional orthogonality with a purely graph-theoretic concept of separation. 
But there are two main approaches to defining separation. The first approach is based on the path-oriented criterion ''$m$-separation``. The second approach uses separation in undirected graphs by applying the operation of augmentation or moralisation to appropriate subgraphs (\citealp{EI07}, Section 3). Since the second approach to defining a global Markov property is not straightforward in the sense that the graph is modified during the test, we just discuss the concept of $m$-separation and refer to \cite{VF24}, who compare the augmented causality graph, the augmentation of the causality graph, with the path diagram, an undirected graphical model for continuous-time stationary processes. 
To define the latter, we start with some definitions from graph theory, which can be found in \cite{EI07, EI10}.

\begin{definition}
Let $G=(V, E)$ be a mixed graph. A {path} $\pi$ between two vertices $a$ and $b$ is a sequence $\pi=\langle e_1,\ldots,e_n \rangle$ of edges $e_i \in E$, such that $e_i$ is an edge between $v_{i-1}$ and $v_i$ for some sequence of vertices $a=v_0,v_1,\ldots,v_n=b$. We say that $a$ and $b$ are the {endpoints} of the path, while $v_1,\ldots,v_{n-1}$ are {intermediate vertices}. $n$ is called {length} of the path. An intermediate vertex $c$ on a path $\pi$ is said to be a {collider} on the path, if the edges preceding and succeeding $c$ on the path both have an arrowhead or a dashed tail at $c$, i.e., $\rarrow c \larrow$, $\rarrow c \inst$, $\inst c \larrow $, $\inst c \inst$. Otherwise the vertex $c$ is said to be a {non-collider} on the path. %Notice that this classification only applies to intermediate vertices, the endpoints are neither colliders nor non-colliders.
A path $\pi$ between vertices $a$ and $b$ is said to be {$m$-connecting} given a set $S$ if
\begin{itemize}
\item[(a)] every non-collider on the path is not in $S$, and
\item[(b)] every collider on the path is in $S$,
\end{itemize}
otherwise we say the path is {$m$-blocked} given $S$. If all paths between $a$ and $b$ are $m$-blocked given $S$, then $a$ and $b$ are said to be {$m$-separated} given $S$. Similarly, sets $A$ and $B$ are said to be $m$-separated in $G$ given $S$, denoted by $A \msep B \: \vert \: S  \:\: [G]$, if for every pair $a \in A$ and $b \in  B$, $a$ and $b$ are $m$-separated given $S$.
\end{definition}

The $m$-separation is the natural extension of the $d$-separation for directed graphs (cf.~\citealp{PE94})
%and the classical separation for undirected graphs ()
to mixed graphs (cf.~\citealp{RI03}),
%, \cite{RI23}, \cite{Sadeghi:Lauritzen}),
and was earlier also called $d$-separation by \cite{SP98} and \cite{KO99}. Since we consider mixed graphs, which are generally not directed, we prefer the notion of $m$-separation. {For a motivation and visualisation of the respective definitions, we also refer to these papers.} Note that condition (a) differs from the original definition of $m$-connecting paths in \cite{RI03} and takes into account that we consider paths that can intersect themselves, as in \cite{EI07}. Nevertheless, the concepts of $m$-separation here and in \cite{RI03} are equivalent. In contrast, \cite{EI10} uses another natural extension of $d$-separation, called $p$-separation and introduced by \cite{Levitz} for chain graphs, where $\inst c \inst$ is considered a non-collider. Let us present the main result, the global AMP Markov property. 
%If the sets $A$ and $B$ are $m$-separated given $C$, then $\HY^A\in \mathcal{L}_{Y_A}$ and $\HY^B \in \mathcal{L}_{Y_B}$  are uncorrelated after removing all of the (linear) information provided by $\mathcal{L}_{Y_C}$. %As desired, the concept of $m$-separation complements the concept of separation for undirected graphs. %from section \Cref{The properties of the partial correlation graph}

\begin{theorem}\label{global markov property}
Let $G_{OG}=(V,E_{OG})$ be the orthogonality graph for $\CY_V$. 
Then $\CY_V$ satisfies the \textsl{global AMP Markov property} with respect to $G_{OG}$, i.e., for all disjoint subsets $A,B,C \subseteq V$,
\begin{align*}
A \msep B \: \vert \: C  \:\: [G_{OG}] \quad \Rightarrow \quad \mathcal{L}_{Y_A} \perp \mathcal{L}_{Y_B} \: \vert \: \mathcal{L}_{Y_{C}}.
\end{align*}
\end{theorem}

In words, if the sets $A$ and $B$ are $m$-separated given $C$, then $\HY^A\in \mathcal{L}_{Y_A}$ and $\HY^B \in \mathcal{L}_{Y_B}$  are uncorrelated after removing all of the (linear) information provided by $\mathcal{L}_{Y_C}$.
 %In the last two references, Eichler proposes the term global AMP Markov property for this characteristic.
%AMP here stands for the authors Andersson, Madigan and Perlman, who introduced this concept in the context of chain graphs.
A visualisation of the global AMP Markov property at a typical mixed graph is illustrated in \cite{EI10}, Example 2.1, {which can also be found in several of his articles.}
 The proof of \Cref{global markov property} is structured into three auxiliary statements that culminate in the actual proof, see \Cref{subsec:proofAMP}. Note that in the latter we need \Cref{Assumption purely nondeterministic of full rank} for the first time.
 %Non te that in the proof of the (linear) global AMP-Markov property, or more precisely in the proof of \Cref{Hilfslemma B3}, we need \Cref{Assumption purely nondeterministic of full rank} respectively \Cref{Eigenschaft für Hilfslemma B3} for the first time.

 \begin{remark}~
% \begin{itemize}
%\item[(a)] 
Similar statements can be found, e.g., in \cite{EI01}, Theorem 4.8, \cite{EI07}, Theorem 3.1 or \cite{EI10}, Theorem 4.1. However, the graphs defined there are based on different definitions of the edges and on processes in discrete time. The definition of the undirected edges in \cite{EI10} further differs from our definition. The linear continuous-time analogue of his definition is that $\mathcal{L}_{Y_A}(t,t+1) \perp \mathcal{L}_{Y_B}(t,t+1) \: \vert \: \mathcal{L}_{Y_{S}}(t) \vee \mathcal{L}_{S\setminus(A \cup B)}(t,t+1)$. Still most of the proofs can be carried over because it makes no difference whether one adds $\mathcal{L}_{S\setminus(A \cup B)}(t,t+1)$ or not.
%\item[(b)] \VF{In \cite{VF24} the augmented causality graph is compared with the path diagram, an undirected graphical model for continuous-time stationary processes, and uses the classical separation in an undirected graph. }
%\end{itemize}
 \end{remark}

The concept of $m$-separation provides a sufficient criterion for conditional orthogonality. However, we would also like to derive sufficient graphical conditions for Granger non-causality and processes being contemporaneously uncorrelated. An obvious first idea would be to start again with $m$-separation. However, this condition is stronger than necessary. A motivating example to only consider paths that point in the ''right`` direction is provided by \cite{EI07}, p.~341. We introduce further graph-theoretic notions %as in \cite{EI07}, p.~341f,
and then provide the main result.

\begin{definition}
Let $G = (V, E)$ be a mixed graph. A path $\pi$ between vertices $a$ and $b$ is called \textsl{$\mathbf{\textit{b}}$-pointing} if it has an arrowhead at the endpoint $b$. More generally, a path $\pi$ between $A$ and $B$ is said to be \textsl{$\mathbf{\textit{B}}$-pointing} if it is $b$-pointing for some $b\in B$.
Furthermore, a path $\pi$ between vertices $a$ and $b$ is said to be \textsl{bi-pointing} if it has an arrowhead at both endpoints $a$ and $b$.
\end{definition}

\begin{theorem}\label{global Markov property}
Let $G_{OG}=(V, E_{OG})$ be the orthogonality graph for $\CY_V$. Then $\CY_V$ satisfies the \textsl{global Markov property} with respect to $G_{OG}$, i.e., for all disjoint subsets $A,B,C\subseteq V$ the following conditions hold:
\begin{itemize}
\item[(a)] If every $B$-pointing path in $G_{OG}$ between $A$ and $B$ is $m$-blocked given $B \cup C$ then
$\CY_A \nrarrow \CY_B \: \vert \: \CY_{A \cup B \cup C}$.
\item[(b)] If $a \inst b \notin E_{OG}$ for all $a\in A$ and $b\in B$, and if every bi-pointing path in $G_{OG}$ between $A$ and $B$ is $m$-blocked given $A\cup B\cup C$, then $\CY_A \nsim \CY_B \: \vert \: \CY_{A\cup B \cup C}$.
\end{itemize}
\end{theorem}

A similar result in discrete time can be found in \cite{EI07}, Theorems 4.1 and 4.2, and \cite{EI10}, Theorem 4.2. For the visualisation of the global AMP Markov property at some mixed graph, we also refer to \cite{EI10}, Example 2.1.
Because of the properties of a graphoid in \Cref{properties of conditional orthogonality}, the block-recursive Markov property in \Cref{block-recursive Markov property} and \Cref{Hilfslemma B3},
 the proof can be carried out similarly as in \cite{EI07} and \cite{EI10}, respectively, and is therefore skipped.
 %However, we need to adapt the proofs to our model, which we do in \Cref{subsec:proofglobalcausalMP}.

%\begin{proposition}
%Let $G_{OG}^0=(V, E_{OG}^0)$ be the local orthogonality graph for $\CY_V$. Then $\CY_V$ satisfies the \textsl{(linear) global Markov property} with respect to $G_{OG}^0$, i.e., for all disjoint subsets $A,B,C\subseteq V$ the following conditions hold:
%\begin{itemize}
%\item[(1)] If every $B$-pointing path in $G_{OG}^0$ between $A$ and $B$ is $m$-blocked given $B \cup C$ then
%$\CY_A \stackrel{0}{\nrarrow} \CY_B \: \vert \: \CY_{A \cup B \cup C}$.
%\item[(2)] If $a \inst b \notin E_{OG}^0$ for all $a\in A$ and $b\in B$ and every bi-pointing path in $G_{OG}^0$ between $A$ and $B$ is $m$-blocked given $A\cup B\cup C$, then $\CY_A \stackrel{0}{\nsim} \CY_B \: \vert \: \CY_{A\cup B \cup C}$.
%\end{itemize}
%\end{proposition}

As a consequence of the global Markov property, we find that the $m$-separation $A \msep B \: \vert \: C \: [G_{OG}]$ is indeed too strong implying causality in both directions between $\CY_A$ and $\CY_B$ as well as their contemporaneous uncorrelation.
%As \cite{EI07}, Corrolary 4.3 noted, the sufficient conditions in \Cref{global Markov property} are fulfilled especially if $A \msep B \: \vert \: C$.
We refer to \cite{EI10}, Corollary 4.1, and \cite{EI07}, Corollary 4.3 for the proof. %who carries out the main idea and the first and second part of this argument.

\begin{corollary}\label{hinreichend global granger non-causal}
Let $G_{OG}=(V,E_{OG})$ be the orthogonality graph for $\CY_V$ and let \mbox{$A,B,C\subseteq V$} be disjoint subsets. Then $A \msep B \: \vert \: C \: [G_{OG}]$ implies
\begin{align*}
%
%\quad \Rightarrow \quad
\CY_A \nrarrow \CY_B \: \vert \: \CY_{A\cup B \cup C},  \quad
\CY_B \nrarrow \CY_A \: \vert \: \CY_{A\cup B \cup C}, \quad \text{ and } \quad
\CY_A \nsim \CY_B \: \vert \: \CY_{A\cup B \cup C}.
\end{align*}
\end{corollary}

\subsubsection{Global Markov properties for the local orthogonality graph $G_{OG}^{0}=(V,E_{OG}^{0})$}\label{Global Markov properties for the local orthogonality graph}

For the local orthogonality graph, the global Markov properties are, as expected, much more difficult due to the weaker definition of the edges. However, we still derive sufficient graphical conditions for local Granger non-causality and local contemporaneous uncorrelation. At least under additional assumptions, the property of $m$-separation implies local Granger non-causality in both directions between $\CY_A$ and $\CY_B$, and that they are locally contemporaneously uncorrelated. We start with a special case where $C= V \setminus (A \cup B)$. The proofs of this subsection are given in \Cref{proofslocalglobalMarkovproperty}.

\begin{proposition}\label{HilfslemmaABCistV}
Let $G_{OG}^0=(V, E_{OG}^0)$ be the local orthogonality graph for $\CY_V$ and let $A,B\subseteq V$ with $A \cap B =\emptyset$. Then
$A \msep B \: \vert \: V \setminus (A \cup B) \: [G_{OG}^0]$ implies
\begin{align*}
%
%\quad \Rightarrow \quad
\CY_A \nrarrownull \CY_B \: \vert \: \CY_{V},  \quad
\CY_B \nrarrownull \CY_A \: \vert \: \CY_{V}, \quad \text{ and } \quad
&\CY_A \nsimnull \CY_B \: \vert \: \CY_{V}.
\end{align*}
\end{proposition}

We consider a second special case where the block-recursive Markov property already leads to local Granger non-causality and local contemporaneous uncorrelation.

\begin{proposition}\label{LemmapaApaBinABC}
Let $G_{OG}^0=(V, E_{OG}^0)$ be the local orthogonality graph for $\CY_V$ and let $A,B,C\subseteq V$ be disjoint subsets. Suppose $\pa(A) \cup \pa(B) \subseteq A \cup B \cup C$.  Then \linebreak
$A \msep B \: \vert \: C \: [G_{OG}^0]$ implies
\begin{align*}
%
%\quad \Rightarrow \quad
\CY_A \nrarrownull \CY_B \: \vert \: \CY_{A \cup B \cup C},  \quad
\CY_B \nrarrownull \CY_A \: \vert \: \CY_{A \cup B \cup C}, \quad \text{ and } \quad
\CY_A \nsimnull \CY_B \: \vert \: \CY_{A \cup B \cup C}.
\end{align*}
\end{proposition}

\begin{remark} $\mbox{}$
\begin{itemize}
    \item[(a)] $\an(A \cup B \cup C)=A \cup B \cup C$ implies $\pa(A) \cup \pa(B) \subseteq A \cup B \cup C$. Therefore, we also have a graphical condition for causality and contemporaneous uncorrelation for ancestral subsets.
    \item[(b)] $\pa(B) \subseteq A \cup B \cup C$ is sufficient for $\CY_A \nrarrownull \CY_B \: \vert \: \CY_{A \cup B \cup C}$.
\end{itemize}
\end{remark}

For the proof of \Cref{LemmapaApaBinABC}, we need the left decomposition property of local Granger non-causality. % as given in \Cref{proofslocalglobalMarkovproperty}.

\begin{lemma}\label{Leftdecomposition}
%Let $\CY_V=(Y_V(t))_{t\in \R}$ be a $k$-dimensional process that satisfies \Cref{Assumption 1}.
Let $A, B, C, D \subseteq V$ be disjoint subsets. Then
\begin{align*}
\CY_{A \cup B} \nrarrownull \CY_C \: \vert \: \CY_{A \cup B \cup C \cup D}
\quad \Rightarrow \quad
\CY_{A} \nrarrownull \CY_C \: \vert \: \CY_{A \cup C \cup D}.
\end{align*}
\end{lemma}

\begin{remark}\mbox{}
\begin{itemize}
\item[(a)]  The right decomposition property, which is that $$\CY_{A} \nrarrownull \CY_{B \cup C} \: \vert \: \CY_{A \cup B \cup C \cup D} \quad \Rightarrow\quad \CY_{A} \nrarrownull \CY_B \: \vert \: \CY_{A \cup B \cup D}$$ cannot be expected. This can be explained as follows: It is possible that $\CY_{A}$ is non-causal for $\CY_{B \cup C}$ given $\CY_{A \cup B \cup C \cup D}$, since the corresponding information of $\CY_{A}$ is already present in $\CY_{C}$. However, if $\CY_{C}$ is omitted, there may be causal influence of $\CY_{A}$ on $\CY_{B}$. This topic has been addressed, e.g., by \cite{DI06} in the context of directed graphs.

\item[(b)] The lack of right decomposability is the key problem when trying to derive the global Markov property from the block-recursive Markov property. In the case that $A \cup B \cup C \subset V$, Corollary 1 and Proposition 2 of \cite{KO99} yield
\begin{align*}
A \msep B \: \vert \: C \:\: [G_{OG}^0]
&\quad \Leftrightarrow \quad
A' \msep B' \: \vert \: C \:\: [G_{OG,\an(A\cup B\cup C)}^0],
\end{align*}
for disjoint subsets $A'$ and $B'$ with $A \subseteq A'$, $B\subseteq B'$ and $A' \cup B' \cup C = \an(A\cup B\cup C)$ as in the proof of \Cref{global markov property}. According to \Cref{HilfslemmaABCistV}, we can conclude 
\begin{align*}
\CY_{A'} \nrarrownull \CY_{B'} \: \vert \: \CY_{A' \cup B' \cup C}, \quad
\CY_{B'} \nrarrownull \CY_{A'} \: \vert \: \CY_{A' \cup B' \cup C} \quad \text{ and } \quad
\CY_{A'} \nsimnull \CY_{B'} \: \vert \: \CY_{A' \cup B' \cup C},
\end{align*}
in $[G_{OG,\an(A\cup B\cup C)}^0]$. Since the definition of local Granger non-causality and local contemporaneous uncorrelation does not depend on whether we choose the subgraph with vertices in $A' \cup B' \cup C$ or the whole graph with vertices in $V$, the statements also hold for $[G_{OG}^0]$. But to obtain from this, e.g., $\CY_{A} \nrarrownull \CY_{B} \: \vert \: \CY_{A \cup B \cup C}$, we not only need the left decomposability but also the right decomposability.

\end{itemize}
\end{remark}

%\begin{remark}
%Let us summarize all the remarks made so far. $\varepsilon^*$ depends on $\varepsilon_M$ and thus on $\BFC, \BB$ and $\BS_L$ as well as the index sets $A$ and $B$. Now $\lambda^*$ depends on $\varepsilon^*$ and $k$, thus on $\BFC, \BB,\BS_L$ and $k$ as well as on the index sets $A$ and $B$. Furthermore $\varepsilon_K$ depends on the compact interval $K=[-\lambda^*, \lambda^*]$, the matrices $\BA, \BFC, \BB$ and $\BS_L$, as well as the index sets $A$ and $B$. Summarized, it depends on $\BA, \BFC, \BB, \BS_L$ and $k$  as well as the index sets $A$ and $B$.
%Therefore $\varepsilon_{AB}= \min\{\varepsilon_K, \varepsilon_M - \varepsilon^* \}$ depends on $\BA, \BFC, \BB, \BS_L$ and $k$  as well as the index sets $A$ and $B$. Finally $\varepsilon$  depends on $\BA, \BFC, \BB, \BS_L$ and $k$.
%\end{remark}
%\color{black}

%%%%%%%%%%%%%%%%%%%%%%%%%%%%%%%%%%%%%%%%%%%%%%%%%%%%%%%%%%%%%%
\section{Orthogonality graphs for MCAR processes}\label{sec:CGMCAR}

To gain a deeper understanding of the theoretical concept of a (local) orthogonality graph, we apply the graphical models to the class of causal MCAR processes. %as in  \Cref{Definition des CAR Prozesses}
%and its special case, the causal Ornstein-Uhlenbeck processes MCAR$(1)$=OU.
We not only give theoretical results but also interpret them and relate them to the results of \cite{EI07} in discrete time. First, we give a brief introduction to MCAR processes and show that they satisfy the assumptions of the (local) orthogonality graph. We then derive linear predictors of MCAR processes, which we require to characterise the edges; which is the ultimate goal of this section. The details of the proofs of this section are moved to \Cref{Appendix C}.
%%%%%%%%%%%%%%%%%%%%%%%%%%%%%%%%%%%%%%%%%%%%%%%%%%%%%%%%%%%%%%%%%%
\subsection{MCAR processes}

A multivariate $k$-dimensional  continuous-time AR (MCAR) process is a continuous-time version of the well-known vector AR (VAR) process in discrete time. % They are applied in diversified fields as, e.g., signal processing and control (cf.~\cite{GarnierWang2008, LarssonMossbergSoederstroem2006}),
%high-frequency financial econometrics (cf.~\cite{Todorov2009})  and financial mathematics (cf.~\cite{BenthKoekebakkerZakamouline2010}).
The driving process  is a $k$-dimensional  Lévy process $(L(t))_{t\in\R}$ as defined in \Cref{Example:Ornstein-Uhlenbeck process}
%which is an $\R^k$-valued stochastic process with $L(0)=0_k$ $\mathbb{P}$-a.s., stationary and independent increments and c\`adl\`ag sample paths. In the following we work with two-sided Lévy processes $L=(L(t))_{t\in \R}$ which are obtained from two independent copies $(L_1(t))_{t\geq 0}$ and $(L_2(t))_{t\geq 0}$ of a one-sided Lévy process via
%\begin{align*}
%L(t) =
%\begin{cases}
%L_1(t) &\mbox{if } t \geq 0, \\
%-\lim_{s\nearrow -t} L_2(s) & \mbox{if } t<0. %\end{cases}
%\end{align*}
and satisfies the following assumption throughout the paper.

\begin{assumption}\label{Assumption on Levy process}
The two-sided Lévy process $L=(L(t))_{t\in \R}$ satisfies $\BE L(1)=0_k$ and  $\BE \Vert L(1) \Vert^2 < \infty$ with $\BS_L=\BE[L(1)L(1)^\top ]$.
\end{assumption}

 The idea is then that a $k$-dimensional MCAR$(p)$ process  is  the solution of the stochastic differential equation
\begin{equation} \label{eq1.1}
     {P}(D)Y(t)=D L(t) \quad \mbox{ for } t\in \R,
\end{equation}
where $D$ is the differential operator with respect to $t$ and
\begin{equation} \label{Pol}
    {P}(\lambda)\coloneqq I_{k}\lambda^p+A_1\lambda^{p-1}+\cdots+A_p, \quad \lambda \in \C,
\end{equation}
 is the autoregressive  polynomial, respectively with $A_1,\ldots, A_p\in M_k(\R)$. However, this is not the formal
definition of an MCAR process, since a Lévy process is not differentiable. The formal definition of a Lévy-driven causal MCAR process used here goes back to \cite{MA07}, Definition 3.20. However, one-dimensional Gaussian CARMA
processes were already investigated by \cite{doob:1944} % in 1944
(cf.~\citealp{DO60}) and L\'{e}vy-driven CARMA processes were propagated by Peter Brockwell at the beginning of this century,  see \cite{BR14} and \cite{BrockwellLindner2024} for an overview. Very early Gaussian MCAR processes were already studied in the economics literature, e.g.,~in \cite{Harvey:Stock:1985, Harvey:Stock:1988, HarveyStock1989} and
were further explored in the well-known paper of \cite{Bergstrom:1997}.

%see as well \VF{include}
%\cite{SC12}, eq.~(3.3)ff., \cite{SC122}, eq.~(3.3)ff., \cite{BR15}, Theorem 3.2
%ggf. \cite{KE17} \cite{BA18pre}

\begin{definition}\label{Definition des CAR Prozesses}
Let $\left( L(t) \right)_{t \in \R}$ be a two sided $k$-dimensional Lévy process. %that satisfies \Cref{Assumption on Levy process}.
%with $\BE L(1)=0_k$ and  $\BE \Vert L(1) \Vert^2 < \infty$.
Further, let \mbox{$\BA \in M_{kp}(\R)$,} $p\geq 1$ with $\sigma(\BA)\subseteq (-\infty, 0) + i \R$, such that
\begin{align*}
\BA &=
\begin{pmatrix}
  0_k & I_k & 0_k & \cdots & 0_k \\
  0_k & 0_k & I_k & \ddots & \vdots \\
  \vdots &  & \ddots & \ddots & 0_k \\
  0_k & \cdots & \cdots & 0_k & I_k \\
  -A_p & -A_{p-1} & \cdots & \cdots & -A_1
\end{pmatrix},
\end{align*}
$\BB^\top  = (0_k, \ldots, 0_k, I_k) \in M_{k\times kp}(\R)$ and $\BFC = (I_k, 0_k, \ldots, 0_k) \in M_{k\times kp}(\R)$.
%\begin{align*}
%\BB=
%\begin{pmatrix}
%0_k \\
%\vdots \\
%0_k \\
%I_k
%\end{pmatrix}
%\in M_{kp\times k}(\R), \quad
%\BFC &=\begin{pmatrix}
%  I_k & 0_k & \cdots & 0_k
%\end{pmatrix} \in M_{k\times kp}(\R).
%\end{align*}
Then the process $\CY_V=(Y_V(t))_{t\in\R}$ given by $$Y_V(t) = \BFC X(t),$$ where $\CX=(X(t))_{t\in\R}$ is the unique $kp$-dimensional stationary solution  of the state equation
\begin{align} \label{state}
dX(t)= \BA X(t)dt+ \BB dL(t),
\end{align}
is called \textsl{(causal) MCAR$(p)$ process}.
\end{definition}

Indeed, if $p=1$, the MCAR(1) process corresponds to the Ornstein-Uhlenbeck process of \Cref{Example:Ornstein-Uhlenbeck process}. We summarise important properties of causal MCAR processes used in this paper. % and derived in \cite{MA07} and \cite{SC122}, respectively.
Details are given in \cite{MA07} and \cite{SC122}. 

\begin{lemma} \label{Lemma 5.2}
Let $\CY_V$ be a causal %$k$-dimensional
MCAR$(p)$ process.
% such that the driving $k$-dimensional Lévy process satisfies \Cref{Assumption on Levy process}.
%Let  $\CY_V=(Y_V(t))_{t\in \R}$  be the $k$-dimensional MCAR process as defined in \Cref{Definition des CAR Prozesses} with driving Lévy process $(L(t))_{t\geq 0}$.
Then the following results hold:
\begin{itemize}
    \item[(a)] The unique stationary solution $\CX$ of the state equation \eqref{state} has the representation
    \begin{align*}
X(t) = \int_{- \infty}^{t} e^{\BA (t-u)} \BB dL(u), \quad t\in \R,
\end{align*}
and $$X(t)=e^{\BA(t-s)}X(s)+\int_s^t e^{\BA(t-u)}\BB dL(u),\quad s,t\in\R, s<t.$$
    \item[(b)]  We denote the $j$-th $k$-block of $\CX$ by
\begin{align} \label{block}
X^{(j)}(t)=
\begin{pmatrix}
X_{(j-1)k+1}(t)\\
\vdots\\
X_{jk}(t)
\end{pmatrix}, \quad t\in\R, \: j=1,\ldots,p,
\end{align}
such that $X(t)=(X^{(1)}(t)^\top ,\ldots,X^{(p)}(t)^\top )^\top $, $t\in\R$. Suppose $\Phi_L(\cdot)$ is the $k$-dimensional random orthogonal measure of the Lévy process $L$, i.e,
\begin{align*}
\Phi_L(\left[a,b\right))=\int_{-\infty}^{\infty}\frac{e^{-i\lambda a}-e^{-i\lambda b}}{2\pi i\lambda}\, dL(\lambda),  \quad -\infty<a<b<\infty,
\end{align*}
with spectral measure
$F_L(d\lambda)=\BS_L / 2\pi \,d\lambda$ and $\BE(\Phi_L(\left[a,b\right)))=0_k$. Then
\begin{align*}
    X^{(j)}(t)=\int_{-\infty}^{\infty} e^{i\lambda t}(i\lambda)^{j-1}P(i\lambda)^{-1} \,\Phi_L(d\lambda), \quad t\in\R,
\end{align*}
and in particular, $Y_V(t)=X^{(1)}(t)=\int_{-\infty}^{\infty} e^{i\lambda t}P(i\lambda)^{-1} \,\Phi_L(d\lambda)$, $t\in\R$.
\item[(c)] The  covariance function $(c_{XX}(t))_{t\in\R}$ of $\CX$ is
\begin{align}\label{alternative representation of covariance for MCARMA}
\begin{split}
c_{XX}(t) &=c_{XX}(-t)^\top = \BE[X(t+h) \overline{X(h)}^\top ] = e^{\BA t} \Gamma(0), \quad t\geq 0,
%c_{XX}(t) &= \Gamma(t) = \BE[X(t+h) \overline{X(h)}^\top ] = \Gamma(0) e^{- \BA^\top  t}, \quad t\leq0,
\end{split}
\end{align}
where
$
\Gamma(0)  = \int_{0}^{\infty} e^{\BA u} \BB \BS_L \BB^\top  e^{\BA^\top  u}du
$
satisfies
\begin{align}\label{Dichtezusammenhang}
 \BA\Gamma(0) + \Gamma(0) \BA^\top =- \BB \BS_L \BB^\top .
\end{align}
\item[(d)] The spectral density  of the causal MCAR process $\CY_V$ is
\begin{align*}
  f_{Y_V Y_V}(\lambda)
    &=\frac{1}{2\pi} P(i\lambda)^{-1}\BS_L \left(P(-i\lambda \right)^{-1})^\top \\
    &=\frac{1}{2\pi} \BFC \left(i\lambda I_{kp}- \BA \right)^{-1} \BB \BS_L \BB^\top   \left(-i\lambda I_{kp}-\BA^\top  \right)^{-1} \BFC^\top ,\quad \lambda \in\R.
\end{align*}
\end{itemize}
\end{lemma}

We point out some more properties that we use later in the paper.

\begin{remark} \label{Remark 5.4} $\mbox{}$
\begin{itemize}
\item[(a)] If $\BS_L>0$,
then %the covariance matrix
$c_{XX}(0)>0$. %of $X(0)$ is positive definite.
Indeed, $\BB$ is of full rank and thus the assumptions of \cite{SC122}, Corollary 3.9, are satisfied.
\begin{comment}
Note
\begin{align*}
c_{X}(0) 	= \int_0^\infty e^{\BA u} \BB \BS_L \BB^\top  e^{\BA^\top  u} du
 			= \int_0^\infty e^{\BA u} MM^\top  e^{\BA^\top  u} du.
\end{align*}
Since $\BS_L$ is positive definite, $\BB$ is of full rank and thus $ \BB \BS_L \BB^\top $ positive definite, $M$ is a positive definite matrix due to \cite{HO13}, Corollary 7.2.8. %Since $M$ is of full rank, the matrix
Thus
\begin{align*}
\left(M, \BA M, \BA^2M,\ldots,\BA^{kp-1}M \right)
\end{align*}
is of full rank too and $c_{X}(0) >0$ due to \cite{BE09}, Corollary 12.6.3.
\end{comment}
\item[(b)] Since the matrix exponential is continuous, we have $c_{XX}(t) \rightarrow c_{XX}(0)$ for $t \rightarrow 0$. Now, $c_{Y_V Y_V}(\cdot)$ corresponds to the upper left $k \times k$ block of $c_{XX}(\cdot)$. Thus, $c_{Y_VY_V}(t) \rightarrow c_{Y_V Y_V}(0)$ for $t \rightarrow 0$. \cite{CR39}, Lemma 1, %Lemma \ref{Lemma and Assumption 4}
then gives that the causal MCAR process $\CY_V$ is mean-square continuous.
\end{itemize}
\end{remark}

For the definition of the local orthogonality graph and, in particular, the local Granger non-causality and the local contemporaneous uncorrelation, respectively, we need some knowledge about the existence and the description of the mean-square derivatives of the MCAR process. Therefore, we note the following.

\begin{remark} \label{Remark 5.3}
Due to the spectral representation of $X^{(j)}$ given in \eqref{block}, we directly obtain the spectral density
\begin{align*}
    f_{X^{(j)}X^{(j)}}(\lambda)=\frac{1}{2\pi} (i\lambda)^{j-1} P(i\lambda)^{-1}\BS_L(P(-i\lambda)^{-1})^\top (-i\lambda)^{j-1}, \quad \lambda \in\R.
\end{align*}
Therefore, it holds that $\int_{-\infty}^{\infty}\lambda^2 \|f_{X^{(j)}X^{(j)}}(\lambda)\|\,d\lambda <\infty$ for $j=1,\ldots,p-1$, but \linebreak $\int_{-\infty}^{\infty}\lambda^2 \|f_{X^{(p)}X^{(p)}}(\lambda)\|\,d\lambda =\infty$. Thus, a conclusion of \Cref{Proposition 3.7} is that  the process $\CX^{(j)}$ is mean-square differentiable with derivative
\begin{align} \label{eq 3.4}
    D^{(1)}X^{(j)}(t)=X^{(j+1)}(t), \quad j=1,\ldots,p-1,
\end{align}
while for $\CX^{(p)}$ the mean-square derivative does not exist. With $Y_V(t)=X^{(1)}(t)$ in mind, we receive iteratively from \eqref{eq 3.4} that $\CY_V$ is $(p-1)$-times mean-square differentiable with
\begin{align} \label{derivatives Y}
    D^{(j)}Y_V(t)=X^{(j+1)}(t), \quad j=1,\ldots,p-1,
\end{align}
but the $p$-th derivative does not exist. By the same arguments, we receive that for any component $Y_v$, $v\in V$, of $Y_V$ there is no derivative higher than $(p-1)$.
\end{remark}

\subsection{Orthogonality graph for MCAR processes}

In the following, we verify that the (local) orthogonality graph for the MCAR process is well-defined. Therefore, we have to check that the Assumptions \ref{Assumption an Dichte} and \ref{Assumption purely nondeterministic of full rank} are satisfied.

\begin{proposition}\label{Dichteannahme erfüllt}
Let $\CY_V$ be a causal %$k$-dimensional
MCAR$(p)$ process % such that the driving $k$-dimensional Lévy process satisfies \Cref{Assumption on Levy process}.
with \mbox{$\BS_L>0$.}
Then $\CY_V$ satisfies Assumptions \ref{Assumption an Dichte} and \ref{Assumption purely nondeterministic of full rank}.
\end{proposition}

The proof of \Cref{Assumption an Dichte}  is elaborate and is therefore presented in the Supplementary Material \ref{suppl:section 6}. However, the basic idea is simple. Note, $\BS_L>0$ results in $f_{Y_VY_V}(\cdot)>0$. On the one hand, we prove that an epsilon bound can always be found on compact intervals. On the other hand, the matrix function converges to a boundary matrix which can also be bounded. Together this then gives \Cref{Assumption an Dichte}.
The proof of \Cref{Assumption purely nondeterministic of full rank} is also given in the Supplementary Material \ref{suppl:section 6} and is based on a characterisation of purely non-deterministic processes by limits of orthogonal projections. It was expected that the MCAR$(p)$ process would satisfy this assumption since in our case the driving Lévy process has no drift term.
%Of course, since the driving Lévy process has no drift term, the MCAR$(p)$ is also purely non-deterministic and satisfies  \Cref{Assumption purely nondeterministic of full rank}.
%\LS{Den letzten Satz verstehe ich nicht und bei den Beweisen hinten führen wir ihn trotzdem? ist als Motivation gedacht, nicht als expliziter Beweis, Konnte man vorher vermuten\ldots}
Since Assumptions \ref{Assumption an Dichte} and \ref{Assumption purely nondeterministic of full rank} hold, a direct consequence of \Cref{sec:path_diagrams}  is then the following.

\begin{proposition}
Let $\CY_V$ be a causal %$k$-dimensional
MCAR$(p)$ process with % such that the driving $k$-dimensional Lévy process satisfies \Cref{Assumption on Levy process}.
%Furthermore, suppose
\mbox{$\BS_L>0$.}
 If we define $V=\{1,\ldots,k\}$ as the vertices and the edges $E_{OG}$ via
\begin{itemize}
\item[(i)\phantom{i}]   $a \rarrow b \notin E_{OG} \quad \Leftrightarrow \quad \CY_a \nrarrow \CY_b \: \vert \: \CY_V$,
\item[(ii)]  $a \inst b \notin E_{OG} \quad \Leftrightarrow \quad \CY_a \nsim \CY_b \: \vert \: \CY_V$,
\end{itemize}
for $a,b\in V$, $a\neq b$, then the orthogonality graph  $G_{OG}=(V,E_{OG})$ for the MCAR process $\CY_V$ is well-defined and satisfies
the pairwise, local, block-recursive, global AMP  and global Markov property.
\end{proposition}

%Therefore, we see that it is sufficient to characterise $\CY_a \nrarrow \CY_b \: \vert \: \CY_V$ and $\CY_a \nsim \CY_b \: \vert \: \CY_V$ for an MCAR process, which is topic of \Cref{sec:charMCAR}.

If we look at the local orthogonality graph, we also get the following from \Cref{sec:path_diagrams}.

\begin{proposition}
Let $\CY_V$ be a causal %$k$-dimensional
MCAR$(p)$ process with
%the driving $k$-dimensional Lévy process satisfies \Cref{Assumption on Levy process}. Furthermore, suppose
\mbox{$\BS_L>0$.}
 If we define $V=\{1,\ldots,k\}$ as the vertices and the edges $E_{OG}^0$ via
\begin{itemize}
\item[(i)\phantom{i}]   $a {\rarrow} b \notin E_{OG}^0 \quad \Leftrightarrow \quad \CY_a \nrarrownull \CY_b \: \vert \: \CY_V$,
\item[(ii)]  $a \inst b \notin E_{OG}^0 \quad \Leftrightarrow \quad \CY_a \nsimnull \CY_b \: \vert \: \CY_V$,
\end{itemize}
for $a, b\in V$, $a\neq b$, then the local orthogonality graph  $G_{OG}^0=(V, E_{OG}^{0})$ for the MCAR process $\CY_V$ is well-defined and satisfies the pairwise, local and block-recursive Markov property. Furthermore, the statements of Propositions \ref{HilfslemmaABCistV} and \ref{LemmapaApaBinABC} hold.
\end{proposition}

%From this, it follows directly that we have to characterise $\CY_a \stackrel{0}{\nrarrow} \CY_b \: \vert \: \CY_V$ and $\CY_a \stackrel{0}{\nsim} \CY_b \: \vert \: \CY_V$ for an MCAR process as is done in \Cref{sec:charMCAR} as well.

\subsection{Prediction of MCAR processes}\label{subsec:predMCAR}
To characterise the different Granger causalities and contemporaneous correlations as is done, e.g., in Theorems \ref{Charakterisierung als Gleichheit der linearen Vorhersage} and \ref{Charakterisierung als Gleichheit der linearen Vorhersage 2}, respectively, we need to compute the linear predictions of the MCAR process and its derivatives on the different subspaces. % which is the aim of this subsection.
To do this, we first give a suitable representation for $Y_v(t+h)$.
\Cref{subsec:proofpredMCAR} contains all proofs of this subsection.
%We include the evidence here as it is instructive for understanding the structure of the process.

\begin{lemma}\label{Yb für MCAR}
Let $\CY_V$ be a causal %$k$-dimensional
MCAR$(p)$ process. Further, let $t\in \R$,  $h\geq 0$, and $v\in V$. %, such that the driving $k$-dimensional Lévy process satisfies \Cref{Assumption on Levy process}.
Then
\begin{align*}
Y_v(t+h)= e_v^\top  \BFC e^{\BA h} \sum_{j=1}^p \BFE_j D^{(j-1)}Y_V(t) + e_v^\top  \BFC \int_t^{t+h} e^{\BA (t+h-u)} \BB dL(u) \quad \mathbb{P}\text{-a.s.}
\end{align*}
 % Here $e_v \in \R^k$ is the $v$-th unit vector and
%\begin{align*}
%\BFE_j  =
%\begin{pmatrix}
%     0_{k(j-1) \times k} \\
%     I_{k\times k} \\
%     0_{k(p-j) \times k}
%\end{pmatrix} \in M_{kp \times k}(\R), \quad j=1,\ldots,p.
%\end{align*}
\end{lemma}

From this representation of $Y_v(t+h)$ we conclude that on the one hand, the past $(Y_V(s), s\leq t)$ of all components and on the other hand, the future of the Lévy process $(L(t+h)-L(s), t\leq s \leq t+h)$ are relevant for $Y_v(t+h)$. %Thus, we obtain a representation
%\begin{align*}
%Y_v(t+h) = f_1(Y_V(s), s\leq t) + f_2(L(s), t\leq s \leq t+h),
%\end{align*}
%$h\geq 0$, $t\in \R$.
Based on this knowledge, we specify the orthogonal projections.

\begin{proposition}\label{Projektionen CAR berechnet}
Let $\CY_V$ be a causal %$k$-dimensional
MCAR$(p)$ process. %, such that the driving $k$-dimensional Lévy process satisfies \Cref{Assumption on Levy process}. Let
%\LS{such that the driving $k$-dimensional Lévy process satisfies \Cref{Assumption on Levy process}?}
Further, let $t\in \R$,  $h\geq 0$, $S \subseteq V$, and $v\in V$. Then
\begin{align*}
P_{\mathcal{L}_{Y_S}(t)}Y_v(t+h)
= & \: e_v^\top  \BFC e^{\BA h}  \sum_{s \in S} \sum_{j=1}^p \BFE_j e_{s} D^{(j-1)} Y_s (t) \\
  &+ e_v^\top  \BFC e^{\BA h}  \sum_{s \in V\setminus S} \sum_{j=1}^p \BFE_j e_{s} P_{\mathcal{L}_{Y_S}(t)} D^{(j-1)} Y_s (t) \quad \mathbb{P}\text{-a.s.}
%= &\sum_{s \in S} \sum_{j=1}^p e_v^\top  e^{\BA h} e_{s + k(j-1)} D^{(j-1)} Y_s (t) \\
%  &+ \sum_{s \in V\setminus S} \sum_{j=1}^p e_v^\top  e^{\BA h} e_{s + k(j-1)} P_{\mathcal{L}_{Y_S}(t)} D^{(j-1)} Y_s (t)
\end{align*}
\end{proposition}

According to \Cref{Yb für MCAR}, the basic idea of the proof is simple:  $Y_s(t)$ and its derivatives are already in $\mathcal{L}_{Y_S}(t)$ (see \Cref{Remark 3.6}) and are therefore projected onto themselves. Additionally, $\sigma(Y_S(s):\, s\leq t)$ and $\sigma(L(t+h)-L(s):\, t \leq s \leq t+h)$ are independent and thus, $e_v^\top  \BFC \int_t^{t+h} e^{\BA (t+h-u)} \BB dL(u)$ is projected on zero. % The technical aspects of the proof can be found in
%\Cref{subsec:proofpredMCAR}.

\begin{remark}\label{Projektion für MCAR für S=V}
For $S=V$ we get the explicit representation
\begin{align*}
P_{\mathcal{L}_{Y_V}(t)}Y_v(t+h)
= e_v^\top  \BFC e^{\BA h}  \sum_{s \in V} \sum_{j=1}^p \BFE_j e_{s} D^{(j-1)}Y_s(t)
= e_v^\top  \BFC e^{\BA h} X(t),
\end{align*}
as in \cite{BR15} for univariate CARMA processes. For an explicit representation in the case $S\subset V$ the methods %developed
in \cite{RO67}, III, 5, can be applied but this is quite elaborate.
\end{remark}

Next, we calculate the projections of $ D^{(p-1)}\CY_V$, which we require for the characterisation of local Granger causality and local contemporaneous correlation. %The proof is given in \Cref{subsec:proofpredMCAR}.

\begin{lemma}\label{Projektionen mit lim}
Let $\CY_V$ be a causal %$k$-dimensional
MCAR$(p)$ process. %LS{such that the driving $k$-dimensional Lévy process satisfies \Cref{Assumption on Levy process}?}
Further, let $t\in \R$,  $h\geq 0$, $S \subseteq V$, and $v\in V$. Then
\begin{align*}
& P_{\mathcal{L}_{Y_{S}}(t)} \left(D^{(p-1)} Y_v(t+h)- D^{(p-1)} Y_v(t)\right) \\
&\quad =\: e_{v}^\top  \BFE_p^\top  \left( e^{\BA h} -I_{kp} \right) \sum_{s \in S}
\sum_{j=1}^p  \BFE_j e_s  D^{(j-1)} Y_s(t) \nonumber \\
&\quad\quad\: +  e_{v}^\top  \BFE_p^\top  \left( e^{\BA h} -I_{kp} \right)
\sum_{s \in V\setminus S} \sum_{j=1}^p \BFE_j e_s  P_{\mathcal{L}_{Y_{S}}(t)}D^{(j-1)} Y_s(t)
\quad\mathbb{P}\text{-a.s.} \nonumber
\end{align*}
and
\begin{align*}
 D^{(p-1)} Y_v(t+h)- P_{\mathcal{L}_{Y_{V}}(t)} D^{(p-1)} Y_v(t+h)
= e_v^\top  \BFE_p^\top  \int_t^{t+h} e^{\BA (t+h-u)} \BB dL(u) \quad\mathbb{P}\text{-a.s.}
\end{align*}
\end{lemma}

%%%%%%%%%%%%%%%%%%%%%%%%%%%%%%%%%%%%%%%%%%%%%%%%%%%%%%%%%%%%%%%%%%%%%%
\subsection{Characterisation of the directed and undirected influences for the  MCAR process}\label{sec:charMCAR}
%\subsubsection{The characterisation of the directed edges for the MCAR process}
In this subsection, we focus on criteria for the directed and undirected influences for causal MCAR$(p)$ processes. All proofs of this subsection are carried out in \Cref{subsec:proofcharMCAR}.
We start with a characterisation of (local) Granger causality for an MCAR process, which is well suited for interpretation and for comparison with \cite{EI07} in discrete time. The proofs %is carried out in \Cref{subsec:proofcharMCAR} and
are based on the characterisation of (local) Granger causality in \Cref{Charakterisierung als Gleichheit der linearen Vorhersage} using the orthogonal projections from \Cref{subsec:predMCAR}. Note that for the definition of local Granger causality and local contemporaneous correlation, we use that all components of $\CY_V$ are $(p-1)$-times mean square differentiable, but the $p$-th derivative does not exist (cf.~\Cref{Remark 5.3}), so that $j_v=p-1$ for any $v\in V$.

\begin{proposition}\label{Erste Charakterisierung gerichtete Kante für CAR}
Let $\CY_V$ be a causal %$k$-dimensional
MCAR$(p)$ process with
%\LS{such that the driving $k$-dimensional Lévy process satisfies \Cref{Assumption on Levy process}?}. Suppose
\mbox{$\BS_L>0$.} Further, let $a,b\in V$ and $a \neq b$. Then the following holds.
%\begin{align*}
%a \rarrow b \notin E_{OG}
%\quad \Leftrightarrow \quad
% e_b^\top  \BFC e^{\BA h} \BFE_j e_a = e_b^\top  e^{\BA h} e_{k(j-1)+a}=0 \quad \text{for} \quad 0 \leq h \leq 1, j=1,\ldots,p.
%\end{align*}
\begin{itemize}
\item[(a)]
 \makebox[2,62cm][l]{$\CY_a \nrarrow \CY_b \: \vert \: \CY_V$} $ \Leftrightarrow \quad
  \left[ \BFC e^{\BA h} \BFE_j \right]_{ba}=\left[ e^{\BA h} \right]_{b\: k(j-1)+a}=0 \quad \forall\, h\in [0,1], \: j=1,\ldots,p.$
 \item[(b)]  \makebox[2,6cm][l]{$\CY_a \nrarrownull \CY_b \: \vert \: \CY_V $} $\Leftrightarrow \quad \left[ \BFE_p^\top  \BA \BFE_j \right]_{ba} = \left[ A_j \right]_{ba} = 0 \quad \forall \,\text{$j=1,\ldots,p$.}$
\end{itemize}
%\begin{align*}
%a \rarrow b \notin E_{OG} \quad &\Leftrightarrow \quad \left[ e^{\BA h} \right]_{b, k(j-1)+a}=0 \quad \text{for} \quad 0 \leq h \leq 1, j=1,\ldots,p\\
% \quad &\Leftrightarrow \quad e_b^\top  \BFC e^{\BA h} \BFE_je_a=\bzero^\top  \quad  \text{for} \quad 0 \leq h \leq 1, j=1,\ldots,p.
%\end{align*}
\end{proposition}

These characterisations of  (local) Granger causality are convenient since we no longer need to compute and compare orthogonal projections. Moreover, the deterministic criteria depend only on the state transition matrix $\BA$ and %the covariance matrix $\BS_L$ of the driving Lévy process, but
not on the %distribution of the
driving \mbox{Lévy process.}

Let us now move on to contemporaneous uncorrelation and also give a first characterisation specifically related to the structure of an MCAR$(p)$ process. Similar to \Cref{Erste Charakterisierung gerichtete Kante für CAR}, the proof is %is given in \Cref{subsec:proofcharMCAR} and is as well
based on the characterisation of contemporaneous uncorrelation by orthogonal projections from \Cref{subsec:predMCAR}   and \eqref{MR}.

\begin{proposition}\label{Einfache Charakterisierung ungerichtete Kante für CAR}
Let $\CY_V$ be a causal %$k$-dimensional
MCAR$(p)$ process. % \LS{such that the driving $k$-dimensional Lévy process satisfies \Cref{Assumption on Levy process}?}.
Further, let $a,b\in V$ and $a \neq b$. Then the following holds.
\begin{itemize}
\item[(a)]
$\CY_a \nsim \CY_b \: \vert \: \CY_V
\quad \Leftrightarrow \quad
\left[ \int\limits_0^{\min(h,\tilde{h})} \BFC e^{\BA (h-u)} \BB \BS_L \BB^\top  e^{\BA^\top (\tilde{h}-u)} \BFC^\top  du\right]_{ab} =0$ $\forall\,h, \tilde{h} \in [0,1].$
\item[(b)] $\CY_a \nsimnull \CY_b \: \vert \: \CY_V \:\: \Leftrightarrow \quad
\left[ \BS_L \right]_{ab} = 0.$
\end{itemize}
%\begin{align*}
%a \inst b \notin E_{OG}
%\quad \Leftrightarrow \quad
%e_a^\top  \BFC \int_0^{\min(h,\tilde{h})} e^{\BA (h-s)} \BB \BS_L \BB^\top  e^{\BA^\top (\tilde{h}-s)} ds \: \BFC^\top  e_b =0 \quad  \text{for} \quad 0\leq h, \tilde{h}\leq 1.
%\end{align*}
\end{proposition}

\begin{remark}~
\begin{itemize}
\item[(a)] \cite{CO96} investigate non-stationary Brownian motion driven MCAR processes on local Granger causality and local instantaneous causality, which are similar to our concepts of local Granger causality and local contemporaneous correlation. %However, they are defined for semimartingales and not set in the context of graphs. 
 In their Proposition 20, \cite{CO96} obtain that $Y_a$ does not locally Granger cause $Y_b$ if and only if $\left[ A_j \right]_{ba} = 0$, for $j=1,\ldots,p$, as in our \Cref{Erste Charakterisierung gerichtete Kante für CAR}. Furthermore, there is no local instantaneous causality between $Y_a$ and $Y_b$ if and only if $\left[ \BS_L \right]_{ab} = 0$, as in \Cref{Einfache Charakterisierung ungerichtete Kante für CAR} for the local orthogonality graph. Statements about local Granger causality and local instantaneous causality for subprocesses under possible partial information, as we present with the Markov properties in \Cref{sec:Markprop}, are not available there.
\item[(b)]
Furthermore, as a generalisation of \cite{DI06}, \cite{Mogensen:Hansen:2022} study the local independence graph for It$\hat{\text{o}}$ processes where the graph models the local independence structure of the underlying stochastic process; in contrast, we model local orthogonality. A special case is the Brownian motion driven Ornstein-Uhlenbeck process.
The edges of the local independence graph of a Brownian motion driven Ornstein-Uhlenbeck process (cf.~Proposition 7 in \citealp{Mogensen:Hansen:2022}) are the same as given here in Propositions \ref{Erste Charakterisierung gerichtete Kante für CAR} and \ref{Einfache Charakterisierung ungerichtete Kante für CAR}, i.e., there is no directed edge from $a$ to $b$ if and only if $[\BA]_{ba}=0$, and there is no undirected edge between $a$ and $b$ if and only if $\left[ \BS_L \right]_{ab} = 0$. Thus, in the case of a Brownian motion driven Ornstein-Uhlenbeck process, the  local independence graph and our conditional orthogonality graph coincide. 
\item[(c)] In both papers,  \cite{CO96} and \cite{Mogensen:Hansen:2022}, it is \linebreak important to have Brownian motion driven It$\hat{\text{o}}$ processes to receive the dependence structure of the underlying processes. Since for Gaussian models conditional orthogonality and conditional independence are equivalent, it is not surprising that we obtain the same edge characterisations as there for Gaussian driven Ornstein-Uhlenbeck processes. However, it will be a challenging task to extend the results in \cite{CO96} and \cite{Mogensen:Hansen:2022} to Lévy-driven It$\hat{\text{o}}$ processes. Our approach is able to fill this gap by presenting a graphical model for Lévy-driven MCAR$(p)$ processes that moves away from the Gaussian assumption and $p\geq 2$ but is still consistent with the existing literature and satisfies some Markov properties.
\end{itemize}
\end{remark}

Let us compare our results for the continuous-time multivariate AR process with the results for discrete-time vector AR (VAR) processes of \cite{EI07}, whose article provided the basis for our considerations. We start with the local orthogonality graph because the comparison is obvious there.

\begin{remark} \label{Comparison VAR local}
%
%to MCAR$(p)$ processes we rewrite the MCAR$(p)$ process $\CY_V$ and the VAR$(p)$ process $\mathcal{Z}_V$ accordingly, so that a comparison becomes possible. Define
%\beao
%    \Theta_j^C=e_b^\top  e^{\BA h} \BFE_{j}, \quad j=1,\ldots p
%    \quad \text{ and } \quad \varepsilon^C(t)=\int_{t}^{t+h} e^{\BA (t+h-s)} \BB L(ds)
%\eeao
%First, using \Cref{Yb für MCAR}  we rewrite the $b$-th component of the MCAR$(p)$ process
%as
%\begin{align*}
%Y_b(t+h) &= \sum_{j=1}^p e_b^\top  e^{\BA h} \BFE_{j} D^{(j-1)} %Y_V(t) + e_b^\top  \int_{t}^{t+h} e^{\BA (t+h-s)} \BB L(ds),\\
%&= \sum_{j=1}^p e_b^\top \Theta_jD^{(j-1)} Y_V(t) +e_b^\top  %\varepsilon^C(t).
%\end{align*}
%$0\leq h \leq 1$, $t\in \R$. Moreover, consider the
The $k$-dimensional VAR$(p)$ process  ${Z}_V=(Z_V(t))_{t\in\Z}$ is defined as
\begin{align} \label{VAR process}
Z_V(t+1) = \sum_{n=1}^p \Phi_n Z_V(t+1-n) + \varepsilon(t+1), \quad t\in\Z,
\end{align}
 where $\varepsilon=(\varepsilon(t))_{t\in\Z}$ is a $k$-dimensional white noise process with non-singular covariance matrix $\Sigma_{\varepsilon}\in M_k(\R)$ and   autoregressive coefficients $\Phi_n\in M_k(\R)$, $n=1,\ldots,p$.
Further, define the AR-polynomial $\Phi(\lambda)=I_k+\Phi_1\lambda+\ldots+\Phi_p\lambda^p$, $\lambda\in \C$, and denote by \texttt{B} the backshift operator. Then
\begin{align*}
    \Phi(\mbox{\texttt{B}})Z_V(t)&=\varepsilon(t),
%\end{align*}
\intertext{which corresponds to the idea
for an MCAR$(p)$ process to be the solution of the stochastic differential equation}
%\begin{equation*} \label{eq1.1}
     {P}(D)Y_V(t)&=D L(t),
\end{align*}
where $
    {P}(\lambda)= I_{k}\lambda^p+A_1\lambda^{p-1}+\ldots+A_p,$ $ \lambda \in \C$.
Let $G=(V, E)$ be the path diagram of $Z_V$ as defined in \cite{EI07}.
\begin{itemize}
\item[(a)] \textit{Directed edges:}  Lemma~2.3 and Definition 2.1  in \cite{EI07} state that the directed edges in the path diagram $G$ of the discrete-time VAR$(p)$ process ${Z}_V$ satisfy
\begin{align*}
Z_a \nrarrow Z_b \: \vert \: Z_V \quad &\Leftrightarrow \quad
 \: a \rarrow b \notin E \quad \Leftrightarrow \quad
\left[\Phi_j\right]_{ba} =0,\quad j=1,\ldots,p.
\intertext{However, this is again  in analogy to the characterisation of directed edges in the local orthogonality graph $G_{OG}^0$ of an MCAR$(p)$ processes  where}
%\begin{align*}
\CY_a \nrarrownull \CY_b \: \vert \: \CY_V \quad &\Leftrightarrow  \quad \: a \rarrow b \notin E_{OG}^0 \quad
\Leftrightarrow
\quad \left[A_j\right]_{ba}=  0, \quad \   j=1,\ldots,p.
\end{align*}
In summary, both continuous and discrete-time models have in common that there is no directed edge between components $a$ and $b$ if and only if the $ba$-th components of the autoregressive coefficients are zero.\\ $\mbox{}$ \vspace*{-0.3cm}

\item[(b)] \textit{Undirected edges:}   On the other hand, for the undirected edges in the path diagram $G$ of the VAR$(p)$ process ${Z}_V$, Lemma~2.3 and Definition 2.1 in \cite{EI07} give the equivalence
\begin{align*}
Z_a \nsim Z_b \: \vert \: Z_V \quad &\Leftrightarrow \quad \: a \inst b \notin E
\quad \quad\Leftrightarrow \quad
\left[\Sigma_{\varepsilon}\right]_{ab}=0.\intertext{However, this is again in analogy to the condition for the undirected edges in the local orthogonality graph $G_{OG}^0$ where}
%\begin{align*}
\CY_a \nsimnull \CY_b \: \vert \: \CY_V
\quad & \Leftrightarrow \quad \: a \inst b \notin E_{OG}^{0}
\quad \Leftrightarrow \quad
\left[ \BS_L \right]_{ab} =0.
\end{align*}
Thus, a common feature of the continuous-time and discrete-time model is that there is no undirected edge between components $a$ and $b$ if and only if the $a$-th  and $b$-th components of the driving process are uncorrelated.
\end{itemize}
\end{remark}

Next, we compare the path diagram of the VAR model with the orthogonality graph of the MCAR model. Before doing so, we need to give some interpretations for the orthogonality graph.

\begin{remark} \label{remark:interpretation:edges:MCAR}
For the purpose of interpretation of the directed and undirected edges in the orthogonality graph $G_{OG}$, recall from \Cref{Yb für MCAR} the representation of the component $Y_v$ of the MCAR process $Y_V$ as
\begin{align} \label{eq:MCAR:derivatives}
Y_v(t+h)
%&= \sum_{j=1}^p e_b^\top  \BFC e^{\BA h} \BFE_j D^{(j-1)} Y_V(t) + e_b^\top  \BFC\int_t^{t+h} e^{\BA (t+h-u)} \BB dL(u)\\
&=\sum_{j=1}^p e_v^\top \Theta_j^{(h)} D^{(j-1)} Y_V(t) +e_v^\top  \varepsilon^{(h)}(t), \quad v\in V,
%&= f_1(Y_V(s), s\leq t) + f_2(L(s), t\leq s \leq t+h)
\end{align}
with
\beao
    \Theta_j^{(h)}\coloneqq \BFC e^{\BA h} \BFE_{j}\in M_k(\R)
    \quad \text{ and } \quad \varepsilon^{(h)}(t)\coloneqq\int_{t}^{t+h} \BFC e^{\BA (t+h-u)} \BB dL(u) \in \R^k.
\eeao
\begin{itemize}
\item[(a)] \textit{Directed edges:} A direct application of \Cref{Erste Charakterisierung gerichtete Kante für CAR} gives the condition for the directed edges in the orthogonality graph $G_{OG}$ as
\begin{align} \label{eq:MCAR:directed_edges}
%a \rarrow b \notin E_{OG} \quad \Leftrightarrow \quad e_b^\top  \Theta_h^C(j)e_a=e_b^\top \BFC e^{\BA h} \BFE_{j}e_a= 0\quad \ \text{for }  0\leq h \leq 1,\, j=1,\ldots,p,
\CY_a \nrarrow \CY_b \: \vert \: \CY_V  \quad \Leftrightarrow \quad \left[\Theta_j^{(h)}\right]_{ba}= 0\quad \forall\, h\in [0,1],\, j=1,\ldots,p.
\end{align}
This means  that the components $Y_a(t)$, $D^{(1)}Y_a(t)$,\ldots, $D^{(p-1)}Y_a(t)$ in the representation of the $b$-th component $Y_b(t+h)$ vanish  due to the corresponding prefactors being zero. $Y_a(t)$ and its derivatives do not matter to predict $Y_b(t+h)$.
\item[(b)] \textit{Undirected edges:} A consequence of \Cref{Einfache Charakterisierung ungerichtete Kante für CAR} is the condition for the undirected edges in the orthogonality graph $G_{OG}$ as
\begin{align}  \label{eq:MCAR:indirected_edges}
\CY_a \nsim \CY_b \: \vert \: \CY_V
\quad \Leftrightarrow \quad
\left[  \BE[\varepsilon^{(h)}(t)\varepsilon^{(\tilde{h})}(t)^\top ] \right]_{ab}=
\left[  \BE[\varepsilon^{(h)}(0)\varepsilon^{(\tilde{h})}(0)^\top ] \right]_{ab} =0 \quad \forall\, h,\tilde{h}\in [0,1],
\end{align}
%\begin{align*}
%a \inst b \notin E_{OG}
%\quad \Leftrightarrow \quad
%e_a^\top  \BFC \int_0^{\min(h,\tilde{h})} e^{\BA (h-s)} \BB \BS_L \BB^\top  e^{\BA^\top (\tilde{h}-s)} ds \: \BFC^\top  e_b %=0 \quad  \text{for} \quad 0\leq h, \tilde{h}\leq 1.
%\end{align*}
i.e., the noise terms $e_a^\top   \varepsilon^{(h)}(t)$ and $e_b^\top   \varepsilon^{(\tilde h)}(t)$
of $Y_a(t+h)$  and $Y_b(t+\tilde h)$ are uncorrelated for any $t\geq 0$.
\end{itemize}
\end{remark}

%Now, we are able to compare the orthogonality graph of the VAR model with the orthogonality graph of the MCAR model.

\begin{remark}\label{Remark:Interpretation gerichtete kante}
The characterisations of the directed and undirected edges of the orthogonality graph in \Cref{remark:interpretation:edges:MCAR} are well suited for comparison with VAR$(p)$ processes in \cite{EI07}. The challenge here is that in representation \eqref{eq:MCAR:derivatives}  of $Y_V(t+h)$ appear 
derivatives which have to be related to appropriate differences in the discrete-time process \eqref{VAR process}.  Thus, our goal is to replace the backshifts \mbox{$Z_V(t+1-n)$,} $n=1,\ldots,p$, by appropriate differences. To do this, we define a discrete-time difference operator
iteratively by
\begin{align*}
\texttt{D}^{(1)} Z_V(t) = Z_V(t) - Z_V(t-1), \quad
\texttt{D}^{(j)} Z_V(t) = \texttt{D}^{(j-1)} \left( Z_V(t) - Z_V(t-1) \right),
\end{align*}
$j=1,\ldots,p-1$, where we set $\texttt{D}^{(0)} Z_V(t) = Z_V(t)$. Furthermore, define
\begin{align*}
    \Theta_j\coloneqq\sum_{n=j}^p   \binom{n-1}{j-1} (-1)^{j-1} \Phi_n, \quad j=1,\ldots,p.
\end{align*}
Then some direct calculations show (see the Supplementary Material \ref{suppl:section 6}) that
\begin{align} \label{eq:VAR_differences}
Z_b(t+1) % &= \sum_{j=1}^{p} \sum_{n=j}^p   \binom{n-1}{j-1} (-1)^{j-1} e_b^\top  \Phi(n) \mathcal{D}^{(j-1)} Z_V(t) + e_b^\top  \varepsilon(t+1)\\
&=\sum_{j=1}^{p} e_b^\top \Theta_j \texttt{D}^{(j-1)} Z_V(t) + e_b^\top  \varepsilon(t+1).
\end{align}
This representation is now in analogy to \eqref{eq:MCAR:derivatives} for MCAR$(p)$ processes.

%If we now compare the $b$-th component of the MCAR$(p)$ process and the VAR$(p)$ process, we recognise the analogy between the weight vectors of the $(j-1)$-th derivative respectively difference
%\begin{align*}
%e_b^\top  e^{\BA h} \BFE_{j} \quad \text{and} \quad \sum_{n=j}^p   \binom{n-1}{j-1} (-1)^{j-1} e_b^\top  \Phi(n)
%\end{align*}
%as well as the analogy between the noise terms
%\begin{align*}
%e_b^\top  \int_{t}^{t+h} e^{\BA (t+h-s)} \BB L(ds) \quad \text{and} \quad e_b^\top  \varepsilon(t+1).
%\end{align*}
\begin{itemize}
\item[(a)] \textit{Directed edges:} In the former \Cref{Comparison VAR local} we just saw that for the discrete-time VAR$(p)$  process ${Z}_V$ the directed edges in the path diagram $G$ satisfy
\begin{align*}
Z_a \nrarrow Z_b \: \vert \: Z_V
\quad \Leftrightarrow \quad
a \rarrow b \notin E
\quad \Leftrightarrow \quad
\left[\Phi_j\right]_{ba} =0,\quad j=1,\ldots,p.
\end{align*}
But
\begin{align*}
 & \left[\Phi_j \right]_{ba}=0,\quad j=1,\ldots,p \\
 &\quad \Leftrightarrow \quad    \left[ \Theta_j\right]_{ba}=\sum_{n=j}^p \binom{n-1}{j-1} (-1)^{j-1} \left[\Phi_n\right]_{ba}= 0, \quad j=1,\ldots,p.
\end{align*}
%such that  we have the equivalence
%\begin{align*}
% Z_a \nrarrow Z_b \: \vert \: Z_V \quad \Leftrightarrow \quad a \rarrow b \notin E \quad
%\Leftrightarrow \quad
% \left[\Theta_j\right]_{ba} =0, \quad j=1,\ldots,p.
%\end{align*}
However, this is again  analogous to the characterisation of directed edges in the orthogonality graph $G_{OG}$ for the MCAR$(p)$ process in  \eqref{eq:MCAR:directed_edges} where
\begin{align*}
\CY_a \nrarrow \CY_b \: \vert \: \CY_V \quad \Leftrightarrow
 \quad a \rarrow b \notin E \quad
\Leftrightarrow \quad
\left[\Theta_j^{(h)}\right]_{ba}=  0 \quad \   \forall\, h \in [0,1], \, j=1,\ldots,p .
\end{align*}
\item[(b)] \textit{Undirected edges:} For the path diagram $G$ for the VAR$(p)$ process ${Z}_V$ we have
\begin{align*}
Z_a \nsim Z_b \: \vert \: Z_V \quad \Leftrightarrow  \quad a \inst b \notin E \quad
\Leftrightarrow \quad \left[\BE [\varepsilon(0)\varepsilon(0)^\top ]\right]_{ab}=0.
\end{align*}
 %That is the noise terms $e_a^\top  \varepsilon(t+1)$ and $e_b^\top  \varepsilon(t+1)$ are uncorrelated.
Here we have the similarity to the condition \eqref{eq:MCAR:indirected_edges} for the undirected edges of the MCAR$(p)$ in the orthogonality graph $G_{OG}$
\begin{align*}
\CY_a \nsim \CY_b \: \vert \: \CY_V
\quad \Leftrightarrow \quad \: a \inst b \notin E_{OG}  \quad
\Leftrightarrow \quad
\left[  \BE[\varepsilon^{(h)}(0)\varepsilon^{(\tilde h)}(0)^\top ] \right]_{ab} =0 \quad  \forall\,h, \tilde{h}\in [0,1].
\end{align*}
\end{itemize}
\end{remark}

Since a continuous-time Ornstein-Uhlenbeck process sampled at discrete equidistant time points  is a discrete-time VAR(1) process, we study the results for an Ornstein-Uhlenbeck process in more detail and, in particular, relate them to the results for VAR models in \cite{EI07}.

\begin{remark}
Let $\CY_V$ be a causal %$k$-dimensional
Ornstein-Uhlenbeck process as given in \Cref{Example:Ornstein-Uhlenbeck process} with $\BS_L>0$.
%where the driving Lévy process has a positive definite covariance matrix $\BS_L$.
% Suppose $G_{OG}=(V,E_{OG})$ is the orthogonality graph for the causal $k$-dimensional Ornstein-Uhlenbeck  process $\CY_V$ and let $a,b\in V$, $a \neq b$.
 %Here, the characterisation of the edges in the continuous-time model reduces to:
%\begin{itemize}
%     \item $a \rarrow b \notin E_{OG} \quad \Leftrightarrow \quad \left[e^{Ah} \right]_{ba}=0$ for all $0\leq h\leq1$.
 %   \item $a \inst b \notin E_{OG}
%\quad  \Leftrightarrow \quad
%\left[  \BE[\varepsilon_h(0)\varepsilon_{\tilde h}(0)^\top ] \right]_{ab}=\left[ \int_0^{\min(h,\tilde{h})} e^{A (h-s)} \BS_L e^{A^\top (\tilde{h}-s)} ds \right]_{ab} =0$ \\
%\hspace*{5cm} for all
%$ 0\leq h, \tilde{h}\leq 1.$
%\end{itemize}
Then the continuous-time process $\CY_V$ sampled at discrete-time points of distance $h$ is a discrete-time VAR(1) process with representation
\begin{align*}
Y_V((k+1)h)&=e^{\BA h}Y_V(kh)+\int_{kh}^{(k+1)h}e^{\BA ((k+1)h-u)}\,dL(u)\\
    &=e^{\BA h} Y_V(kh)+\varepsilon^{(h)}(kh), \quad k\in\Z,
\end{align*}
 %where
 %\beao
 %   \Phi=e^{Ah} \quad &\text{ and }& \quad \varepsilon_h(hk) =\int_{kh}^{(k+1)h}e^{A((k+1)h-s)}\,L(ds).
 %\eeao
 which we denote by  $\CY_V^{(h)}=(Y_V((k+1)h))_{k\in \Z}$ and the corresponding discrete-time path diagram by $G^{(h)}=(V,E^{(h)})$.
 Then a direct conclusion of \Cref{Comparison VAR local} is that for $a,b\in V$ and $a \neq b$:
 \begin{itemize}
     \item[(a)] \makebox[2,3cm][l]{$\CY_a \nrarrow \CY_b \: \vert \: \CY_V$}
       \makebox[5,5cm][l]{$ \Rightarrow \quad \left[e^{\BA h} \right]_{ba}=0$}
       $ \Rightarrow \quad \CY_a^{(h)} \nrarrow \CY_b^{(h)} \: \vert \: \CY_V^{(h)}.$
    \item[(b)] \makebox[2,28cm][l]{$\CY_a \nsim \CY_b \: \vert \: \CY_V$}
        \makebox[5,5cm][l]{$\Rightarrow \quad \left[  \BE[\varepsilon^{(h)}(0)\varepsilon^{ (h)}(0)^\top ] \right]_{ab}=0$}
        $\Rightarrow$ \quad $ \CY_a^{(h)} \nsim \CY_b^{(h)} \: \vert \: \CY_V^{(h)}$.
\end{itemize}
%\marginpar{\LS{Warum keine Äquivalenz?}\VF{Man hat ja nur ein festes $h$. Da kann man Gegenbeispiele konstruieren}  \LS{Da das $e^{Ah}$ für alle $h$ null ist, dachte ich auch rechts hat man die Nichtkausalität für alle $h$.}\VF{Normalerweise beobachtet man den Prozess für ein festes $h$ gesampled, nicht für alle $h$}}
This means that a directed (undirected) edge $a \rarrow b \in E^{(h)}$ ($a \inst b \in E^{(h)}$) in the discrete-time model $\CY_V^{(h)}$ implies a (undirected) directed edge $a \rarrow b \in E_{OG}$ ($a \inst b \in E_{OG}$) in the continuous-time model $\CY_V$. In summary, $E^{(h)}\subseteq E_{OG}$ for every $h \in [0,1]$. We believe that this result may hold for general MCAR$(p)$ processes.
%However, the opposite direction does in general not hold. If for every $0\leq h\leq 1$,
% $a \rarrow b \notin E^{(h)}$ ($a \inst b \notin E^{(h)}$) we can in general not say anything for the linear orthogonality graph $G_{OG}=(V, E_{OG})$ of $Y_V$. Thus, most of the time $E^{(h)}$ is a real subset of $ E_{OG}$.
%there is no indirect influence between $a$ and $b$ there still might be an indirect influence between $a$ and $b$ in the continuous-time process $\CY_V$.
This phenomenon is an advantage of the orthogonality graph over the local orthogonality graph, where there is generally no relationship between the edges $E_{OG}^{(0)}$ and $E^{(h)}$.
\end{remark}

The characterisation of the directed edges in \Cref{Erste Charakterisierung gerichtete Kante für CAR} and the characterisation of the undirected edges in \Cref{Einfache Charakterisierung ungerichtete Kante für CAR} are nice for interpretation, but depend on the lags $h, \tilde{h}$. We provide simpler necessary and sufficient criteria for the directed and undirected edges, respectively, where the lags $h, \tilde{h}$ no longer play a role. 

\begin{theorem}\label{Zweite Charakterisierung gerichtete Kante für CAR}\label{Erste Charakterisierung ungerichtete Kante für CAR}
%Suppose $G_{OG}=(V,E_{OG})$ is the orthogonality graph for the causal $k$-dimensional MCAR$(p)$ process $\CY_V$ and let $a,b\in V$, $a \neq b$. Then the following hold:
Let $\CY_V$ be a causal %$k$-dimensional
MCAR$(p)$ process with
%\LS{such that the driving $k$-dimensional Lévy process satisfies \Cref{Assumption on Levy process}?}. Suppose
\mbox{$\BS_L>0$.} Further, let $a,b\in V$, $a \neq b$. Then the following holds.
\begin{itemize}
\item[(a)]
\makebox[2,15cm][l]{$\CY_a \nrarrow \CY_b \: \vert \: \CY_V$}
$\Leftrightarrow \:\,
\left[ \BFC \BA^\alpha \BFE_j \right]_{ba} = \left[ \BA^\alpha \right]_{b, k(j-1)+a}=0, \quad \alpha=1,\ldots,kp-1, \, j=1,\ldots,p.$
%\begin{align*}
%a \rarrow b \notin E_{OG} \quad \Leftrightarrow \quad
%e_b \BFC \BA^\alpha \BFE_j e_a^\top  = e_b \BA^\alpha e_{k(j-1)+a}=0 \quad \text{for} \quad \alpha=1,\ldots,kp-1, j=1,\ldots,p.
%\end{align*}
\item[(b)] \makebox[2,13cm][l]{$\CY_a \nsim \CY_b \: \vert \: \CY_V $}
$\Leftrightarrow \: \left[\BFC \BA^\alpha \BB \BS_L \BB^\top  \left(\BA^\top  \right)^\beta \BFC^\top  \right]_{ab}=0, \quad  \alpha,\beta=0,\ldots,kp-1.$
\end{itemize}
\end{theorem}

%\begin{remark}
\begin{remark} \label{Remark:step size}
 The proof of \Cref{Erste Charakterisierung ungerichtete Kante für CAR} shows that in the definition of Granger causality and contemporaneous correlation the choice of the step size $h$ as defined in \Cref{Remark 3.2} (cf.~\eqref{eq:Remark:3.2}) and \Cref{Remark 4.2} (cf.~\eqref{eqref:Remark 4.2}), respectively, has no influence on the final characterisations of the edges in the MCAR model. For any choice $h>0$ we obtain the characterisations as in \Cref{Zweite Charakterisierung gerichtete Kante für CAR}.  In particular, it follows that Granger causality and global Granger causality as well as contemporaneous correlation and global contemporaneous correlation are equivalent for MCAR$(p)$ processes, and hence the global orthogonality graph also satisfies the different Markov properties. 
\end{remark}

We obtain the following direct conclusion from Propositions \ref{Erste Charakterisierung gerichtete Kante für CAR}, \ref{Einfache Charakterisierung ungerichtete Kante für CAR} and \Cref{Zweite Charakterisierung gerichtete Kante für CAR},  setting $\alpha=p$  in \Cref{Zweite Charakterisierung gerichtete Kante für CAR} (a) and $\alpha=\beta=p-1$  in \Cref{Zweite Charakterisierung gerichtete Kante für CAR} (b).

%\LS{Wollen: $\BFE_p^\top  \BA \BFE_j = \BFC \BA^\alpha \BFE_j$ wobei $\BFC = \BFE_1^\top  \neq \BFE_p$. $\BFC \BA = (0_k I_k 0_k \cdots 0_k)$ und $\BFC \BA \BFE_j = (0_k I_k 0_k \cdots 0_k) \BFE_j = 0_k$ bzw. $I_k$ für $j=2$. Durch das Potenzieren kommt die Zeile mit den $A_j$ in $\BA$ immer weiter nach oben. $\BFC \BA^p = (-A_p ... -A_1)$
%und $\BFC \BA^p \BFE_j = (-A_p ... -A_1) \BFE_j = -A_{k-j} = - \BFE_p^\top  \BA \BFE_{k-j}$. Genauso war die Idee beim ungerichteten, nur will man da $I_k = \BFC \BA^\alpha \BB$ und $\BFC \BA^{p-1}=(0_k...0_k I_k)$.}
%\VF{Aus irgendeinem Grund hatte ich gedacht, das nächste Resultat wäre nur für OU Prozesse und $p=1$. Dann ist es schon klar.}

\begin{corollary}
Let $\CY_V$ be a causal %$k$-dimensional
MCAR$(p)$ process with %\LS{such that the driving $k$-dimensional Lévy process satisfies \Cref{Assumption on Levy process}?} with
\mbox{$\BS_L>0$}, orthogonality graph $G_{OG}=(V,E_{OG})$, and local orthogonality graph $G_{OG}=(V,E_{OG})$. Then $E_{OG}^{(0)} \subseteq E_{OG},$ and in general the sets are not equal.
\end{corollary}

In particular, in the case of an Ornstein-Uhlenbeck process, the characterisation of the edges in an orthogonality graph can be reduced to the following.

\begin{comment}
Completely analogous one can show that
\begin{align*}
\CY_a \stackrel{\infty}{\nsim} \CY_b \: \vert \: \CY_V \quad \Leftrightarrow \quad \left[ \BFC \BA^\alpha \BB \BS_L \BB^\top  \left(\BA^\top  \right)^\beta \BFC^\top  \right]_{ab}=0 \quad  \text{for} \quad \alpha,\beta=0,\ldots kp-1.
\end{align*}
Thus there is no difference between contemporaneously uncorrelated component series and contemporaneously uncorrelated component series at any horizon for an MCAR$(p)$ process. Let us continue \Cref{Remark: Interpretation gerichtete kante} and give an interpretation of the characterisations.
\end{comment}

%\LS{Frage: Doch erst noch die zweite Charakterisierung angeben? So ist für mich die Reihenfolge nicht schlüssig. Vorschlag: Gerichtete Kante 1, Gerichtete Kante 2 mit Kommentar zu long run, Ungerichtete Kante 1, Ungerichtete Kante 2 mit Kommentar zu long run, Interpretation von beidem? Oder Interpretation dazwischen? Dann hört das Ganze auch nicht auf p.30 mit einem Lemma so aprupt auf.}
%\VF{Das macht Sinn. Ich hätte die Tendenz die Interpretation dazwischen zu machen.}

\begin{corollary} \label{Corollary:OU processes}
Let $\CY_V$ be a causal %$k$-dimensional
Ornstein-Uhlenbeck process with \mbox{$\BS_L>0$.}
%\LS{such that the driving $k$-dimensional Lévy process satisfies \Cref{Assumption on Levy process}?}. Suppose $\BS_L$ is positive definite and
Further, let $a,b\in V$, $a \neq b$. Then the following holds.
%Suppose $G_{OG}=(V,E_{OG})$ is the orthogonality graph for the causal $k$-dimensional Ornstein-Uhlenbeck process $\CY_V$ and let $a,b\in V$, $a \neq b$. Then the following hold:
\begin{itemize}
\item[(a)] \makebox[2,45cm][l]{$ \CY_a \nrarrow \CY_b \: \vert \: \CY_V$}
%\quad &\Leftrightarrow \quad \left[ e^{\BA h} \right]_{ba}=0 \quad \text{for} \quad 0 \leq h \leq 1, \\
\makebox[5cm][l]{$\Leftrightarrow \quad \left[ \BA^\alpha \right]_{ba}=0,$}
$\alpha=1,\ldots,k-1.$
\item[(b)]  \makebox[2,43cm][l]{$\CY_a \nsim \CY_b \: \vert \: \CY_V$}
%\quad & \Leftrightarrow \quad  \left[ \int_0^{\min(h,\tilde{h})} e^{\BA (h-s)} \BS_L e^{\BA^\top (\tilde{h}-s)} ds \right]_{ab} =0 \quad  \text{for} \quad 0\leq h, \tilde{h}\leq 1, \\
\makebox[5cm][l]{$\Leftrightarrow \quad \left[ \BA^\alpha \BS_L  \left(\BA^\top  \right)^\beta \right]_{ab}=0,$}
$\alpha,\beta=0,\ldots,k-1.$
\end{itemize}
\end{corollary}

\begin{remark}
Suppose $\BS_L$ is a diagonal matrix and $\CY_V$ is a causal Ornstein-Uhlenbeck process. Then \Cref{Corollary:OU processes}  implies that from $\CY_a \nsim \CY_b \: \vert \: \CY_V$ directly  follows $ \CY_a \nrarrow \CY_b \: \vert \: \CY_V$. Thus, a directed edge in such an orthogonality graph of an Ornstein-Uhlenbeck process induces an undirected edge.
\end{remark}

 %, which is an advantage of the local orthogonality graph.
%Both graphical models have their advantages and disadvantages. For an adequate investigation of the directed and undirected relationships between the components of a multivariate continuous-time process, it is, therefore, advisable to consider both the orthogonality graph and the local orthogonality graph.

\section{Conclusion}
In this paper, we have introduced concepts of directed and undirected influences for stochastic processes in continuous time, defined (local) orthogonality graphs, discussed their properties, and applied them to MCAR processes. The main results are as follows:
\begin{itemize}
    \item[(a)] (Local) orthogonality graphs provide a simple visualisation and a concise way to communicate directed and undirected (local) conditional orthogonality structures of the process.
    \item[(b)] (Local) orthogonality graphs are defined using the pairwise Markov property to represent the pairwise relationships between variables. However, the associated orthogonality graph can be interpreted using the global AMP Markov and the global Markov property. In this way, new Granger non-causality relations and contemporaneous uncorrelations between subprocesses can be obtained.
    \item[(c)] For MCAR models the (local) orthogonality graphs are closely related to the moving average parameters and the covariance matrix of the driving Lévy process.     Any local orthogonality graph can be constructed by an MCAR model, but this is generally not true for an orthogonality graph. However, if there is no edge in the orthogonality graph, then there is no edge in the discrete-time sampled model.
\end{itemize}

\appendix
%%%%%%%%%%%%%%%%%%%%%%%%%%%%%%%%%%%%%%%%%%%%%%%%%%%%%%%%%%%%%%%%%%%%%%%%%%%%%%%%%%%%%%%%%%%%%%%%%%
\section{Proofs of Section~\ref{sec:influence}} \label{sec:Proof:Granger_non-causal:project}
%%%%%%%%%%%%%%%%%%%%%%%%%%%%%%%%%%%%%%%%%%%%%%%%%%%%%%%%%%%%%%%%%%%%%%%%%%%%%%%%%%%%%%%%%%%%%%%%%%

%%%%%%%%%%%%%%%%%%%%%%%%%%%%%%%%%%%%%%%%%%%%%%%%%%%%%%%%%%%%%%%%%%%%%%%%%%%%%%%%%%%%%%%%%%%%%%%%%%
%\subsection{Proofs of  Section~\ref{sec:influence}}
%%%%%%%%%%%%%%%%%%%%%%%%%%%%%%%%%%%%%%%%%%%%%%%%%%%%%%%%%%%%%%%%%%%%%%%%%%%%%%%%%%%%%%%%%%%%%%%%%%
%
\begin{comment}
\begin{remark} \label{Remark:B}
We point out, that we could just rewrite the first linear space again. Then $\CY_A$ is Granger non-causal for $\CY_B$ with respect to $\CY_S$, if and only if
\begin{align*}
L_{Y_b}(t+h) \perp \mathcal{L}_{Y_A}(t) \: \vert \: \mathcal{L}_{Y_{S \setminus A}}(t),
\end{align*}
for $b \in B$, $0 \leq h \leq 1$,  $t \in \R$. Since the orthogonal projection and the expectation are linear, we obtain that $\CY_A$ is Granger non-causal for $\CY_B$ with respect to $\CY_S$, if and only if
\begin{align*}
L_{Y_B}(t+h) \perp \mathcal{L}_{Y_A}(t) \: \vert \: \mathcal{L}_{Y_{S \setminus A}}(t),
\end{align*}
for $0 \leq h \leq 1$,  $t \in \R$. This characterisation would be another natural way to define Granger non-causality as a direct generalisation from \cite{EI07}.
\end{remark}
\end{comment}

\begin{proof}[Proof of \Cref{Charakterisierung als Gleichheit der linearen Vorhersage}]
Due to \cite{LI15}, Proposition 2.4.2, $\mathcal{L}_{Y_B}(t,t+1) \perp \mathcal{L}_{Y_A}(t) \: \vert \: \mathcal{L}_{Y_{S \setminus A}}(t)$ is equivalent to $P_{\mathcal{L}_{Y_S}(t)}Y^B = P_{\mathcal{L}_{Y_{S \setminus A}}(t)}Y^B$ $\mathbb{P}$-a.s.  for all $Y^B\in \mathcal{L}_{Y_B}(t,t+1)$. Due to the linearity and continuity of orthogonal projections, this is in turn equivalent to
$P_{\mathcal{L}_{Y_S}(t)}Y_b(t+h) = P_{\mathcal{L}_{Y_{S \setminus A}}(t)}Y_b(t+h)$ $\mathbb{P}$-a.s. for all $h \in [0,1]$, $t\in \R$ and $b\in B$.
\end{proof}

\begin{proof}[Proof of \Cref{Charakterisierung local Granger non-causal}]
First assume that $\CY_A \nrarrownull \CY_B \: \vert \: \CY_S$, i.e., $\mathbb{P}$-a.s.
\begin{align}\label{eq:hilfseq1}
&\limh P_{\mathcal{L}_{Y_S}(t)} \left(\frac{D^{(j_b)} Y_b(t+h)- D^{(j_b)} Y_b(t)}{h}\right) \nonumber\\
&\quad =\limh P_{\mathcal{L}_{Y_{S \setminus A}}(t)} \left(\frac{D^{(j_b)} Y_b(t+h)- D^{(j_b)} Y_b(t)}{h}\right),
\end{align}
for all $t\in \R$ and $b\in B$. Now let $Y^A \in \mathcal{L}_{Y_{A}}(t)$, $b\in B$, and $t\in\R$. Then
as well $Y^A \in \mathcal{L}_{Y_{S}}(t)$ and $D^{(j_b)} Y_b(t+h) - P_{\mathcal{L}_{Y_{S}}(t)} D^{(j_b)} Y_b(t+h) \in \mathcal{L}_{Y_{S}}(t)^\perp$, so
%Since $D^{(j_b)} Y_b(t+h) - P_{\mathcal{L}_{Y_{S}}(t)} D^{(j_b)} Y_b(t+h) \in \mathcal{L}_{Y_{S}}(t)^\perp$ and $Y^A \in \mathcal{L}_{Y_{S}}(t)$
\begin{align*}
\frac{1}{h} \BE \left[
\left(D^{(j_b)} Y_b(t+h) - P_{\mathcal{L}_{Y_{S}}(t)} D^{(j_b)} Y_b(t+h)\right)
\overline{Y^A} \right]=0.
\end{align*}
Adding and subtracting $P_{\mathcal{L}_{Y_{S \setminus A}}(t)} D^{(j_b)} Y_b(t+h)$ in the first factor and then forming the limit gives
\begin{align} \label{A2*}
&\lim_{h \rightarrow 0} \frac{1}{h} \BE \left[
\left(D^{(j_b)} Y_b(t+h) - P_{\mathcal{L}_{Y_{S \setminus A}}(t)} D^{(j_b)} Y_b(t+h)\right)
\overline{Y^A} \right] \\
&+ \lim_{h \rightarrow 0} \frac{1}{h} \BE \left[
\left(P_{\mathcal{L}_{Y_{S \setminus A}}(t)} D^{(j_b)} Y_b(t+h) - P_{\mathcal{L}_{Y_{S}}(t)} D^{(j_b)} Y_b(t+h)\right)
\overline{Y^A} \right] =0 . \nonumber
\end{align}
Due to \Cref{Remark 3.6} and $A\cap B=\emptyset$, we already know that
$D^{(j_b)} Y_b(t) \in \mathcal{L}_{Y_{S \setminus A}}(t) \subseteq \mathcal{L}_{Y_{S}}(t)$. Then it follows together with \eqref{eq:hilfseq1} and  \eqref{limit E} that the second summand in \eqref{A2*} is zero and thus, the first summand is zero as well, i.e.,
\begin{align*}
\lim_{h \rightarrow 0} \frac{1}{h} \BE \left[
\left(D^{(j_b)} Y_b(t+h) - P_{\mathcal{L}_{Y_{S \setminus A}}(t)} D^{(j_b)} Y_b(t+h)\right)
\overline{Y^A} \right]=0.
\end{align*}
%Due to \Cref{Remark 3.6} and $A\cap B=\emptyset$, we already know that
%$D^{(j_b)} Y_b(t) \in \mathcal{L}_{Y_{S \setminus A}}(t)$. Then it follows together with \eqref{eq:hilfseq1} that
%\begin{align*}
%&\lim_{h \rightarrow 0} \frac{1}{h} \BE \left[
%\left(D^{(j_b)} Y_b(t+h) - P_{\mathcal{L}_{Y_{S \setminus A}}(t)} D^{(j_b)} Y_b(t+h)\right)
%Y^A \right]\\
%&\quad = \lim_{h \rightarrow 0} \BE \left[
%\left(\frac{D^{(j_b)} Y_b(t+h)- D^{(j_b)} Y_b(t)}{h} - P_{\mathcal{L}_{Y_{S \setminus A}}(t)} \frac{D^{(j_b)} %Y_b(t+h)- D^{(j_b)} Y_b(t)}{h}\right)
%Y^A \right]\\
%&\quad = \lim_{h \rightarrow 0} \BE \left[
%\left(\frac{D^{(j_b)} Y_b(t+h)- D^{(j_b)} Y_b(t)}{h} - P_{\mathcal{L}_{Y_{S}}(t)} \frac{D^{(j_b)} Y_b(t+h)- D^{(j_b)} %Y_b(t)}{h}\right)
%Y^A \right]\\
%&\quad = \lim_{h \rightarrow 0} \frac{1}{h} \BE \left[
%\left(D^{(j_b)} Y_b(t+h) - P_{\mathcal{L}_{Y_{S}}(t)} D^{(j_b)} Y_b(t+h)\right)
%Y^A \right]
%\end{align*}
%for all $b\in B$. Since $D^{(j_b)} Y_b(t+h) - P_{\mathcal{L}_{Y_{S}}(t)} D^{(j_b)} Y_b(t+h) \in %\mathcal{L}_{Y_{S}}(t)^\perp$ and $Y^A \in \mathcal{L}_{Y_{S}}(t)$ the right hand side is zero.
Further, $D^{(j_b)} Y_b(t+h) - P_{\mathcal{L}_{Y_{S\setminus A}}(t)} D^{(j_b)} Y_b(t+h) \in \mathcal{L}_{Y_{S \setminus A}}(t)^\perp$ and $P_{\mathcal{L}_{Y_{S \setminus A}}(t)} Y^A  \in \mathcal{L}_{Y_{S \setminus A}}(t)$ give
\begin{align*}
\frac{1}{h} \BE \left[ \left(D^{(j_b)} Y_b(t+h) - P_{\mathcal{L}_{Y_{S \setminus A}}(t)} D^{(j_b)} Y_b(t+h)\right) \overline{P_{\mathcal{L}_{Y_{S \setminus A}}(t)} Y^A} \right] =0.
\end{align*}
Adding the limit, the last two equations yield as claimed
\begin{equation*}
\lim_{h \rightarrow 0} \frac{1}{h} \BE \left[
\left(D^{(j_b)} Y_b(t+h) - P_{\mathcal{L}_{Y_{S \setminus A}}(t)} D^{(j_b)} Y_b(t+h)\right)
\overline{\left(Y^A - P_{\mathcal{L}_{Y_{S \setminus A}}(t)} Y^A \right)}
\right]=0. \qedhere
\end{equation*}
\end{proof}

\begin{proof}[Proof of \Cref{Lemma:Granger_causality_relations}]
$\mbox{}$\\
(a) This is obvious by definitions.\\
(b) The implication $\Rightarrow$ follows instantly. For the proof of $\Leftarrow$ we use mathematical induction and show that
\begin{align} \label{IV}
\mathcal{L}_{Y_{S\setminus A}}(t+k) &\perp \mathcal{L}_{Y_A}(t) \: \vert \: \mathcal{L}_{Y_{S\setminus A}}(t) \quad \forall \: t \in \R, \: k\in \N.
\end{align}
First, we note that $\CY_A  \nrarrow \CY_{S\setminus A} \: \vert \: \CY_S$ and \Cref{Charakterisisierung linear granger non-causal} (b) yield the initial case
\begin{align} \label{IA}
\mathcal{L}_{Y_{S\setminus A}}(t+1) &\perp \mathcal{L}_{Y_A}(t) \: \vert \: \mathcal{L}_{Y_{S \setminus A}}(t) \quad \forall \: t \in \R.
\end{align}
Now, replacing $t$ by $t+1$ in the induction hypothesis gives
%\begin{align*}
%\mathcal{L}_{Y_{S\setminus A}}(t+k) &\perp \mathcal{L}_{Y_A}(t) \: \vert \: \mathcal{L}_{Y_{S \setminus A}}(t), \quad \forall \: t \in \R.
%\end{align*}
%We can replace $t$ by $t+1$ and obtain
\begin{align*}
\mathcal{L}_{Y_{S\setminus A}}(t+k+1) &\perp \mathcal{L}_{Y_A}(t+1) \: \vert \: \mathcal{L}_{Y_{S \setminus A}}(t+1) \quad \forall \: t \in \R.
\end{align*}
 Since by \Cref{additivity in time domain} we have $\mathcal{L}_{Y_A}(t+1)= \mathcal{L}_{Y_A}(t) \vee \mathcal{L}_{Y_A}(t,t+1)$, the property of decomposition (C2) from \Cref{properties of conditional orthogonality} implies
\begin{align*}
\mathcal{L}_{Y_{S\setminus A}}(t+k+1) &\perp \mathcal{L}_{Y_A}(t) \: \vert \: \mathcal{L}_{Y_{S \setminus A}}(t+1) \quad \forall \: t \in \R,
\end{align*}
which is by \Cref{additivity in time domain} again
\begin{align*}
\mathcal{L}_{Y_{S\setminus A}}(t+k+1) &\perp \mathcal{L}_{Y_A}(t) \: \vert \: \mathcal{L}_{Y_{S \setminus A}}(t) \vee \mathcal{L}_{Y_{S \setminus A}}(t,t+1) \quad \forall \: t \in \R.
\end{align*}
This result together with the initial case \eqref{IA} and the properties of decomposition (C2) and contraction (C4) from \Cref{properties of conditional orthogonality} yield
\begin{align*}
\mathcal{L}_{Y_{S\setminus A}}(t+k+1) \vee \mathcal{L}_{Y_{S\setminus A}}(t,t+1) \perp \mathcal{L}_{Y_A}(t) \: \vert \: \mathcal{L}_{Y_{S \setminus A}}(t) \quad \forall \: t \in \R.
\end{align*}
Finally, the property of decomposition (C2) gives the induction step
\begin{align*}
\mathcal{L}_{Y_{S\setminus A}}(t+k+1) &\perp \mathcal{L}_{Y_A}(t) \: \vert \: \mathcal{L}_{Y_{S \setminus A}}(t) \quad \forall \: t \in \R.
\end{align*}
To bring the proof to an end, let $\lceil \cdot \rceil$ be the ceiling function. Then $\mathcal{L}_{Y_{S\setminus A}}(t+h) \subseteq \mathcal{L}_{Y_{S\setminus A}}(t+\lceil h \rceil)$. Now it follows from \eqref{IV} and the decomposition property (C2) that

\begin{align*}
\mathcal{L}_{Y_{S\setminus A}}(t+h) &\perp \mathcal{L}_{Y_A}(t) \: \vert \: \mathcal{L}_{Y_{S \setminus A}}(t) \quad \forall \: t \in \R,\,h\geq 0.
\end{align*}
(c) This follows directly due to (b), the decomposition property (C2), and $B\subseteq S\setminus A$.\\
%(d)
%Let $\CY_A  \nrarrow \CY_B \: \vert \: \CY_S$. Then by \Cref{Charakterisierung als Gleichheit der linearen Vorhersage} for all $h \in [0,1]$, $t\in \R$ and $b\in B$,
%\begin{align*}
%P_{\mathcal{L}_{Y_S}(t)}Y_b(t+h) = P_{\mathcal{L}_{Y_{S \setminus A}}(t)}Y_b(t+h) \quad \mathbb{P}\text{-a.s.}
%\end{align*}
%In particular, due to $Y_b(t) \in \mathcal{L}_{Y_{S \setminus A}}(t)$, we have for all $h \in [0,1]$, $t\in \R$ and $b\in B\subseteq S\backslash A$,
%\begin{align*}
%P_{\mathcal{L}_{Y_S}(t)} \frac{Y_b(t+h)- Y_b(t)}{h}
%= P_{\mathcal{L}_{Y_{S \setminus A}}(t)} \frac{Y_b(t+h)- Y_b(t)}{h} \quad %\mathbb{P}\text{-a.s.}
%\end{align*}
%For $h\rightarrow 0$ we receive \LS{Hier past es nicht mehr, wie soll Ableitung an der Stelle t+h rauskommen.}
%\begin{align*}
%P_{\mathcal{L}_{Y_S}(t)} D^{(1)}Y_b(t+h)
%= P_{\mathcal{L}_{Y_{S \setminus A}}(t)} D^{(1)}Y_b(t+h).
%\end{align*}
% Since  $D^{(i)}Y_b(t)\in \mathcal{L}_{Y_{S\setminus A}}(t)$ for $i=1,\ldots,j_b$,  %due to \Cref{Remark 3.6}, we can show in the same way by induction that
%\begin{align*}
%P_{\mathcal{L}_{Y_S}(t)} D^{(i)}Y_b(t+h)
%= P_{\mathcal{L}_{Y_{S \setminus A}}(t)} D^{(i)}Y_b(t+h),
%\quad i=1,\ldots,j_b.
%\end{align*}
%But then finally, using \Cref{Remark 3.6} again,
%\begin{align*}
%P_{\mathcal{L}_{Y_S}(t)} \frac{D^{(j_b)}Y_b(t+h)-D^{(j_b)}Y_b(t)}{h}
%= P_{\mathcal{L}_{Y_{S \setminus A}}(t)}
%\frac{D^{(j_b)}Y_b(t+h)-D^{(j_b)}Y_b(t)}{h}.
%\end{align*}
%Then, letting $h\to 0$ we receive the statement.

(d)\: Let $\CY_A  \nrarrow \CY_B \: \vert \: \CY_S$, i.e., $\mathcal{L}_{Y_B}(t+1) \perp \mathcal{L}_{Y_A}(t) \: \vert \: \mathcal{L}_{Y_{S\setminus A}}(t)$ for all $t\in \R$ due to \Cref{Charakterisisierung linear granger non-causal} (b). Then, as in the proof of \Cref{Charakterisierung als Gleichheit der linearen Vorhersage} (cf.~Proposition 2.4.2 in \citealp{LI15}), we have
\begin{align*}
P_{\mathcal{L}_{Y_S}(t)}Y^B = P_{\mathcal{L}_{Y_{S \setminus A}}(t)}Y^B \quad \mathbb{P}\text{-a.s.}
\end{align*}
for all $Y^B\in \mathcal{L}_{Y_B}(t+1)$ and $t\in \R$. Furthermore, \Cref{Remark 3.6} provides that, for $b\in B$ and $h \in [0,1]$, we have $D^{(j_b)} Y_b(t+h) \in \mathcal{L}_{Y_{B}}(t+h) \subseteq \mathcal{L}_{Y_{B}}(t+1)$. All together result in
\begin{align*}
P_{\mathcal{L}_{Y_S}(t)} D^{(j_b)} Y_b(t+h) = P_{\mathcal{L}_{Y_{S \setminus A}}(t)} D^{(j_b)} Y_b(t+h) \quad \mathbb{P}\text{-a.s.}
\end{align*}
Since, in addition, $D^{(j_b)}Y_b(t)\in \mathcal{L}_{Y_{S\setminus A}}(t) \subseteq \mathcal{L}_{Y_{S}}(t) $ by \Cref{Remark 3.6} again, we have
\begin{align*}
P_{\mathcal{L}_{Y_S}(t)} \left( \frac{D^{(j_b)}Y_b(t+h)-D^{(j_b)}Y_b(t)}{h} \right)
= P_{\mathcal{L}_{Y_{S \setminus A}}(t)} \left( \frac{D^{(j_b)}Y_b(t+h)-D^{(j_b)}Y_b(t)}{h} \right).
\end{align*}
Letting $h\to 0$, we receive the statement.
\end{proof}

%\subsection{Proofs of Section~\ref{sec:undirected_influences}} %\label{sec:Proof:undirected_influence:proj}

\section{Proofs of Section~\ref{sec:path_diagrams}}

\subsection{Proofs of Subsection~\ref{subsec:condorth}} \label{Sec:Proofs:conditional orthogonality}

\begin{proof}[Proof of \Cref{Eigenschaften der linearen Räume}]
Let $A,B\subseteq V$ be disjoint with $\# A = \alpha$, $ \# B  = \beta$. First, according to \Cref{Assumption an Dichte}, there exists an $0<\varepsilon<1$ such that
\begin{align*}
f_{Y_AY_A}(\lambda)^{-1/2}f_{Y_AY_B}(\lambda)f_{Y_BY_B}(\lambda)^{-1}f_{Y_BY_A}(\lambda)f_{Y_AY_A}(\lambda)^{-1/2} \leq_{L} (1 - \varepsilon)I_{\alpha},
\end{align*}
for almost all $\lambda \in \R$ %Multiplication with $f_{Y_AY_A}(\lambda)^{1/2}$ from both sides and Bernstein
%\cite{BE09}, Proposition 8.1.2. xii), yields to
and hence,
\begin{align*}
(1 - \varepsilon) f_{Y_AY_A}(\lambda) - f_{Y_AY_B}(\lambda)f_{Y_BY_B}(\lambda)^{-1}f_{Y_BY_A}(\lambda) \geq 0,
\end{align*}
for almost all $\lambda \in \R$. If we choose $0< \widetilde{\varepsilon} <1$, such that $(1 - \widetilde{\varepsilon})^2=(1 - \varepsilon)$,
%and multiply the non-negative definite matrix with $1/(1- \tilde{\varepsilon})$,
we obtain
\begin{align*}
(1 -  \widetilde{\varepsilon}) f_{Y_AY_A}(\lambda) - f_{Y_AY_B}(\lambda)\left( (1- \tilde{\varepsilon}) f_{Y_BY_B}(\lambda)\right)^{-1}f_{Y_BY_A}(\lambda) \geq 0,
\end{align*}
for almost all $\lambda \in \R$. Since $(1-\widetilde{\varepsilon})f_{Y_BY_B}(\lambda)  \geq 0$, \cite{BE09}, Proposition 8.2.4., provides
\begin{align*}
\left( \begin{matrix}
(1 -  \widetilde{\varepsilon}) f_{Y_AY_A}(\lambda)  &  f_{Y_AY_B}(\lambda) \\
f_{Y_BY_A}(\lambda) & (1 -  \widetilde{\varepsilon}) f_{Y_BY_B}(\lambda)
\end{matrix} \right) \geq 0,
\end{align*}
respectively
\begin{align} \label{eqB1}
\left( \begin{matrix}
f_{Y_AY_A}(\lambda) &  f_{Y_AY_B}(\lambda) \\
f_{Y_BY_A}(\lambda) & f_{Y_BY_B}(\lambda)
\end{matrix} \right)
\geq_{L}
\widetilde{\varepsilon}
\left( \begin{matrix}
f_{Y_AY_A}(\lambda)  &  0_{\alpha \times \beta} \\
0_{\beta \times \alpha} & f_{Y_BY_B}(\lambda)
\end{matrix} \right),
\end{align}
for almost all $\lambda\in \R$. With this preliminary work in mind, we can now provide the actual proof of the assertion. Let $\HY^A \in \mathcal{L}_{Y_A}(t)$ and $\HY^B \in \mathcal{L}_{Y_B}(t)$, $t\in \R$. Then $\HY^A \in \mathcal{L}_{Y_A}$ and $\HY^B \in \mathcal{L}_{Y_B}$. Due to \cite{RO67}, I, (7.2), the spectral representation
\begin{align*}
\HY^A = \uint \varphi(\lambda) \Phi_A(d\lambda) \quad \text{ and } \quad \quad \HY^B = \uint \psi(\lambda) \Phi_B(d\lambda)
\quad \mathbb{P}\text{-a.s.}
\end{align*}
holds,
 where $\Phi_A(\cdot)$ and $\Phi_B(\cdot)$ are the random spectral measures form  the subprocesses $\CY_A$ and $\CY_B$ from \eqref{spectral representation of stationary process}. Furthermore, $\varphi(\cdot) \in C^{1 \times \alpha}$ and $\psi(\cdot) \in C^{1 \times \beta}$ are measurable vector functions that satisfy

\begin{align*}
\uint \varphi(\lambda) f_{Y_A Y_A}(\lambda) \overline{\varphi(\lambda)}^\top  d\lambda < \infty \quad \text{and}\quad
\uint \psi(\lambda) f_{Y_B Y_B}(\lambda) \overline{\psi(\lambda)}^\top  d\lambda < \infty.
\end{align*}
Using \eqref{eqB1} and the monotonicity of the integral in the inequality, we obtain
\begin{align*}
\Vert \HY^A + \HY^B \Vert_{L^2}^2
%= &\: \uint \varphi(\lambda) f_{Y_A Y_A}(\lambda) \overline{\varphi(\lambda)}^\top  d\lambda
%+ \uint \varphi(\lambda) f_{Y_A Y_B}(\lambda) \overline{\psi(\lambda)}^\top  d\lambda \\
 %  &\: + \uint \psi(\lambda) f_{Y_B Y_A}(\lambda) \overline{\varphi(\lambda)}^\top  d\lambda
%+ \uint \psi(\lambda) f_{Y_B Y_B}(\lambda) \overline{\psi(\lambda)}^\top  d\lambda \\
=  &\: \uint \left( \varphi(\lambda) \: \: \psi(\lambda) \right) \left( \begin{matrix}
f_{Y_AY_A}(\lambda) &  f_{Y_AY_B}(\lambda) \\
f_{Y_BY_A}(\lambda) & f_{Y_BY_B}(\lambda)
\end{matrix} \right) \overline{\left( \varphi(\lambda) \: \: \psi(\lambda) \right)}^\top  d\lambda \\
\geq &\: \widetilde{\varepsilon} \uint \left( \varphi(\lambda) \: \: \psi(\lambda) \right)
\left( \begin{matrix}
f_{Y_AY_A}(\lambda)  &  0_{\alpha \times \beta} \\
0_{\beta \times \alpha} & f_{Y_BY_B}(\lambda)
\end{matrix} \right) \overline{\left( \varphi(\lambda) \: \: \psi(\lambda) \right)}^\top  d\lambda \\
%= &\: \tilde{\varepsilon} \left( \uint \varphi(\lambda) f_{Y_A Y_A}(\lambda) \overline{\varphi(\lambda)}^\top  d\lambda + \uint \psi(\lambda) f_{Y_B Y_B}(\lambda) \overline{\psi(\lambda)}^\top  d\lambda \right) \\
= &\: \widetilde{\varepsilon} \left( \Vert \HY^A \Vert^2 + \Vert \HY^B \Vert_{L^2}^2  \right).
\end{align*}
%Since $\varepsilon$ and hence $\tilde{\varepsilon}$ do not depend on $A$ and $B$ on assumption, the assertion stays valid for any $\HY^A \in \mathcal{L}_{Y_A}(t)$, $\HY^B \in \mathcal{L}_{Y_B}(t)$, $t\in \R$, and disjoint subsets $A,B \subseteq V$.
Then \cite{FE12}, Proposition 2.3, provides that for $t\in\R$,
\begin{align*}
\mathcal{L}_{Y_{A}}(t) \cap \mathcal{L}_{Y_{B }}(t) =  \{0\} \quad \text{and}\quad
\mathcal{L}_{Y_{A}}(t) + \mathcal{L}_{Y_{B}}(t) = \mathcal{L}_{Y_{A}}(t) \vee \mathcal{L}_{Y_{B}}(t)
\quad \mathbb{P}\text{-a.s.}
\end{align*}
 Thus, \Cref{Lemma lineare Separabilität} yields the final statement
$\mathcal{L}_{Y_{A\cup C}}(t) \cap \mathcal{L}_{Y_{B \cup C}}(t) =  \mathcal{L}_{Y_C}(t) $
$\mathbb{P}$-a.s.
\end{proof}

\begin{proof}[Proof of \Cref{Proposition 5.6}] $\mbox{}$ \\
(a) The direction $\Rightarrow$ is already given in \eqref{eq3b}. Thus, let us prove $\Leftarrow$ and assume  that $\CY_{a} \nrarrow \CY_b\: \vert \: \CY_S$ for all $a\in A$, $b\in B$. Then we receive due to \Cref{Charakterisierung als Gleichheit der linearen Vorhersage} that
\begin{align*}
P_{\mathcal{L}_{Y_S}(t)}Y_b(t+h) = P_{\mathcal{L}_{Y_{S \setminus \{a\}}}(t)}Y_b(t+h) \quad \mathbb{P}\text{-a.s.}
\end{align*}
for all $h \in [0,1]$, $t\in\R$, $a\in A$, $b\in B$. 
This implies that
\begin{align*}
    P_{\mathcal{L}_{Y_S}(t)}Y_b(t+h) \in \mathcal{L}_{Y_{S \setminus \{a\}}}(t)\quad \forall\, a\in A.
\end{align*}
Now, from \Cref{Eigenschaften der linearen Räume}, which requires \Cref{Assumption an Dichte}, we conclude that
\begin{align*}
    P_{\mathcal{L}_{Y_S}(t)}Y_b(t+h) \in \bigcap_{a\in A}\mathcal{L}_{Y_{S \setminus \{a\}}}(t)=
    \mathcal{L}_{Y_{S \setminus A}}(t),
\end{align*}
implying due to \cite{BR91}, Proposition 2.3.2. (vii) that
\begin{align*}
P_{\mathcal{L}_{Y_S}(t)}Y_b(t+h)
= P_{\mathcal{L}_{Y_{S \setminus A}(t)}} P_{\mathcal{L}_{Y_S}(t)} Y_b(t+h)
= P_{\mathcal{L}_{Y_{S \setminus A}(t)}}Y_b(t+h) \quad \mathbb{P}\text{-a.s.}
\end{align*}
for all $b\in B$, $t\in \R$, and $h \in [0,1]$. We apply \Cref{Charakterisierung als Gleichheit der linearen Vorhersage} again and obtain $\CY_{A} \nrarrow \CY_B\: \vert \: \CY_S$.\\
(b) The direction $\Rightarrow$ is already given in \eqref{eq3.4} and we just prove $\Leftarrow$.
Thus assume  that $\CY_{a} \nrarrownull \CY_b\: \vert \: \CY_S$ for all $a\in A$, $b\in B$. By \Cref{definition (linear) local Granger non-causal} that is
\begin{align*}
&\limh P_{\mathcal{L}_{Y_S}(t)} \left(\frac{D^{(j_b)} Y_b(t+h)- D^{(j_b)} Y_b(t)}{h}\right) \\
&\quad =\limh P_{\mathcal{L}_{Y_{S \setminus \{a\}}}(t)} \left(\frac{D^{(j_b)} Y_b(t+h)- D^{(j_b)} Y_b(t)}{h}\right) \quad \mathbb{P}\text{-a.s.}
\end{align*}
for all $t\in \R$, $a\in A$, $b\in B$. Since $\mathcal{L}_{Y_{S \setminus \{a\}}}(t)$ is closed in the mean-square sense, we obtain
\begin{align*}
\limh P_{\mathcal{L}_{Y_S}(t)} \left(\frac{D^{(j_b)} Y_b(t+h)- D^{(j_b)} Y_b(t)}{h}\right) \in
\mathcal{L}_{Y_{S \setminus \{a\}}}(t) \quad \forall\, a\in A.
\end{align*}
As in (a), \Cref{Eigenschaften der linearen Räume}, which requires \Cref{Assumption an Dichte}, yields

\begin{align*}
\limh P_{\mathcal{L}_{Y_S}(t)} \left(\frac{D^{(j_b)} Y_b(t+h)- D^{(j_b)} Y_b(t)}{h}\right) \in
\mathcal{L}_{Y_{S \setminus A}}(t).
\end{align*}
Due to \cite{BR91}, Proposition 2.3.2. (iv) and (vii), it follows
\begin{align*}
&\limh P_{\mathcal{L}_{Y_S}(t)} \left(\frac{D^{(j_b)} Y_b(t+h)- D^{(j_b)} Y_b(t)}{h}\right) \\
&\quad = P_{\mathcal{L}_{Y_{S\setminus A}}(t)} \limh P_{\mathcal{L}_{Y_S}(t)} \left(\frac{D^{(j_b)} Y_b(t+h)- D^{(j_b)} Y_b(t)}{h}\right) \\
&\quad = \limh P_{\mathcal{L}_{Y_{S\setminus A}}(t)} P_{\mathcal{L}_{Y_S}(t)} \left(\frac{D^{(j_b)} Y_b(t+h)- D^{(j_b)} Y_b(t)}{h}\right) \\
&\quad = \limh P_{\mathcal{L}_{Y_{S\setminus A}}(t)} \left(\frac{D^{(j_b)} Y_b(t+h)- D^{(j_b)} Y_b(t)}{h}\right) \quad \mathbb{P}\text{-a.s.}
\end{align*}
for all $b\in B$, $t\in \R$. By \Cref{definition (linear) local Granger non-causal} that is $\CY_{A} \nrarrownull \CY_B\: \vert \: \CY_S$.\\
(c) The proof is the same as in (a).
\end{proof}

\subsection{Proof of Theorem~\ref{global markov property}}
\label{subsec:proofAMP}
The proof of the global AMP Markov property is structured in three auxiliary lemmata and is based on the ideas of \cite{EI07} and  \cite{EI10}.
%we structure the proof in
%the evidence by showing
%three auxiliary lemmatas that build on each other and are based on the ideas of  \cite{EI07} and  \cite{EI10}. % in four corresponding lemmas.
At the end, we present the proof of \Cref{global markov property}.

\begin{comment}
\begin{lemma}\label{Hilfslemma B1}
Let $G=(V, E)$ be a mixed graph. Then
\begin{align*}
A \msep B \: \vert \: V \setminus (A \cup B)  \:\: [G] \quad \Rightarrow \quad
\dis\left(A \cup \ch(A) \right) \cap \dis \left( B\cup \ch(B) \right) = \emptyset.
\end{align*}

Here
\begin{align*}
\ch(a)=\{v\in V \vert a \rarrow v \in E \},
\quad
\dis(a)= \{ v\in V \vert v \inst \cdots \inst a \text{ or } v=a \},
\end{align*}
denote the set of {children} respectively the district of $a \in V$. For $A \subseteq V$ let
\begin{align*}
\ch(A) = \bigcup_{a \in A} \ch(a),
\quad\dis(A) = \bigcup_{a\in A} \dis(a).
\end{align*}
%be the set of all children of vertices in the subset $A$ respectively the district of $A$.
% that is the set of all vertices that are connected to the vertex $a$ by a path that consists only of undirected edges complemented by the vertex $a$ itself.
\end{lemma}
\begin{proof}
The proof %of this assertion
does not depend on the specific definition of the edges in the mixed graph. Thus, the proof from \cite{EI07}, Lemma B.1, is directly applicable.
%A more detailed proof can be found in \cite{EI11}, Lemma 3.2.
\end{proof}
\end{comment}

\begin{lemma}\label{Hilfslemma B2}
Let $G_{OG}=(V,E_{OG})$ be the orthogonality graph for $\CY_V$. Suppose $A,B \subseteq V$ are disjoint subsets, $t\in \R$, and $k\in \N$. Then
\begin{align*}
A \msep B \: \vert \: V\setminus (A \cup B)  \:\: [G_{OG}] \quad \Rightarrow \quad \mathcal{L}_{Y_A}(t) \perp \mathcal{L}_{Y_B}(t) \: \vert \: \mathcal{L}_{Y_{ V\setminus (A \cup B)}}(t) \vee \mathcal{L}_{Y_{A\cup B}}(t-k).
\end{align*}
\end{lemma}

\begin{proof}
The proof can be done step by step as in \cite{EI10}, proof of Lemma 4.1, by induction over $k$, using the properties of a semi-graphoid given in our \Cref{properties of conditional orthogonality}.
\end{proof}

\begin{lemma}\label{Hilfslemma B3}
Let $G_{OG}=(V,E_{OG})$ be the orthogonality graph for $\CY_V$. Suppose $A, B \subseteq V$ are disjoint subsets and $t\in\R$. Then
\begin{align*}
A \msep B \: \vert \: V\setminus (A \cup B)  \:\: [G_{OG}] \quad \Rightarrow \quad \mathcal{L}_{Y_A}(t) \perp \mathcal{L}_{Y_B}(t) \: \vert \: \mathcal{L}_{Y_{ V\setminus (A \cup B)}}(t).
\end{align*}
\end{lemma}

\begin{proof}
 First, $\mathcal{L}_{Y_{A \cup B}}(t-k) \vee \mathcal{L}_{Y_{V\setminus (A \cup B)}}(t) \supseteq \mathcal{L}_{Y_{A \cup B}}(t-k-1) \vee \mathcal{L}_{Y_{V\setminus (A \cup B)}}(t)$ \mbox{for $k\in \N$ and}
\begin{align*}
\bigcap_{k \in \N} \left( \mathcal{L}_{Y_{A \cup B}}(t-k) \vee \mathcal{L}_{Y_{V\setminus (A \cup B)}}(t) \right) =\mathcal{L}_{Y_{V\setminus (A \cup B)}}(t),
\end{align*}
due to \Cref{Eigenschaft für Hilfslemma B3}. Theorems 4.31 (b) and 4.32 in \cite{WE80} provide
%a monotonically decreasing sequence of orthogonal projections
% \begin{align*}
% \left( P_{\mathcal{L}_{Y_{A \cup B}}(t-k) \vee \mathcal{L}_{Y_{V\setminus (A \cup B)}}(t)} \right)_{k\in \N}
% \end{align*}
%with corresponding limiting orthogonal projection
%\begin{align*}
%P_{\bigcap_{k\in \N} \mathcal{L}_{Y_{A \cup B}}(t-k) \vee %\mathcal{L}_{Y_{V\setminus (A \cup B)}}(t)}
%=P_{\mathcal{L}_{Y_{V\setminus (A \cup B)}}(t)},
%\end{align*}
%where we apply \Cref{Assumption purely nondeterministic of full rank}, respectively \Cref{Eigenschaft für Hilfslemma B3}. % to obtain the equality.
%This means that for $\HY^A \in \mathcal{L}_{Y_A}(t)$ and $\HY^B \in \mathcal{L}_{Y_B}(t)$ in particular the following convergences in mean square are valid
\begin{align*}
\limk P_{\mathcal{L}_{Y_{A \cup B}}(t-k) \vee \mathcal{L}_{Y_{V\setminus (A \cup B)}}(t)}Y &=
P_{\mathcal{L}_{Y_{V\setminus (A \cup B)}}(t)}Y, \quad Y\in L^2.
%\\
%\limk P_{\mathcal{L}_{Y_{A \cup B}}(t-k) \vee \mathcal{L}_{Y_{V\setminus (A \cup B)}}(t)}\HY^B &=
%P_{\mathcal{L}_{Y_{V\setminus (A \cup B)}}(t)}\HY^B.
\end{align*}
 Let $\HY^A \in \mathcal{L}_{Y_A}(t)$ and $\HY^B \in \mathcal{L}_{Y_B}(t)$.
Then, using \eqref{limit E},
\begin{align*}
&\: \BE \left[ \left(\HY^A -P_{\mathcal{L}_{Y_{V\setminus (A \cup B)}}(t)}\HY^A \right) \overline{\left(\HY^B -P_{\mathcal{L}_{Y_{V\setminus (A \cup B)}}(t)}\HY^B \right)} \right] \\
&\quad = \: \lim_{k\rightarrow \infty} \BE \left[ \left(\HY^A - P_{\mathcal{L}_{Y_{A \cup B}}(t-k) \vee \mathcal{L}_{Y_{V\setminus (A \cup B)}}(t)}\HY^A  \right) \right. \\
& \quad\quad \left. \phantom{\limk \BE}
 \times \overline{\left(\HY^B - \:P_{\mathcal{L}_{Y_{A \cup B}}(t-k) \vee \mathcal{L}_{Y_{V\setminus (A \cup B)}}(t)}\HY^B \right)} \right].
\end{align*}
The expression on the right-hand side is zero since, due to \Cref{Hilfslemma B2}, $\mathcal{L}_{Y_A}(t) \perp \mathcal{L}_{Y_B}(t) \: \vert \: \mathcal{L}_{Y_{ V\setminus (A \cup B)}}(t) \vee \mathcal{L}_{Y_{A\cup B}}(t-k)$ for $t\in \R$, $k\in \N$. Thus, the expression on the left-hand side is also zero and $\mathcal{L}_{Y_A}(t) \perp \mathcal{L}_{Y_B}(t) \: \vert \: \mathcal{L}_{Y_{ V\setminus (A \cup B)}}(t)$.
\end{proof}

\begin{lemma}\label{Hilfslemma B4}
Let $G_{OG}=(V, E_{OG})$ be the orthogonality graph for $\CY_V$ and suppose $A, B \subseteq V$ are disjoint subsets. Then
\begin{align*}
A \msep B \: \vert \: V\setminus (A \cup B)  \:\: [G_{OG}] \quad \Rightarrow \quad \mathcal{L}_{Y_A} \perp \mathcal{L}_{Y_B} \: \vert \: \mathcal{L}_{Y_{ V\setminus (A \cup B)}}.
\end{align*}
\end{lemma}

\begin{proof}
First, note from \Cref{additivity in time domain} that
$\overline{\bigcup_{n\in \N} \mathcal{L}_{Y_S}(n)} = \mathcal{L}_{Y_S}$
$\mathbb{P}$-a.s. for any $S\subseteq V$.
Let $\HY^A \in \mathcal{L}_{Y_A}$ and $\HY^B \in \mathcal{L}_{Y_B}$.
Then analogue arguments as in the proof of \Cref{Hilfslemma B3} give
\begin{align*}
\HY^A - P_{\mathcal{L}_{Y_{V \setminus (A \cup B)}}} \HY^A
&= \limm \: P_{\mathcal{L}_{Y_A}(n)} \HY^A -  P_{\mathcal{L}_{Y_{V \setminus (A \cup B)}}(n)} P_{\mathcal{L}_{Y_A}(n)} \HY^A , \\
\HY^B - P_{\mathcal{L}_{Y_{V \setminus (A \cup B)}}} \HY^B
&= \limm \: P_{\mathcal{L}_{Y_B}(n)} \HY^B -  P_{\mathcal{L}_{Y_{V \setminus (A \cup B)}}(n)} P_{\mathcal{L}_{Y_B}(n)} \HY^B.
\end{align*}
Further, \eqref{limit E} yields
%with the same argument as in the proof of \Cref{Hilfslemma B3} that
\begin{align*}
& \: \BE \left[
\left( \HY^A - P_{\mathcal{L}_{Y_{V \setminus (A \cup B)}}} \HY^A \right)
\overline{\left( \HY^B - P_{\mathcal{L}_{Y_{V \setminus (A \cup B)}}} \HY^B \right)}
\right] \\
& \quad=  \:  \lim_{n\rightarrow \infty} \: \BE \left[
\left( P_{\mathcal{L}_{Y_A}(n)} \HY^A - P_{\mathcal{L}_{Y_{V \setminus (A \cup B)}}(n)} P_{\mathcal{L}_{Y_A}(n)} \HY^A \right) \right. \\
&  \left. \phantom{\limm \: \BE} \quad \quad\times\overline{\left( P_{\mathcal{L}_{Y_B}(n)} \HY^B - P_{\mathcal{L}_{Y_{V \setminus (A \cup B)}}(n)} P_{\mathcal{L}_{Y_B}(n)} \HY^B \right)}
\right].
\end{align*}
The expression on the right-hand side is zero, since $\mathcal{L}_{Y_A}(t) \perp \mathcal{L}_{Y_B}(t) \: \vert \: \mathcal{L}_{Y_{V \setminus (A \cup B)}}(t)$ for $t\in \R$ due to \Cref{Hilfslemma B3}. Thus, the left-hand side is zero and  $\mathcal{L}_{Y_A} \perp \mathcal{L}_{Y_B} \: \vert \: \mathcal{L}_{Y_{ V\setminus (A \cup B)}}$.
\end{proof}

\begin{proof}[Proof of \Cref{global markov property}]
For the proof of \Cref{global markov property}, we refer to the proof of Theorem 3.1 in \cite{EI07}, since it is based only on \Cref{Hilfslemma B4}, properties of mixed graphs, and \Cref{properties of conditional orthogonality}.
\end{proof}

\subsection{Proofs of Subsection~\ref{Global Markov properties for the local orthogonality graph}}\label{proofslocalglobalMarkovproperty}

\begin{proof}[Proof of \Cref{HilfslemmaABCistV}]
For a graph $G=(V,E)$ let
\begin{align*}
\ch(a)=\{v\in V \vert a \rarrow v \in E \}
\quad \text{ and } \quad
\dis(a)= \{ v\in V \vert v \inst \cdots \inst a \text{ or } v=a \},
\end{align*}
denote the set of children and the district of $a \in V$, respectively. For $A \subseteq V$ let \linebreak
$
\ch(A) = \bigcup_{a \in A} \ch(a)$ and
$\dis(A) = \bigcup_{a\in A} \dis(a).$
Due to \cite{EI07}, Lemma B.1, $A \msep B \: \vert \: V \setminus (A \cup B) \: [G_{OG}^0]$ yields  $$\dis\left(A \cup \ch(A) \right) \cap \dis \left( B \cup \ch(B) \right) = \emptyset.$$  In particular, $\ch(A) \cap B=\emptyset$, $A \cap \ch(B) = \emptyset$, and $\ne(A) \cap B=\emptyset$. Thus, as claimed, $\CY_A \nrarrownull \CY_B \: \vert \: \CY_{V}$, $\CY_B \nrarrownull \CY_A \: \vert \: \CY_{V}$, and $\CY_A \nsimnull \CY_B \: \vert \: \CY_{V}$.
\end{proof}

\begin{proof}[Proof of \Cref{Leftdecomposition}]
The assumption $\CY_{A \cup B} \nrarrownull \CY_C \: \vert \: \CY_{A \cup B \cup C \cup D}$ states that   for all $t\in \R$ and $c \in C$,
\begin{align*}
&\limh P_{\mathcal{L}_{Y_{A \cup B \cup C \cup D}}(t)} \left(\frac{D^{(j_c)} Y_c(t+h)- D^{(j_c)} Y_c(t)}{h}\right)\\
&\quad=\limh P_{\mathcal{L}_{Y_{C \cup D}}(t)} \left(\frac{D^{(j_c)} Y_c(t+h)- D^{(j_c)} Y_c(t)}{h}\right) \quad \mathbb{P}\text{-a.s.}
\end{align*}
An application of  $ P_{\mathcal{L}_{Y_{A  \cup C \cup D}}(t)}$  on the left and the right hand side, \cite{BR91}, Proposition 2.3.2. (iv) and (vii), and 

\begin{align*}
    P_{\mathcal{L}_{Y_{A \cup C \cup D}}(t)}  P_{\mathcal{L}_{Y_{A \cup B \cup C \cup D}}(t)}=P_{\mathcal{L}_{Y_{A \cup C \cup D}}(t)} \quad \text{and} \quad P_{\mathcal{L}_{Y_{A \cup C \cup D}}(t)}  P_{\mathcal{L}_{Y_{C \cup D}}(t)} = P_{\mathcal{L}_{Y_{C \cup D}}(t)},
\end{align*}
respectively, give for $t\in \R$ and $c \in C$,
\begin{comment}
\LS{
Betrachte
\begin{align*}
\limh P_{\mathcal{L}_{Y_{A \cup B \cup C \cup D}}(t)} \left(\frac{D^{(j_c)} Y_c(t+h)- D^{(j_c)} Y_c(t)}{h}\right)
= X
\end{align*}
dann hat man eine Folge, die gegen X konviergiert. Dann konvergiert mit Brockwell
\begin{align*}
\limh P_{\mathcal{L}_{Y_{A \cup C \cup D}}(t)}  P_{\mathcal{L}_{Y_{A \cup B \cup C \cup D}}(t)} \left(\frac{D^{(j_c)} Y_c(t+h)- D^{(j_c)} Y_c(t)}{h}\right)
=  \limh P_{\mathcal{L}_{Y_{A \cup C \cup D}}(t)}  X
\end{align*}
und das ist wegen den aufsteigenden Räumen
\begin{align*}
\limh P_{\mathcal{L}_{Y_{A \cup C \cup D}}(t)} \left(\frac{D^{(j_c)} Y_c(t+h)- D^{(j_c)} Y_c(t)}{h}\right)
=  \limh P_{\mathcal{L}_{Y_{A \cup C \cup D}}(t)}  X
\end{align*}
Analog wegen der Gleichheit
\begin{align*}
\limh P_{\mathcal{L}_{Y_{C \cup D}}(t)} \left(\frac{D^{(j_c)} Y_c(t+h)- D^{(j_c)} Y_c(t)}{h}\right)
= X
\end{align*}
dann hat man eine Folge, die gegen X konviergiert. Dann konvergiert wieder mit Brockwell
\begin{align*}
\limh P_{\mathcal{L}_{Y_{A \cup C \cup D}}(t)}  P_{\mathcal{L}_{Y_{C \cup D}}(t)} \left(\frac{D^{(j_c)} Y_c(t+h)- D^{(j_c)} Y_c(t)}{h}\right)
=  \limh P_{\mathcal{L}_{Y_{A \cup C \cup D}}(t)}  X
\end{align*}
und das ist wegen den aufsteigenden Räumen
\begin{align*}
\limh P_{\mathcal{L}_{Y_{C \cup D}}(t)} \left(\frac{D^{(j_c)} Y_c(t+h)- D^{(j_c)} Y_c(t)}{h}\right)
=  \limh P_{\mathcal{L}_{Y_{A \cup C \cup D}}(t)}  X
\end{align*}
Gleichsetzen liefert das Resultat}
\end{comment}
\begin{align*}
&\limh P_{\mathcal{L}_{Y_{A \cup C \cup D}}(t)} \left(\frac{D^{(j_c)} Y_c(t+h)- D^{(j_c)} Y_c(t)}{h}\right)\\
&\quad =\limh P_{\mathcal{L}_{Y_{C \cup D}}(t)} \left(\frac{D^{(j_c)} Y_c(t+h)- D^{(j_c)} Y_c(t)}{h}\right)
\quad \mathbb{P}\text{-a.s.}
\end{align*}
By definition that is $\CY_{A} \nrarrownull \CY_C \: \vert \: \CY_{A \cup C \cup D}$.
\end{proof}

\begin{proof}[Proof of \Cref{LemmapaApaBinABC}]
The block-recursive Markov property (\Cref{block-recursive Markov property}) says that
$
\CY_{V\setminus (B \cup \pa(B))} \nrarrownull \CY_B \: \vert \: \CY_{V}.
$
By assumption, $B \cup \pa(B) \subseteq A \cup B \cup C$. However, $A \cap \pa(B)=\emptyset$. Otherwise, there are vertices $a \in A$ and $b\in B$ such that $a \rarrow b\in E_{OG}^0$ is a $m$-connecting path between $A$ and $B$ given $C$ which is a contradiction to $A \msep B \: \vert \: C \: [G_{OG}^0]$. Thus, $B \cup \pa(B) \subseteq B \cup C$ and \Cref{Proposition 5.6} yields
%\begin{align*}
$\CY_{V\setminus (B \cup C)} \nrarrownull \CY_B \: \vert \: \CY_{V}.$
%\end{align*}
The property of left decomposition (\Cref{Leftdecomposition}) gives
%\begin{align*}
$\CY_{A} \nrarrownull \CY_B \: \vert \: \CY_{A \cup B \cup C}.$
%\end{align*}
By symmetry of $m$-separation $\CY_{B} \nrarrownull \CY_A \: \vert \: \CY_{A \cup B \cup C}$ follows.

It remains to show that $\CY_A \nsimnull \CY_B \: \vert \: \CY_{A \cup B \cup C}$.
\Cref{block-recursive Markov property} provides
%\begin{align*}
$\CY_{V\setminus (B \cup \ne(B))} \nsimnull \CY_B \: \vert \: \CY_{V}.$
%\end{align*}
Here, $A \cap \ne(B)=\emptyset$. Else there are vertices $a \in A$ and $b\in B$ such that $a \inst b\in E_{OG}^0$ is a $m$-connecting path between $A$ and $B$ given $C$ which is again a contradiction to $A \msep B \: \vert \: C \: [G_{OG}^0]$. So \Cref{eq:correspondenceXAXa local} yields
%\begin{align*}
$\CY_A \nsimnull \CY_B \: \vert \: \CY_{V}.$
%\end{align*}
By definition and  $D^{(j_a)} Y_a(t), D^{(j_b)} Y_b(t) \in \mathcal{L}_{Y_{A \cup B \cup C}}(t)\subseteq \mathcal{L}_{Y_{V}}(t)$
we get
\begin{align}\label{eq:hilf1}
0= \limhh \frac{1}{h} \: \BE & \left[\left( D^{(j_a)} Y_a(t+h)- P_{\mathcal{L}_{Y_{V}}(t)} D^{(j_a)} Y_a(t+h) \right) \right. \nonumber \\
		  &\times \left. \overline{\left( D^{(j_b)} Y_b(t+h)- P_{\mathcal{L}_{Y_{V}}(t)} D^{(j_b)} Y_b(t+h) \right)} \right] \nonumber \\
= \limhh h  \: \BE & \left[\left( \frac{D^{(j_a)} Y_a(t+h)- D^{(j_a)} Y_a(t)}{h}- P_{\mathcal{L}_{Y_{V}}(t)} \frac{D^{(j_a)} Y_a(t+h)- D^{(j_a)} Y_a(t)}{h} \right) \right. \nonumber \\
		  &\times \left. \overline{\left( \frac{D^{(j_b)} Y_b(t+h)- D^{(j_b)} Y_b(t)}{h}- P_{\mathcal{L}_{Y_{V}}(t)} \frac{D^{(j_b)} Y_b(t+h)- D^{(j_b)} Y_b(t)}{h} \right)} \right] ,
\end{align}
for $t\in \R$, $a\in A$, $b\in B$. Due to \Cref{block-recursive Markov property} and $\pa(A) \cup \pa(B) \subseteq A \cup B \cup C$ we receive, as in the first part of this proof,
\begin{align*}
\CY_{V\setminus (A \cup B \cup C)} \nrarrownull \CY_B \: \vert \: \CY_{V} \quad \text{and} \quad
\CY_{V\setminus (A \cup B \cup C)} \nrarrownull \CY_A \: \vert \: \CY_{V},
\end{align*}
which means that $\mathbb{P}$-a.s.
\begin{align*}
&\limh P_{\mathcal{L}_{Y_{V}}(t)} \left(\frac{D^{(j_b)} Y_b(t+h)- D^{(j_b)} Y_b(t)}{h}\right)\\
&\quad =\limh P_{\mathcal{L}_{Y_{A\cup B \cup C}}(t)} \left(\frac{D^{(j_b)} Y_b(t+h)- D^{(j_b)} Y_b(t)}{h}\right) \quad \text{and} \\
&\limh P_{\mathcal{L}_{Y_{V}}(t)} \left(\frac{D^{(j_a)} Y_a(t+h)- D^{(j_a)} Y_a(t)}{h}\right)\\
&\quad =\limh P_{\mathcal{L}_{Y_{A\cup B \cup C}}(t)} \left(\frac{D^{(j_a)} Y_a(t+h)- D^{(j_a)} Y_a(t)}{h}\right),
\end{align*}
for $t\in \R$, $a\in A$, $b\in B$. Similar arguments as in the proof of
\Cref{Charakterisierung local Granger non-causal} and \eqref{eq:hilf1} yield 

%\eqref{limit E}  we can replace now the  projections in \eqref{eq:hilf1} and obtain
\begin{align*}
0= \limhh h \: \BE & \left[\left( \frac{D^{(j_a)} Y_a(t+h)- D^{(j_a)} Y_a(t)}{h}- P_{\mathcal{L}_{Y_{A \cup B \cup C}}(t)}  \frac{D^{(j_a)} Y_a(t+h)- D^{(j_a)} Y_a(t)}{h} \right) \right.\\
		  &\times \left. \overline{\left( \frac{D^{(j_b)} Y_b(t+h)- D^{(j_b)} Y_b(t)}{h}- P_{\mathcal{L}_{Y_{A \cup B \cup C}}(t)} \frac{D^{(j_b)} Y_b(t+h)- D^{(j_b)} Y_b(t)}{h} \right)} \right] \\
= \limhh \frac{1}{h} \: \BE & \left[\left( D^{(j_a)} Y_a(t+h)- P_{\mathcal{L}_{Y_{A \cup B \cup C}}(t)} D^{(j_a)} Y_a(t+h) \right) \right. \nonumber \\
  &\times \left. \overline{\left( D^{(j_b)} Y_b(t+h)- P_{\mathcal{L}_{Y_{A \cup B \cup C}}(t)} D^{(j_b)} Y_b(t+h) \right)} \right] \nonumber
\end{align*}
for $t\in \R$, $a\in A$, $b\in B$, which says that $\CY_A \nsimnull \CY_B \: \vert \: \CY_{A \cup B \cup C}$.
\end{proof}

\section{Proofs of Section~\ref{sec:CGMCAR}} \label{Appendix C}

\subsection{Proofs of Subsection~\ref{subsec:predMCAR}} \label{subsec:proofpredMCAR}

\begin{proof}[Proof of \Cref{Yb für MCAR}]%[Proof of Proposition~\ref{Yb für MCAR}]
Let $t\in \R$, $h\geq 0$, and $v\in V$. First of all, due to \Cref{Lemma 5.2},
\begin{align*}
Y_v(t+h)
= e_v^\top  \BFC X(t+h)
= e_v^\top  \BFC \left(  e^{\BA h} X(t) + \int_t^{t+h} e^{\BA (t+h-u)} \BB dL(u)  \right).
\end{align*}
With the definition of the $j$-th $k$-block $X^{(j)}$ of $\CX$ as in \eqref{block} and with
 \eqref{derivatives Y} it follows
\begin{align*}
Y_v(t+h)
%&= e_v^\top  \BFC \left(  e^{\BA h} X(t) + \int_t^{t+h} e^{\BA (t+h-u)} \BB dL(u)  \right) \\
&= e_v^\top  \BFC e^{\BA h}  \sum_{j=1}^p \BFE_j X^{(j)}(t) + e_v^\top  \BFC \int_t^{t+h} e^{\BA (t+h-u)} \BB dL(u)\\
&= e_v^\top  \BFC e^{\BA h}  \sum_{j=1}^p \BFE_j D^{(j-1)}Y_V(t) + e_v^\top  \BFC \int_t^{t+h} e^{\BA (t+h-u)} \BB dL(u). \qedhere
\end{align*}
%$\mathbb{P}$-a.s. due to \cite{BR12}, Lemma 3.3. Finally, $Y_V(t) = \BFC X(t) = X^{(1)}(t)$ provides
%\begin{align*}
%Y_v(t+h)
%&= e_v^\top  \BFC e^{\BA h} \sum_{j=1}^p \BFE_j D^{(j-1)} Y_V(t) + e_v^\top  \BFC \int_t^{t+h} e^{\BA (t+h-u)} \BB dL(u),
%\end{align*}
%$\mathbb{P}$-a.s. for $t\in \R$, $h\geq 0$, and $v\in V$.
\end{proof}

\begin{proof}[Proof of \Cref{Projektionen mit lim}]
%\underline{Proof of \eqref{CG1}:}
For the proof of the first equation note that the MCAR$(p)$ process $\CY_V$ is $(p-1)$-times differentiable with $D^{(p-1)} Y_V(t) =  X^{(p)}(t)=\BFE_p^\top X(t)$, see \Cref{Remark 5.3}. Then, as in the proof of \Cref{Yb für MCAR},
\begin{align*}
&D^{(p-1)} Y_v(t+h)- D^{(p-1)} Y_v(t) \\
%&= e_v^\top  \left(X^{(p)}(t+h)- X^{(p)}(t) \right)\\
%&\quad =  e_{v}^\top  \BFE_p^\top  \left( X(t+h)- X(t)\right)\\
%&=  e_{b}^\top  \BFE_p^\top   \left( e^{\BA h} X(t) + \int_t^{t+h} e^{\BA (t+h-u)} \BB dL(u)- X(t) \right) \\
&\quad=  e_{v}^\top  \BFE_p^\top  \left( \left( e^{\BA h} -I_{kp} \right) X(t) + \int_t^{t+h} e^{\BA (t+h-u)} \BB dL(u)\right) \\
&\quad=  e_{v}^\top  \BFE_p^\top  \left( e^{\BA h} -I_{kp} \right) \sum_{j=1}^p \BFE_j    D^{(j-1)} Y_V(t)
+ e_{v}^\top  \BFE_p^\top  \int_t^{t+h} e^{\BA (t+h-u)} \BB dL(u).
\end{align*}
\Cref{Remark 3.6} states that $Y_s(t)$ and its derivatives are already in $\mathcal{L}_{Y_S}(t)$ and are therefore projected onto themselves. Additionally, $\sigma(Y_S(t'), t'\leq t)$ and $\sigma(L(t+h)-L(t'), t \leq t' \leq t+h)$ are independent and thus, $e_v^\top  \BFE_p^\top  \int_t^{t+h} e^{\BA (t+h-u)} \BB dL(u)$ is projected on zero. It follows
\begin{align*}
& P_{\mathcal{L}_{Y_{S}}(t)}\left(D^{(p-1)} Y_v(t+h)- D^{(p-1)} Y_v(t)\right) \\
%&\quad = \: P_{\mathcal{L}_{Y_{S}}(t)}  e_{b}^\top  \BFE_p^\top  \left( \left( e^{\BA h} -I_{kp} \right) X(t) + \int_t^{t+h} e^{\BA (t+h-u)} \BB dL(u)\right)\\
%&\quad = \: P_{\mathcal{L}_{Y_{S}}(t)}\left( e_{v}^\top  \BFE_p^\top \left(  e^{\BA h} -I_{kp} \right) X(t)\right)\\
%&\quad= \: P_{\mathcal{L}_{Y_{S}}(t)}\left( e_{v}^\top  \BFE_p^\top \left( e^{\BA h} -I_{kp}\right) \sum_{j=1}^p\BFE_j      X^{(j)}(t)\right).
%&\quad= \:  e_{b}^\top  \BFE_p^\top \sum_{j=1}^p \left( e^{\BA h} -I_{kp} \right)\BFE_j P_{\mathcal{L}_{Y_{S}}(t)}
%   D^{(j-1)} X^{(1)}(t)\\
%\intertext{An application of \eqref{derivatives Y}  and \Cref{Remark 3.6}  give}
%& P_{\mathcal{L}_{Y_{S}}(t)}\left(D^{(p-1)} Y_v(t+h)- D^{(p-1)} Y_v(t)\right)\\
%&\quad =\:  P_{\mathcal{L}_{Y_{S}}(t)}\left(e_{v}^\top  \BFE_p^\top  \left(e^{\BA h} -I_{kp} \right)  \sum_{j=1}^p \BFE_j    D^{(j-1)} Y_V(t)\right)\\
%&\quad =\: P_{\mathcal{L}_{Y_{S}}(t)}
% \sum_{j=1}^p \sum_{s=1}^k  e_{b}^\top  \BFE_p^\top  \left( e^{\BA h} -I_{kp} \right) \BFE_j e_s D^{(j-1)} Y_s(t)\\
&\quad =\: e_{v}^\top  \BFE_p^\top  \left( e^{\BA h} -I_{kp} \right) \sum_{s \in S} \sum_{j=1}^p  \BFE_j e_s  D^{(j-1)} Y_s(t) \\
&\quad\quad\: + e_{v}^\top  \BFE_p^\top  \left( e^{\BA h} -I_{kp} \right)  \sum_{s \in V\setminus S} \sum_{j=1}^p   \BFE_j e_s  P_{\mathcal{L}_{Y_{S}}(t)} \left( D^{(j-1)} Y_s(t)  \right)\quad \mathbb{P}\text{-a.s.}
\end{align*}
%\underline{Proof of \eqref{CG2}:}
For the proof of the second equation, we apply the same arguments to receive
%\LS{Hä?}
%\begin{align*}
%&\: D^{(j)} Y_v(t+h)- P_{\mathcal{L}_{Y_V}(t)} D^{(j)} Y_v(t+h) \\
%&\quad =  \:e_v^\top  \BFE_p^\top  \left( e^{\BA h} X(t) + %\int_t^{t+h} e^{\BA (t+h-u)} \BB dL(u)\right)  \\
%   &\quad\quad\:- P_{\mathcal{L}_{Y_V}(t)}\left(  e_v^\top  \BFE_p^\top  \left( e^{\BA h} X(t) + \int_t^{t+h} e^{\BA (t+h-u)} \BB dL(u)\right)\right) \\
%&\quad = \: e_v^\top  \BFE_p^\top  \left( e^{\BA h} X(t) + \int_t^{t+h} e^{\BA (t+h-u)} \BB dL(u)\right) - e_v^\top  \BFE_p^\top  e^{\BA h} X(t) \\
%& \quad = \: e_v^\top  \BFE_p^\top  \int_t^{t+h} e^{\BA (t+h-u)} \BB dL(u).
%\end{align*}
\begin{align*}
&\: D^{(p-1)} Y_v(t+h)- P_{\mathcal{L}_{Y_V}(t)} D^{(p-1)} Y_v(t+h) \\
&\quad =  \:e_v^\top  \BFE_p^\top  \left( e^{\BA h} X(t) + \int_t^{t+h} e^{\BA (t+h-u)} \BB dL(u)\right)  \\
   &\quad\quad\:- P_{\mathcal{L}_{Y_V}(t)}\left(  e_v^\top  \BFE_p^\top  \left( e^{\BA h} X(t) + \int_t^{t+h} e^{\BA (t+h-u)} \BB dL(u)\right)\right) \\
%&\quad = \: e_v^\top  \BFE_p^\top  \left( e^{\BA h} X(t) + \int_t^{t+h} e^{\BA (t+h-u)} \BB dL(u)\right) - e_v^\top  \BFE_p^\top  e^{\BA h} X(t) \\
& \quad = \: e_v^\top  \BFE_p^\top  \int_t^{t+h} e^{\BA (t+h-u)} \BB dL(u) \quad \mathbb{P}\text{-a.s.}  \qedhere
\end{align*}
\end{proof}

\subsection{Proofs of Subsection~\ref{sec:charMCAR}} \label{subsec:proofcharMCAR}

\begin{proof}[Proof of \Cref{Erste Charakterisierung gerichtete Kante für CAR}] $\mbox{}$\\
%We perform this proof based on the characterisation of the directed edges in \Cref{Charakterisierung als Gleichheit der linearen Vorhersage}.
(a) \, Recall that, due to \Cref{Charakterisierung als Gleichheit der linearen Vorhersage}, $\CY_a \nrarrow \CY_b \: \vert \: \CY_V$ if and only if,
%$\CY_a$ is (linear) Granger non-causal for $\CY_b$ with respect to $\CY_V$. That is
\begin{align*}
&P_{\mathcal{L}_{Y_V}(t)}Y_b(t+h) = P_{\mathcal{L}_{Y_{V\setminus\{a\}}}(t)}Y_b(t+h)  \quad \text{$\mathbb{P}$-a.s.  $\forall\,h \in [0,1]$, $t\in \R$.}
\end{align*}
From \Cref{Projektionen CAR berechnet} we know that
%know what the two orthogonal projections look like. We have
\begin{align*}
P_{\mathcal{L}_{Y_V}(t)}Y_b(t+h) = & \sum_{j=1}^p \sum_{s \in V} e_b^\top  \BFC e^{\BA h}  \BFE_j e_{s} D^{(j-1)}Y_s(t),\\
P_{\mathcal{L}_{Y_{V\setminus\{a\}}}(t)}Y_b(t+h) = & \sum_{j=1}^p \sum_{s \in V\setminus \{a\}} e_b^\top  \BFC e^{\BA h}  \BFE_j e_{s} D^{(j-1)}Y_s(t) \\
  & + \sum_{j=1}^p e_b^\top  \BFC e^{\BA h}  \BFE_j e_{a} P_{\mathcal{L}_{Y_{ V\setminus \{a\}}}(t)} D^{(j-1)}Y_a(t) \quad \forall\, h \in [0,1],\, t\in \R.
\end{align*}
 We equate the two orthogonal projections and remove the coinciding terms. Then we receive $\CY_a \nrarrow \CY_b \: \vert \: \CY_V$ if and only if
\begin{align*}
 \sum_{j=1}^p  e_b^\top  \BFC e^{\BA h} \BFE_j e_{a} D^{(j-1)} Y_a(t)
= \sum_{j=1}^p e_b^\top  \BFC e^{\BA h} \BFE_j e_{a} P_{\mathcal{L}_{Y_ {V\setminus \{a\}}}(t)} D^{(j-1)} Y_a(t) \quad \mathbb{P}\text{-a.s.}
\end{align*}
for $h \in [0,1]$, $t\in \R$.  The expression on the right side is in $\mathcal{L}_{Y_ {V\setminus \{a\}}}(t)$ and the expression on the left side is in $\mathcal{L}_{Y_a}(t)$. Due to their equality, they are in $\mathcal{L}_{Y_ {V\setminus \{a\}}}(t) \cap \mathcal{L}_{Y_a}(t) = \{0 \}$, making use of \Cref{Eigenschaften der linearen Räume}. Thus, $\CY_a \nrarrow \CY_b \: \vert \: \CY_V$ if and only if
\begin{align} \label{eqaaa}
\sum_{j=1}^p e_b^\top  \BFC e^{\BA h} \BFE_j e_{a} D^{(j-1)}Y_a(t)  =0 \quad \mathbb{P}\text{-a.s.} \quad \forall\, h \in [0,1],\, t\in \R.
\end{align}
In the following, we show that \eqref{eqaaa} %of the directed edges
is  equivalent to
\begin{align} \label{equi1}
e_b^\top  \BFC e^{\BA h} \BFE_j e_{a} =0 \quad \forall\, h \in [0,1], \: j=1,\ldots,p.
\end{align}
Clearly, \eqref{equi1} implies \eqref{eqaaa}.
%For the first direction of \eqref{equi1}, let $e_b^\top  \BFC e^{\BA h} \BFE_j e_{a}=0$ for $h \in [0,1]$, $j=1,\ldots,p$. Of course,
%\begin{align*}
%\sum_{j=1}^p e_b^\top  \BFC e^{\BA h} \BFE_j e_{a} D^{(j-1)}Y_a(t)  =0, \quad  h \in [0,1],\, t\in \R,
%\end{align*}
%as claimed.
For the opposite direction, suppose \eqref{eqaaa} holds.
%\begin{align*}
%\sum_{j=1}^p e_b^\top  \BFC e^{\BA h} \BFE_j e_{a} D^{(j-1)}Y_a(t)  =0, \quad  h \in [0,1],\, t\in \R.
%\end{align*}
 Define the $kp$-dimensional vector $\by=(y_1,\ldots,y_{kp})$  with entries
\begin{align*}
y_i=
\begin{cases}
e_b^\top  \BFC e^{\BA h} \BFE_j e_{a} &\mbox{if } i=(j-1)k+a, \: j=1,\ldots,p, \\
0 & \mbox{else}. \end{cases}
\end{align*}
 Then \eqref{eqaaa} implies $\mathbb{P}$-a.s.
\begin{align*}
0
= \sum_{j=1}^p e_b^\top  \BFC e^{\BA h} \BFE_j e_{a} D^{(j-1)}Y_a(t)
= \sum_{j=1}^p e_b^\top  \BFC e^{\BA h} \BFE_j e_{a} X_{(j-1)k+a}(t)
= \by^\top  X(t)
\end{align*}
and, in particular,
\begin{align*}
0 =  \BE \left[ \left(\by^\top  X(t)\right)^2 \right] = \by^\top   c_{XX}(0) \by.
\end{align*}
But $c_{XX}(0)>0$ (cf.~\Cref{Remark 5.4} (a)) such that $\by$ is the zero vector and  \eqref{equi1} is valid. \\
%as claimed
%\begin{align*}
%e_b^\top  \BFC e^{\BA h} \BFE_j e_{a} = 0 , \quad h \in [0,1], \: j=1,\ldots,p.
%\end{align*}
(b) \, Let $S\subseteq V$, $v\in V$, $t\in\R$, and $h\geq 0$. From \Cref{Projektionen mit lim} we already know that
\begin{align*}
& \frac{1}{h}P_{\mathcal{L}_{Y_{S}}(t)} \left(D^{(p-1)} Y_v(t+h)- D^{(p-1)} Y_v(t)\right) \\
&\quad =\:  \sum_{j=1}^p \sum_{s \in S} e_{v}^\top  \BFE_p^\top  \frac{\left( e^{\BA h} -I_{kp} \right)}{h} \BFE_j e_s  D^{(j-1)} Y_s(t) \\
&\quad\quad\: +  \sum_{j=1}^p \sum_{s \in V\setminus S}  e_{v}^\top  \BFE_p^\top  \frac{\left( e^{\BA h} -I_{kp} \right)}{h} \BFE_j e_s  P_{\mathcal{L}_{Y_{S}}(t)}( D^{(j-1)} Y_s(t))
\quad\mathbb{P}\text{-a.s.}
\end{align*}
But
$
\lim_{h \rightarrow 0} \left( e^{\BA h} -I_{kp} \right)/h =\BA
$
implies that
\begin{align*}
& \: \limh P_{\mathcal{L}_{Y_{S}}(t)}\left( \frac{D^{(p-1)} Y_v(t+h)- D^{(p-1)} Y_v(t)}{h}\right)\\
&\quad=\: \sum_{j=1}^p \sum_{s \in S} e_{v}^\top  \BFE_p^\top  \BA \BFE_j e_s  D^{(j-1)} Y_s(t)  + \sum_{j=1}^p \sum_{s \in V\setminus S}  e_{v}^\top  \BFE_p^\top  \BA \BFE_j e_s  P_{\mathcal{L}_{Y_{S}}(t)} D^{(j-1)} Y_s(t).
\end{align*}
Then the remaining proof is similar to the proof of (a).
\end{proof}

\begin{proof}[Proof of \Cref{Einfache Charakterisierung ungerichtete Kante für CAR}]
$\mbox{}$ \\
(a) \, % From \Cref{Projektionen CAR berechnet} we already know that for $v\in V$, $t\in\R$, $h,\tilde h \geq0$,
%\begin{align*}
%P_{\mathcal{L}_{Y_V}(t)} Y_v(t+h) = e_v^\top  \BFC e^{\BA h} X(t), % \quad
%P_{\mathcal{L}_{Y_V}(t)} Y_b(t+\tilde{h}) = %e_b^\top  \BFC e^{\BA \tilde{h}} X(t).
%\end{align*}
A combination of \Cref{Projektion für MCAR für S=V}  and \Cref{Lemma 5.2} (a) results in
\begin{align*}
Y_v(t+h) - P_{\mathcal{L}_{Y_V}(t)} Y_v(t+h)
%&= e_a^\top  \BFC  \left( e^{\BA h} X(t) + \int_t^{t+h} e^{\BA (t+h-u)} \BB dL(u) \right) - e_a^\top  \BFC e^{\BA h} X(t) \\
&= e_v^\top  \BFC \int_t^{t+h} e^{\BA (t+h-u)} \BB  dL(u). %, \\
%Y_b(t+\tide{h}) - P_{\mathcal{L}_{Y_V}(t)} Y_b(t+\tilde{h})
%&&= e_b^\top  \BFC \int_t^{t+\tilde{h}} e^{\BA %(t+\tilde{h}-u)} \BB dL(u),
\end{align*}
Thus, %due to \Cref{Charakterisisierung contemporaneously uncorrelated} (d),
$\CY_a \nsim \CY_b \: \vert \: \CY_V$  if and only if
\begin{align*}
 0= &\: \BE \left[\left(Y_a(t+h) - P_{\mathcal{L}_{Y_V}(t)} Y_a(t+h) \right)
		  \left(Y_b(t+\tilde{h}) - P_{\mathcal{L}_{Y_V}(t)} Y_b(t+\tilde{h}) \right) \right] \\
  =&\:	\BE \left[\left(e_a^\top  \BFC \int_t^{t+h} e^{\BA (t+h-u)} \BB dL(u) \right)
		  \left(e_b^\top  \BFC \int_t^{t+\tilde{h}} e^{\BA (t+\tilde{h}-u)} \BB  dL(u) \right) \right] \\
%  =&\: e_a^\top  \BFC \int_t^{\min(t+h,t+\tilde{h})} e^{\BA (t+h-s)} \BB  \BS_L \BB^\top  e^{\BA^\top (t+\tilde{h}-s)} ds \: \BFC^\top  e_b \\
  =&\: e_a^\top  \BFC \int_0^{\min(h,\tilde{h})} e^{\BA (h-u)} \BB \BS_L \BB^\top  e^{\BA^\top (\tilde{h}-u)} du \: \BFC^\top  e_b
\end{align*}
%where we substitute $t-s$ by $-s$ in the last step.
for $h, \tilde{h} \in [0,1]$, $t\in \R$. \\
(b) \, Let $a,b,v\in V$, $t\in\R$, and $h\geq 0$. An application of \Cref{Projektionen mit lim} gives that
\begin{align*}
D^{(p-1)} Y_v(t+h)- P_{\mathcal{L}_{Y_{V}}(t)} D^{(p-1)} Y_v(t+h)
= e_v^\top  \BFE_p^\top  \int_t^{t+h} e^{\BA (t+h-u)} \BB dL(u) \quad\mathbb{P}\text{-a.s.}
\end{align*}
Thus,
\begin{align*}
&\: \BE \left[\left( D^{(p-1)} Y_a(t+h)- P_{\mathcal{L}_{Y_{V}}(t)} D^{(p-1)} Y_a(t+h) \right)\right. \\
	&\quad\quad\quad \: \times   \left.\overline{\left( D^{(p-1)} Y_b(t+h)- P_{\mathcal{L}_{Y_{V}}(t)} D^{(p-1)} Y_b(t+h) \right)} \right] \\
		% &\quad = \: e_a^\top  \BFE_p^\top  \int_0^{h}	e^{\BA (h-u)} \BB \BS_L \BB^\top  e^{\BA^\top (h-u)} du \: \BFE_p e_b \\
		  &\quad = \: e_a^\top  \BFE_p^\top    \int_0^{h}	e^{\BA u} \BB \BS_L \BB^\top  e^{\BA^\top  u} du \:\BFE_p e_b.
\end{align*}
Setting $f(u)=e^{\BA u} \BB \BS_L \BB^\top  e^{-\BA^\top  u}$ and $F(\cdot)$ as its primitive function, we obtain
\begin{align*}
&\:\limhh \frac{1}{h} \BE \left[\left( D^{(p-1)} Y_a(t+h)- P_{\mathcal{L}_{Y_{V}}(t)} D^{(p-1)} Y_a(t+h) \right) \right. \\
		&\quad\quad\quad \: \times  \left. \overline{\left( D^{(p-1)} Y_b(t+h)- P_{\mathcal{L}_{Y_{V}}(t)} D^{(p-1)} Y_b(t+h) \right)} \right] \\
%&\quad =\:\limhh \frac{1}{h} e_a^\top  \BFE_p^\top  e^{\BA h}  \int_0^{h}	e^{-\BA s} \BB \BS_L \BB^\top  e^{-\BA^\top  s} ds \:e^{\BA^\top h} \BFE_p e_b \\
&\quad =\:e_a^\top  \BFE_p^\top  \left[\limhh \frac{F(h) -F(0)}{h} \:\right] \BFE_p e_b\\
%&\quad =\:\lim_{h \rightarrow 0} e_a^\top  \BFE_p^\top  e^{\BA h} \frac{F(h) -F(0)}{h} e^{\BA^\top h} \BFE_p e_b \\
%&\quad =\: e_a^\top  \BFE_p^\top  f(0) \BFE_p e_b \\
%&\quad =\: e_a^\top  \BFE_p^\top  \BB \BS_L \BB^\top  \BFE_p e_b \\
&\quad =\: e_a^\top  \BS_L e_b. \qedhere
\end{align*}
\end{proof}

\begin{proof}[Proof of \Cref{Zweite Charakterisierung gerichtete Kante für CAR}] $\mbox{}$\\
(a) \, $\Leftarrow$: \,
Suppose $e_b^\top  \BFC \BA^\alpha \BFE_j e_a=0$ for $\alpha=1,\ldots,kp-1$ and $j=1,\ldots,p$. \cite{BE09}, (11.2.1) provides %arepresentation of the matrix exponential in terms of the first $kp-1$ powers of $\BA$. For $h \in \R$
\begin{align} \label{eqababa}
e^{\BA h} = \sum_{\alpha=0}^{kp-1} \psi_\alpha(h) \BA^\alpha,
\quad h \in \R,
\end{align}
 where
\begin{align*}
\psi_\alpha(h) = \frac{1}{2 \pi i} \oint_{\mathcal{C}} \frac{\chi_\BA^{[\alpha+1]}(z)}{\chi_\BA(z)}e^{tz}dz,
\end{align*}
$\chi_\BA^{[1]}(\cdot),\ldots,\chi_\BA^{[kp]}(\cdot)$ are polynomials defined by recursion and $\mathcal{C}$ is a simple, closed contour in the complex plane enclosing $\sigma(\BA)$. With $e_b^\top  \BFC \BA^\alpha \BFE_j e_a=0$ we can conclude then that
\begin{align*}
e_b^\top  \BFC e^{\BA h} \BFE_j e_a = \sum_{\alpha=0}^{kp-1} \psi_\alpha(h) e_b^\top  \BFC \BA^\alpha \BFE_je_a=0 \quad  \forall\, 
h \in [0,1],
\end{align*}
such that \Cref{Erste Charakterisierung gerichtete Kante für CAR} results in $\CY_a \nrarrow \CY_b \: \vert \: \CY_V$.

$\Rightarrow$:
Assume $\CY_a \nrarrow \CY_b \: \vert \: \CY_V$. Thus, $e_b^\top  \BFC e^{\BA h} \BFE_j e_a=0$ for $h \in [0,1]$ and $j=1,\ldots,p$ by \Cref{Erste Charakterisierung gerichtete Kante für CAR}. Define
\begin{align*}
f(h)
= e_b^\top  \BFC e^{\BA h} \BFE_j e_{a}, \quad h\in\R,
\end{align*}
 and differentiate this function using \cite{BE09}, Proposition 11.1.4. Then
\begin{align*}
f^{(\alpha)}(h)
= e_b^\top  \BFC \BA^\alpha e^{\BA h} \BFE_j e_{a}, \quad h\in\R,
\,\alpha = 1,\ldots,kp-1.
\end{align*}
 Since $f(h)=0$ for $h \in [0,1]$ and $f^{(\alpha)}(\cdot)$ is continuous, we obtain $f^{(\alpha)}(h)=0$ for \mbox{$h \in [0,1]$.} Putting $h=0$, we get as claimed
\begin{equation*}
0
%=  e_b^\top  \BA^\alpha e_{k(j-1)+a}
= e_b^\top  \BFC \BA^\alpha \BFE_j e_a, \quad \alpha=1,\ldots,kp-1, \: j=1,\ldots,p.
\end{equation*}
(b) \, $\Leftarrow$: Let $e_a^\top  \BFC \BA^\alpha \BB \BS_L \BB^\top  (\BA^\top )^\beta\BFC^\top  e_b=0$ for $\alpha,\beta=0,\ldots,kp-1$. We apply the representation of the matrix exponential \eqref{eqababa} and  obtain
\begin{align*}
 &\: e_a^\top  \BFC \int_0^{\min(h,\tilde{h})}	e^{\BA (h-s)} \BB \BS_L \BB^\top  e^{\BA^\top (\tilde{h}-s)} ds \: \BFC^\top  e_b \\
%&\quad =\: e_a^\top  \BFC \int_0^{\min(h,\tilde{h})} \left(\sum_{\alpha=0}^{kp-1} \psi_\alpha(h-s) \BA^\alpha \right) \BB \BS_L \BB^\top 
%  \left(\sum_{\beta=0}^{kp-1} \varphi_\beta(\tilde{h}-s) \left(\BA^\top  \right)^\beta \right)ds \:\BFC^\top  e_b \\	
&\quad =\: \sum_{\alpha=0}^{kp-1} \sum_{\beta=0}^{kp-1} \int_0^{\min(h,\tilde{h})} \psi_\alpha(h-s) \varphi_\beta(\tilde{h}-  s) e_a^\top   \BFC \BA^\alpha \BB \BS_L \BB^\top  \left(\BA^\top  \right)^\beta \BFC^\top  e_b \: ds
=0,
\end{align*}
for $h, \tilde{h} \in [0,1]$, $t\in \R$, by assumption. \Cref{Einfache Charakterisierung ungerichtete Kante für CAR} yields then $\CY_a \nsim \CY_b \: \vert \: \CY_V$.

$\Rightarrow$: Assume $\CY_a \nsim \CY_b \: \vert \: \CY_V$. Due to \Cref{Charakterisierung als Gleichheit der linearen Vorhersage 2} we have for $h \in [0,1]$ and $t \in \R$,
\begin{align*}
P_{\mathcal{L}_{Y_{V}}(t)\vee \mathcal{L}_{Y_{b}}(t,t+1)} Y_a(t+h)  = P_{\mathcal{L}_{Y_{V}}(t)} Y_a(t+h)\quad \mathbb{P}\text{-a.s.}
\end{align*}
  In addition, we know from \Cref{Projektionen CAR  berechnet} that
$
P_{\mathcal{L}_{Y_{V}}(t)} Y_a(t+h)  = e_a^\top  \BFC e^{\BA h} X(t).
$
Both together provide
\begin{align} \label{eq5.1}
P_{\mathcal{L}_{Y_{V}}(t)\vee \mathcal{L}_{Y_{b}}(t,t+1)} Y_a(t+h)  =  e_a^\top  \BFC e^{\BA h} X(t)\quad \mathbb{P}\text{-a.s.}
\end{align}
for $h \in [0,1]$ and $t \in \R$. Since $Y_b(t+\tilde{h}) \in \mathcal{L}_{Y_{V}}(t)\vee \mathcal{L}_{Y_{b}}(t,t+1)$ for $\tilde{h} \in [0,1]$ as well as $Y_a(t+h) - P_{\mathcal{L}_{Y_{V}}(t)\vee \mathcal{L}_{Y_{b}}(t,t+1)} Y_a(t+h) \in ( \mathcal{L}_{Y_{V}}(t)\vee \mathcal{L}_{Y_{b}}(t,t+1))^{\bot}$, we obtain
\begin{align*}
0 & = \BE \left[ \left(Y_a(t+h) - P_{\mathcal{L}_{Y_{V}}(t) \vee \mathcal{L}_{Y_{b}}(t,t+1)} Y_a(t+h) \right)Y_b(t+\tilde{h}) \right].
%\end{align*}
%for $0 \leq h, \tilde{h} \leq 1$.
\intertext{Plugging in \eqref{eq5.1} gives}
%\begin{align*}
%0 &= \BE \left[ \left(Y_a(t+h) - P_{\mathcal{L}_{Y_{V}}(t) \vee \mathcal{L}_{Y_{b}}(t,t+1)} Y_a(t+h) \right)Y_b(t+\tilde{h}) \right] \\
  0&= \BE \left[ \left(Y_a(t+h) - e_a^\top  \BFC e^{\BA h} X(t) \right)Y_b(t+\tilde{h}) \right] \\
  &= e_a^\top  \BFC \BE \left[ \left(X(t+h) - e^{\BA h} X(t) \right)X(t+\tilde{h}) \right] \BFC^\top  e_b \\
  &= e_a^\top  \BFC \left( c_{XX}(h-\tilde{h}) - e^{\BA h} c_{XX}(-\tilde{h}) \right) \BFC^\top  e_b,
%\end{align*}
\intertext{for $h, \tilde{h} \in [0,1]$. If we only consider the case {$0 \leq \tilde{h} \leq h \leq 1$}  then \eqref{alternative representation of covariance for MCARMA} provides}
%   &= e_a^\top  \BFC \left( c_{XX}(h-\tilde{h}) - e^{\BA h} c_{XX}(-\tilde{h}) \right) \BFC^\top  e_b \\
   0 &= e_a^\top  \BFC \left( e^{\BA (h -\tilde{h})} c_{XX}(0) - e^{\BA h} c_{XX}(0) e^{\BA^\top  \tilde{h}} \right) \BFC^\top  e_b \\
    &= e_a^\top  \BFC e^{\BA h} \left( e^{-\BA \tilde{h}} c_{XX}(0) - c_{XX}(0) e^{\BA^\top  \tilde{h}} \right) \BFC^\top  e_b,
\end{align*}
using \cite{BE09}, Corollary 11.1.6. Now, we define
\begin{align*}
\gamma (h, \tilde{h}) = e_a^\top  \BFC e^{\BA h} \left( e^{-\BA \tilde{h}} c_{XX}(0) - c_{XX}(0) e^{\BA^\top  \tilde{h}} \right) \BFC^\top  e_b, \quad 0 \leq \tilde{h} \leq h \leq 1.
\end{align*}
Differentiating this function several times  (cf.~\citealp{BE09}, Proposition 11.1.4) provides
\begin{align*}
\frac{\partial^m}{\partial h^m} \frac{\partial^n}{\partial \tilde{h}^n}  \gamma(h, \tilde{h})
= e_a^\top  \BFC \BA^m e^{\BA h} \left( \left(-\BA \right)^n e^{-\BA \tilde{h}} c_{XX}(0) - c_{XX}(0) \left(\BA^\top \right)^n e^{\BA^\top  \tilde{h}} \right) \BFC^\top  e_b.
\end{align*}
 Furthermore, since $\gamma (h,\tilde{h})=0$ for $0 \leq \tilde{h} \leq h \leq 1$ and due to the continuity of the function under consideration, we obtain that the derivatives are zero for $0 \leq \tilde{h} \leq h \leq 1$. Now, plugging in  $h= \tilde{h}=0$ yields
\begin{align} \label{eqasa}
 e_a^\top  \BFC \BA^m  c_{XX}(0) \left(\BA^\top \right) ^n  \BFC^\top  e_b=e_a^\top  \BFC \BA^m  \left(-\BA \right)^n c_{XX}(0)  \BFC^\top  e_b, \quad m,n\in \N_0.
\end{align}
 Finally,  \eqref{Dichtezusammenhang} leads to
\begin{align*}
&\: e_a^\top  \BFC \BA^\alpha \BB \BS_L \BB^\top  \left(\BA^\top  \right)^\beta \BFC^\top  e_b \\
&\quad = \: e_a^\top  \BFC \BA^\alpha \left( -\BA c_{XX}(0) - c_{XX}(0) \BA^\top  \right) \left(\BA^\top  \right)^\beta \BFC^\top  e_b \\
&\quad = \: - e_a^\top  \BFC \BA^{\alpha+1}  c_{XX}(0) \left(\BA^\top  \right)^\beta \BFC^\top  e_b
  - e_a^\top  \BFC \BA^\alpha c_{XX}(0) \left(\BA^\top  \right)^{\beta+1} \BFC^\top  e_b.
\intertext{Applying \eqref{eqasa} gives then}
&\: e_a^\top  \BFC \BA^\alpha \BB \BS_L \BB^\top  \left(\BA^\top  \right)^\beta \BFC^\top  e_b \\
&\quad = \: - e_a^\top  \BFC (-1)^{\beta}  \BA^{\alpha + \beta + 1} c_{XX}(0) \BFC^\top  e_b
  - e_a^\top  \BFC (-1)^{\beta+1}  \BA^{\alpha + \beta + 1} c_{XX}(0) \BFC^\top  e_b
= 0,
\end{align*}
for $\alpha, \beta = 0,\ldots,kp-1$, the desired statement.
\end{proof}

\bibliographystyle{imsart-nameyear}
\bibliography{Literatur}

\begin{thebibliography}{77}
% BibTex style file: imsart-nameyear.bst, 2017-11-03
% Default style options (sort=1,type=nameyear).
% Used options (sort=1,type=nameyear).

\bibitem[\protect\citeauthoryear{Aalen}{1987}]{Aalen87}
\begin{barticle}[author]
\bauthor{\bsnm{Aalen},~\bfnm{Odd~O.}\binits{O.~O.}}
(\byear{1987}).
\btitle{Dynamic modelling and causality}.
\bjournal{Scand. Actuar. J.}
\bvolume{1987}
\bpages{177-190}.
\end{barticle}
\endbibitem

\bibitem[\protect\citeauthoryear{Andersson, Madigan and Perlman}{2001}]{AN00}
\begin{barticle}[author]
\bauthor{\bsnm{Andersson},~\bfnm{Steen~A.}\binits{S.~A.}},
  \bauthor{\bsnm{Madigan},~\bfnm{David}\binits{D.}} \AND
  \bauthor{\bsnm{Perlman},~\bfnm{Michael~D.}\binits{M.~D.}}
(\byear{2001}).
\btitle{Alternative {M}arkov properties for chain graphs}.
\bjournal{Scand. Stat. Theory Appl.}
\bvolume{28}
\bpages{33-85}.
\end{barticle}
\endbibitem

\bibitem[\protect\citeauthoryear{Bergmann and Hartwigsen}{2021}]{Bergmann21}
\begin{barticle}[author]
\bauthor{\bsnm{Bergmann},~\bfnm{T.~O.}\binits{T.~O.}} \AND
  \bauthor{\bsnm{Hartwigsen},~\bfnm{G.}\binits{G.}}
(\byear{2021}).
\btitle{Inferring causality from noninvasive brain stimulation in cognitive
  neuroscience}.
\bjournal{J. Cogn. Neurosci.}
\bvolume{33}
\bpages{195-224}.
\end{barticle}
\endbibitem

\bibitem[\protect\citeauthoryear{Bergstrom}{1997}]{Bergstrom:1997}
\begin{barticle}[author]
\bauthor{\bsnm{Bergstrom},~\bfnm{A.~R.}\binits{A.~R.}}
(\byear{1997}).
\btitle{Gaussian estimation of mixed-order continuous-time dynamic models with
  unobservable stochastic trends from mixed stock and flow data}.
\bjournal{Econom. Theory}
\bvolume{13}
\bpages{467-505}.
\end{barticle}
\endbibitem

\bibitem[\protect\citeauthoryear{Bernstein}{2009}]{BE09}
\begin{bbook}[author]
\bauthor{\bsnm{Bernstein},~\bfnm{Dennis~S.}\binits{D.~S.}}
(\byear{2009}).
\btitle{Matrix Mathematics: Theory, Facts, and Formulas},
\bedition{2.} ed.
\bpublisher{Princeton University Press}, \baddress{Princeton}.
\end{bbook}
\endbibitem

\bibitem[\protect\citeauthoryear{Bhatia}{1997}]{BH97}
\begin{bbook}[author]
\bauthor{\bsnm{Bhatia},~\bfnm{Rajendra}\binits{R.}}
(\byear{1997}).
\btitle{Matrix Analysis}.
\bpublisher{Springer}, \baddress{New York}.
\end{bbook}
\endbibitem

\bibitem[\protect\citeauthoryear{Brillinger}{2001}]{BR01}
\begin{bbook}[author]
\bauthor{\bsnm{Brillinger},~\bfnm{David~R.}\binits{D.~R.}}
(\byear{2001}).
\btitle{Time Series: Data Analysis and Theory}.
\bpublisher{Society for Industrial and Applied Mathematics},
  \baddress{Philadelphia}.
\end{bbook}
\endbibitem

\bibitem[\protect\citeauthoryear{Brockwell}{2014}]{BR14}
\begin{barticle}[author]
\bauthor{\bsnm{Brockwell},~\bfnm{Peter}\binits{P.}}
(\byear{2014}).
\btitle{Recent results in the theory and applications of CARMA processes}.
\bjournal{Ann. Inst. Stat. Math.}
\bvolume{66}
\bpages{647–685}.
\end{barticle}
\endbibitem

\bibitem[\protect\citeauthoryear{Brockwell and Davis}{1991}]{BR91}
\begin{bbook}[author]
\bauthor{\bsnm{Brockwell},~\bfnm{Peter~J.}\binits{P.~J.}} \AND
  \bauthor{\bsnm{Davis},~\bfnm{Richard~A.}\binits{R.~A.}}
(\byear{1991}).
\btitle{Time Series: Theory and Methods},
\bedition{2.} ed.
\bpublisher{Springer}, \baddress{New York}.
\end{bbook}
\endbibitem

\bibitem[\protect\citeauthoryear{Brockwell and Lindner}{2015}]{BR15}
\begin{barticle}[author]
\bauthor{\bsnm{Brockwell},~\bfnm{P.}\binits{P.}} \AND
  \bauthor{\bsnm{Lindner},~\bfnm{A.}\binits{A.}}
(\byear{2015}).
\btitle{Prediction of L{é}vy-driven {CARMA} processes}.
\bjournal{J. Econom.}
\bvolume{189}
\bpages{263-271}.
\end{barticle}
\endbibitem

\bibitem[\protect\citeauthoryear{Brockwell and
  Lindner}{2024}]{BrockwellLindner2024}
\begin{bbook}[author]
\bauthor{\bsnm{Brockwell},~\bfnm{Peter~J.}\binits{P.~J.}} \AND
  \bauthor{\bsnm{Lindner},~\bfnm{Alexander~M.}\binits{A.~M.}}
(\byear{2024}).
\btitle{Continuous-Parameter Time Series}.
\bpublisher{De Gruyter}.
\end{bbook}
\endbibitem

\bibitem[\protect\citeauthoryear{Bühler and Salamon}{2018}]{BU18}
\begin{bbook}[author]
\bauthor{\bsnm{Bühler},~\bfnm{Theo}\binits{T.}} \AND
  \bauthor{\bsnm{Salamon},~\bfnm{Dietmar~A.}\binits{D.~A.}}
(\byear{2018}).
\btitle{Functional Analysis}.
\bpublisher{American Mathematical Society}, \baddress{Providence}.
\end{bbook}
\endbibitem

\bibitem[\protect\citeauthoryear{Chamberlain}{1982}]{CH82}
\begin{barticle}[author]
\bauthor{\bsnm{Chamberlain},~\bfnm{Gary}\binits{G.}}
(\byear{1982}).
\btitle{The general equivalence of Granger and Sims causality}.
\bjournal{Econometrica}
\bvolume{50}
\bpages{569-581}.
\end{barticle}
\endbibitem

\bibitem[\protect\citeauthoryear{Commenges and Gégout-Petit}{2009}]{CO07}
\begin{barticle}[author]
\bauthor{\bsnm{Commenges},~\bfnm{Daniel}\binits{D.}} \AND
  \bauthor{\bsnm{Gégout-Petit},~\bfnm{Anne}\binits{A.}}
(\byear{2009}).
\btitle{A general dynamical statistical model with causal interpretation}.
\bjournal{J. R. Stat. Soc. Ser. B.}
\bvolume{71}
\bpages{719-736}.
\end{barticle}
\endbibitem

\bibitem[\protect\citeauthoryear{Comte and Renault}{1996}]{CO96}
\begin{barticle}[author]
\bauthor{\bsnm{Comte},~\bfnm{F.}\binits{F.}} \AND
  \bauthor{\bsnm{Renault},~\bfnm{E.}\binits{E.}}
(\byear{1996}).
\btitle{Noncausality in continuous time models}.
\bjournal{Econom. Theory}
\bvolume{12}
\bpages{215-256}.
\end{barticle}
\endbibitem

\bibitem[\protect\citeauthoryear{Cox and Popken}{2015}]{Cox:Popken}
\begin{barticle}[author]
\bauthor{\bsnm{Cox},~\bfnm{L.}\binits{L.}} \AND
  \bauthor{\bsnm{Popken},~\bfnm{D.}\binits{D.}}
(\byear{2015}).
\btitle{Has reducing fine particulate matter and ozone caused reduced mortality
  rates in the United States?}
\bjournal{Ann. Epidemiol.}
\bvolume{25}
\bpages{162-73}.
\end{barticle}
\endbibitem

\bibitem[\protect\citeauthoryear{Cram{\'e}r}{1940}]{CR39}
\begin{barticle}[author]
\bauthor{\bsnm{Cram{\'e}r},~\bfnm{H.}\binits{H.}}
(\byear{1940}).
\btitle{On the theory of stationary random processes}.
\bjournal{Ann. Math.}
\bvolume{41}
\bpages{215-230}.
\end{barticle}
\endbibitem

\bibitem[\protect\citeauthoryear{Cram{\'e}r}{1961}]{CR61}
\begin{barticle}[author]
\bauthor{\bsnm{Cram{\'e}r},~\bfnm{Harald}\binits{H.}}
(\byear{1961}).
\btitle{On the structure of purely non-deterministic stochastic processes}.
\bjournal{Ark. Mat.}
\bvolume{4}
\bpages{249-266}.
\end{barticle}
\endbibitem

\bibitem[\protect\citeauthoryear{Cram{\'e}r}{1964}]{CR64}
\begin{barticle}[author]
\bauthor{\bsnm{Cram{\'e}r},~\bfnm{Harald}\binits{H.}}
(\byear{1964}).
\btitle{Stochastic processes as curves in Hilbert space}.
\bjournal{Theory Probab. Appl.}
\bvolume{9}
\bpages{169-179}.
\end{barticle}
\endbibitem

\bibitem[\protect\citeauthoryear{Cram{\'e}r}{1971}]{CR71}
\begin{bmisc}[author]
\bauthor{\bsnm{Cram{\'e}r},~\bfnm{H.}\binits{H.}}
(\byear{1971}).
\btitle{Structural and Statistical Problems for a Class of Stochastic
  Processes: The First Samuel Stanley Wilks Lecture at Princeton University,
  March 7, 1970}.
\end{bmisc}
\endbibitem

\bibitem[\protect\citeauthoryear{Didelez}{2006}]{DI06}
\begin{binproceedings}[author]
\bauthor{\bsnm{Didelez},~\bfnm{V.}\binits{V.}}
(\byear{2006}).
\btitle{Asymmetric separation for local independence graphs}.
In \bbooktitle{Proceedings of the Twenty-Second Conference on Uncertainty in
  Artificial Intelligence}
(\beditor{\bfnm{R.}\binits{R.}~\bsnm{Dechter}} \AND
  \beditor{\bfnm{T.}\binits{T.}~\bsnm{Richardson}}, eds.)
\bpages{130–137}.
\bpublisher{AUAI Press}, \baddress{Arlington}.
\end{binproceedings}
\endbibitem

\bibitem[\protect\citeauthoryear{Didelez}{2007}]{DI07}
\begin{barticle}[author]
\bauthor{\bsnm{Didelez},~\bfnm{Vanessa}\binits{V.}}
(\byear{2007}).
\btitle{Graphical models for composable finite Markov processes}.
\bjournal{Scand. J. Stat.}
\bvolume{34}
\bpages{169-185}.
\end{barticle}
\endbibitem

\bibitem[\protect\citeauthoryear{Didelez}{2008}]{DI08}
\begin{barticle}[author]
\bauthor{\bsnm{Didelez},~\bfnm{Vanessa}\binits{V.}}
(\byear{2008}).
\btitle{Graphical models for marked point processes based on local
  independence}.
\bjournal{J. R. Stat. Soc. Ser. B.}
\bvolume{70}
\bpages{245-264}.
\end{barticle}
\endbibitem

\bibitem[\protect\citeauthoryear{Doob}{1944}]{doob:1944}
\begin{barticle}[author]
\bauthor{\bsnm{Doob},~\bfnm{J.~L.}\binits{J.~L.}}
(\byear{1944}).
\btitle{The elementary {G}aussian processes}.
\bjournal{Ann. Math. Stat.}
\bvolume{15}
\bpages{229-282}.
\end{barticle}
\endbibitem

\bibitem[\protect\citeauthoryear{Doob}{1953}]{DO60}
\begin{bbook}[author]
\bauthor{\bsnm{Doob},~\bfnm{Joseph~L.}\binits{J.~L.}}
(\byear{1953}).
\btitle{Stochastic Processes},
\bedition{3.} ed.
\bpublisher{Wiley}, \baddress{New York}.
\end{bbook}
\endbibitem

\bibitem[\protect\citeauthoryear{Dufour and Renault}{1998}]{DU98}
\begin{barticle}[author]
\bauthor{\bsnm{Dufour},~\bfnm{Jean-Marie}\binits{J.-M.}} \AND
  \bauthor{\bsnm{Renault},~\bfnm{Eric}\binits{E.}}
(\byear{1998}).
\btitle{Short run and long run causality in time series: Theory}.
\bjournal{Econometrica}
\bvolume{66}
\bpages{1099-1126}.
\end{barticle}
\endbibitem

\bibitem[\protect\citeauthoryear{Eichler}{2001}]{EI01}
\begin{btechreport}[author]
\bauthor{\bsnm{Eichler},~\bfnm{M.}\binits{M.}}
(\byear{2001}).
\btitle{Markov properties for graphical time series models.}
\btype{Technical Report},
\bpublisher{University of Heidelberg},
\baddress{Heidelberg}.
\end{btechreport}
\endbibitem

\bibitem[\protect\citeauthoryear{Eichler}{2007}]{EI07}
\begin{barticle}[author]
\bauthor{\bsnm{Eichler},~\bfnm{Michael}\binits{M.}}
(\byear{2007}).
\btitle{Granger causality and path diagrams for multivariate time series}.
\bjournal{J. Econom.}
\bvolume{137}
\bpages{334-353}.
\end{barticle}
\endbibitem

\bibitem[\protect\citeauthoryear{Eichler}{2012}]{EI10}
\begin{barticle}[author]
\bauthor{\bsnm{Eichler},~\bfnm{Michael}\binits{M.}}
(\byear{2012}).
\btitle{Graphical modelling of multivariate time series}.
\bjournal{Probab. Theory Relat. Fields}
\bvolume{153}
\bpages{233–268}.
\end{barticle}
\endbibitem

\bibitem[\protect\citeauthoryear{Eichler}{2013}]{Eichler:2013a}
\begin{barticle}[author]
\bauthor{\bsnm{Eichler},~\bfnm{Michael}\binits{M.}}
(\byear{2013}).
\btitle{Causal inference with multiple time series: Principles and problems}.
\bjournal{Philos. Trans. Royal Soc. A}
\bvolume{371}.
\end{barticle}
\endbibitem

\bibitem[\protect\citeauthoryear{Eichler, Dahlhaus and
  Dueck}{2017}]{Eichler:Dahlhaus:Dueck}
\begin{barticle}[author]
\bauthor{\bsnm{Eichler},~\bfnm{Michael}\binits{M.}},
  \bauthor{\bsnm{Dahlhaus},~\bfnm{Rainer}\binits{R.}} \AND
  \bauthor{\bsnm{Dueck},~\bfnm{Johannes}\binits{J.}}
(\byear{2017}).
\btitle{Graphical modeling for multivariate {H}awkes processes with
  nonparametric link functions}.
\bjournal{J. Time Series Anal.}
\bvolume{38}
\bpages{225-242}.
\end{barticle}
\endbibitem

\bibitem[\protect\citeauthoryear{Fasen-Hartmann and Schenk}{2023}]{VF23preb}
\begin{barticle}[author]
\bauthor{\bsnm{Fasen-Hartmann},~\bfnm{V.}\binits{V.}} \AND
  \bauthor{\bsnm{Schenk},~\bfnm{L.}\binits{L.}}
(\byear{2023}).
\btitle{Mixed causality graphs for continuous-time state space models and
  orthogonal projections}.
\bjournal{arXiv:2311.04478}.
\end{barticle}
\endbibitem

\bibitem[\protect\citeauthoryear{Fasen-Hartmann and Schenk}{2024}]{VF24}
\begin{barticle}[author]
\bauthor{\bsnm{Fasen-Hartmann},~\bfnm{V.}\binits{V.}} \AND
  \bauthor{\bsnm{Schenk},~\bfnm{L.}\binits{L.}}
(\byear{2024}).
\btitle{Partial correlation graphs for continuous-parameter time series}.
\bjournal{arXiv:2401.16970}.
\end{barticle}
\endbibitem

\bibitem[\protect\citeauthoryear{Feshchenko}{2012}]{FE12}
\begin{barticle}[author]
\bauthor{\bsnm{Feshchenko},~\bfnm{I.}\binits{I.}}
(\byear{2012}).
\btitle{On closeness of the sum of $n$ subspaces of a Hilbert space}.
\bjournal{Ukr. Math. J.}
\bvolume{63}
\bpages{1566-1622}.
\end{barticle}
\endbibitem

\bibitem[\protect\citeauthoryear{Florens and Fougère}{1996}]{FL96}
\begin{barticle}[author]
\bauthor{\bsnm{Florens},~\bfnm{Jean-Pierre}\binits{J.-P.}} \AND
  \bauthor{\bsnm{Fougère},~\bfnm{D.}\binits{D.}}
(\byear{1996}).
\btitle{Noncausality in continuous time}.
\bjournal{Econometrica}
\bvolume{64}
\bpages{1195-1212}.
\end{barticle}
\endbibitem

\bibitem[\protect\citeauthoryear{Florens and Mouchart}{1982}]{FL82}
\begin{barticle}[author]
\bauthor{\bsnm{Florens},~\bfnm{Jean-Pierre}\binits{J.-P.}} \AND
  \bauthor{\bsnm{Mouchart},~\bfnm{M.}\binits{M.}}
(\byear{1982}).
\btitle{A note on noncausality}.
\bjournal{Econometrica}
\bvolume{50}
\bpages{583-591}.
\end{barticle}
\endbibitem

\bibitem[\protect\citeauthoryear{Florens and Mouchart}{1985}]{FL85}
\begin{barticle}[author]
\bauthor{\bsnm{Florens},~\bfnm{Jean-Pierre}\binits{J.-P.}} \AND
  \bauthor{\bsnm{Mouchart},~\bfnm{M.}\binits{M.}}
(\byear{1985}).
\btitle{A linear theory for noncausality}.
\bjournal{Econometrica}
\bvolume{53}
\bpages{157-175}.
\end{barticle}
\endbibitem

\bibitem[\protect\citeauthoryear{Gihman and Skorokhod}{2004}]{GI04}
\begin{bbook}[author]
\bauthor{\bsnm{Gihman},~\bfnm{I.~I.}\binits{I.~I.}} \AND
  \bauthor{\bsnm{Skorokhod},~\bfnm{A.~V.}\binits{A.~V.}}
(\byear{2004}).
\btitle{The Theory of Stochastic Processes I}.
\bpublisher{Springer}, \baddress{Berlin}.
\end{bbook}
\endbibitem

\bibitem[\protect\citeauthoryear{Gladyshev}{1958}]{GL58}
\begin{barticle}[author]
\bauthor{\bsnm{Gladyshev},~\bfnm{E.~G.}\binits{E.~G.}}
(\byear{1958}).
\btitle{On multi-dimensional stationary random processes}.
\bjournal{Theory Probab. Appl.}
\bvolume{3}
\bpages{425-428}.
\end{barticle}
\endbibitem

\bibitem[\protect\citeauthoryear{Granger}{1969}]{GR69}
\begin{barticle}[author]
\bauthor{\bsnm{Granger},~\bfnm{C.~W.~J.}\binits{C.~W.~J.}}
(\byear{1969}).
\btitle{Investigating causal relations by econometric models and cross-spectral
  methods}.
\bjournal{Econometrica}
\bvolume{37}
\bpages{424-438}.
\end{barticle}
\endbibitem

\bibitem[\protect\citeauthoryear{Halmos}{1957}]{HA51}
\begin{bbook}[author]
\bauthor{\bsnm{Halmos},~\bfnm{Paul~R.}\binits{P.~R.}}
(\byear{1957}).
\btitle{Introduction to Hilbert Space and the Theory of Spectral Multiplicity},
\bedition{2.} ed.
\bpublisher{Chelsea Publ.}, \baddress{New York}.
\end{bbook}
\endbibitem

\bibitem[\protect\citeauthoryear{Harvey and Stock}{1985a}]{Harvey:Stock:1985}
\begin{barticle}[author]
\bauthor{\bsnm{Harvey},~\bfnm{A.~C.}\binits{A.~C.}} \AND
  \bauthor{\bsnm{Stock},~\bfnm{J.~H.}\binits{J.~H.}}
(\byear{1985}a).
\btitle{The estimation of higher-order continuous time autoregressive models}.
\bjournal{Econom. Theory}
\bvolume{1}
\bpages{97-112}.
\end{barticle}
\endbibitem

\bibitem[\protect\citeauthoryear{Harvey and Stock}{1985b}]{Harvey:Stock:1988}
\begin{barticle}[author]
\bauthor{\bsnm{Harvey},~\bfnm{A.~C.}\binits{A.~C.}} \AND
  \bauthor{\bsnm{Stock},~\bfnm{J.~H.}\binits{J.~H.}}
(\byear{1985}b).
\btitle{Continuous time autoregressive models with common stochastic trends}.
\bjournal{J. Econ. Dyn. Control}
\bvolume{12}
\bpages{365–384}.
\end{barticle}
\endbibitem

\bibitem[\protect\citeauthoryear{Harvey and Stock}{1989}]{HarveyStock1989}
\begin{barticle}[author]
\bauthor{\bsnm{Harvey},~\bfnm{A.~C.}\binits{A.~C.}} \AND
  \bauthor{\bsnm{Stock},~\bfnm{James~H.}\binits{J.~H.}}
(\byear{1989}).
\btitle{Estimating integrated higher-order continuous time autoregressions with
  an application to money-income causality}.
\bjournal{J. Econom.}
\bvolume{42}
\bpages{319-336}.
\end{barticle}
\endbibitem

\bibitem[\protect\citeauthoryear{Heerah et~al.}{2021}]{Heerah}
\begin{barticle}[author]
\bauthor{\bsnm{Heerah},~\bfnm{S.}\binits{S.}},
  \bauthor{\bsnm{Molinari},~\bfnm{R.}\binits{R.}},
  \bauthor{\bsnm{Guerrier},~\bfnm{S.}\binits{S.}} \AND
  \bauthor{\bsnm{Marshall-Colon},~\bfnm{A.}\binits{A.}}
(\byear{2021}).
\btitle{Granger-causal testing for irregularly sampled time series with
  application to nitrogen signalling in Arabidopsis}.
\bjournal{Bioinformatics}
\bvolume{37}
\bpages{2450-2460}.
\end{barticle}
\endbibitem

\bibitem[\protect\citeauthoryear{Imbens}{2022}]{Imbens:2022}
\begin{barticle}[author]
\bauthor{\bsnm{Imbens},~\bfnm{Guido~W.}\binits{G.~W.}}
(\byear{2022}).
\btitle{Causality in econometrics: Choice vs. chance}.
\bjournal{Econometrica}
\bvolume{90}
\bpages{2541-2566}.
\end{barticle}
\endbibitem

\bibitem[\protect\citeauthoryear{Koster}{1999}]{KO99}
\begin{barticle}[author]
\bauthor{\bsnm{Koster},~\bfnm{Jan T.~A.}\binits{J.~T.~A.}}
(\byear{1999}).
\btitle{On the validity of the Markov interpretation of path diagrams of
  Gaussian structural equations systems with correlated errors}.
\bjournal{Scand. J. Stat.}
\bvolume{26}
\bpages{413-431}.
\end{barticle}
\endbibitem

\bibitem[\protect\citeauthoryear{Kuersteiner}{2010}]{Kuersteiner2018}
\begin{binproceedings}[author]
\bauthor{\bsnm{Kuersteiner},~\bfnm{G.~M.}\binits{G.~M.}}
(\byear{2010}).
\btitle{Granger-Sims causality}.
In \bbooktitle{Macroeconometrics and Time Series Analysis}
(\beditor{\bfnm{S.~N.}\binits{S.~N.}~\bsnm{Durlauf}} \AND
  \beditor{\bfnm{L.~E.}\binits{L.~E.}~\bsnm{Blume}}, eds.)
\bpages{5413-5425}.
\bpublisher{Palgrave Macmillan}.
\end{binproceedings}
\endbibitem

\bibitem[\protect\citeauthoryear{Kuzma, Cruickshank and Carley}{2022}]{Kuzma}
\begin{binproceedings}[author]
\bauthor{\bsnm{Kuzma},~\bfnm{R.}\binits{R.}},
  \bauthor{\bsnm{Cruickshank},~\bfnm{I.~J.}\binits{I.~J.}} \AND
  \bauthor{\bsnm{Carley},~\bfnm{K.~M.}\binits{K.~M.}}
(\byear{2022}).
\btitle{Influencing the Influencers: Evaluating Person-to-Person Influence on
  Social Networks Using Granger Causality}.
In \bbooktitle{Complex Networks \& Their Applications X}
(\beditor{\bfnm{R.~M.}\binits{R.~M.}~\bsnm{Benito}},
  \beditor{\bfnm{C.}\binits{C.}~\bsnm{Cherifi}},
  \beditor{\bfnm{H.}\binits{H.}~\bsnm{Cherifi}},
  \beditor{\bfnm{E.}\binits{E.}~\bsnm{Moro}},
  \beditor{\bfnm{L.~M.}\binits{L.~M.}~\bsnm{Rocha}} \AND
  \beditor{\bfnm{M.}\binits{M.}~\bsnm{Sales-Pardo}}, eds.)
\bpages{89-99}.
\bpublisher{Springer}, \baddress{Cham}.
\end{binproceedings}
\endbibitem

\bibitem[\protect\citeauthoryear{Lauritzen}{2004}]{LA04}
\begin{bbook}[author]
\bauthor{\bsnm{Lauritzen},~\bfnm{Steffen~L.}\binits{S.~L.}}
(\byear{2004}).
\btitle{Graphical models},
\bedition{2.} ed.
\bpublisher{Clarendon Press}, \baddress{Oxford}.
\end{bbook}
\endbibitem

\bibitem[\protect\citeauthoryear{Lemmens, Valkenburg and Peter}{2011}]{LE11}
\begin{barticle}[author]
\bauthor{\bsnm{Lemmens},~\bfnm{J.~S.}\binits{J.~S.}},
  \bauthor{\bsnm{Valkenburg},~\bfnm{P.~M.}\binits{P.~M.}} \AND
  \bauthor{\bsnm{Peter},~\bfnm{J.}\binits{J.}}
(\byear{2011}).
\btitle{The effects of pathological gaming on aggressive behavior}.
\bjournal{J. Youth Adolescence}
\bvolume{40}
\bpages{38–47}.
\end{barticle}
\endbibitem

\bibitem[\protect\citeauthoryear{Levitz, Perlman and Madigan}{2001}]{Levitz}
\begin{barticle}[author]
\bauthor{\bsnm{Levitz},~\bfnm{Michael}\binits{M.}},
  \bauthor{\bsnm{Perlman},~\bfnm{Michael~D.}\binits{M.~D.}} \AND
  \bauthor{\bsnm{Madigan},~\bfnm{David}\binits{D.}}
(\byear{2001}).
\btitle{Separation and completeness properties for {AMP} chain graph {M}arkov
  models}.
\bjournal{Ann. Statist.}
\bvolume{29}
\bpages{1751-1784}.
\end{barticle}
\endbibitem

\bibitem[\protect\citeauthoryear{Lindquist and Picci}{2015}]{LI15}
\begin{bbook}[author]
\bauthor{\bsnm{Lindquist},~\bfnm{A.}\binits{A.}} \AND
  \bauthor{\bsnm{Picci},~\bfnm{G.}\binits{G.}}
(\byear{2015}).
\btitle{Linear Stochastic Systems: A Geometric Approach to Modeling, Estimation
  and Identification}.
\bpublisher{Springer}, \baddress{Berlin, Heidelberg}.
\end{bbook}
\endbibitem

\bibitem[\protect\citeauthoryear{Maathuis et~al.}{2019}]{Handbook:graphical}
\begin{bbook}[author]
\bauthor{\bsnm{Maathuis},~\bfnm{M.}\binits{M.}},
  \bauthor{\bsnm{Drton},~\bfnm{M.}\binits{M.}},
  \bauthor{\bsnm{Lauritzen},~\bfnm{S.~L.}\binits{S.~L.}} \AND
  \bauthor{\bsnm{Wainwright},~\bfnm{M.}\binits{M.}}
(\byear{2019}).
\btitle{Handbook of Graphical Models}.
\bpublisher{CRC Press}, \baddress{Boca Raton}.
\end{bbook}
\endbibitem

\bibitem[\protect\citeauthoryear{Marquardt}{2007}]{Marquardt}
\begin{barticle}[author]
\bauthor{\bsnm{Marquardt},~\bfnm{Tina}\binits{T.}}
(\byear{2007}).
\btitle{Multivariate fractionally integrated {CARMA} processes}.
\bjournal{J. Multivariate Anal.}
\bvolume{98}
\bpages{1705--1725}.
\end{barticle}
\endbibitem

\bibitem[\protect\citeauthoryear{Marquardt and Stelzer}{2007}]{MA07}
\begin{barticle}[author]
\bauthor{\bsnm{Marquardt},~\bfnm{Tina}\binits{T.}} \AND
  \bauthor{\bsnm{Stelzer},~\bfnm{Robert}\binits{R.}}
(\byear{2007}).
\btitle{Multivariate CARMA processes}.
\bjournal{Stoch. Process. their Appl.}
\bvolume{117}
\bpages{96-120}.
\end{barticle}
\endbibitem

\bibitem[\protect\citeauthoryear{Masuda}{2004}]{Masuda:2004}
\begin{barticle}[author]
\bauthor{\bsnm{Masuda},~\bfnm{Hiroki}\binits{H.}}
(\byear{2004}).
\btitle{On multidimensional {O}rnstein-{U}hlenbeck processes driven by a
  general {L}\'{e}vy process}.
\bjournal{Bernoulli}
\bvolume{10}
\bpages{97-120}.
\end{barticle}
\endbibitem

\bibitem[\protect\citeauthoryear{Matveev}{1961}]{MA61}
\begin{barticle}[author]
\bauthor{\bsnm{Matveev},~\bfnm{R.~F.}\binits{R.~F.}}
(\byear{1961}).
\btitle{On multi-dimensional regular stationary processes}.
\bjournal{Theory Probab. Appl.}
\bvolume{6}
\bpages{149-165}.
\end{barticle}
\endbibitem

\bibitem[\protect\citeauthoryear{Mogensen and
  Hansen}{2020}]{Mogensen:Hansen:2020}
\begin{barticle}[author]
\bauthor{\bsnm{Mogensen},~\bfnm{S{\o}ren~Wengel}\binits{S.~W.}} \AND
  \bauthor{\bsnm{Hansen},~\bfnm{Niels~Richard}\binits{N.~R.}}
(\byear{2020}).
\btitle{Markov equivalence of marginalized local independence graphs}.
\bjournal{Ann. Math. Stat.}
\bvolume{48}
\bpages{539--559}.
\end{barticle}
\endbibitem

\bibitem[\protect\citeauthoryear{Mogensen and
  Hansen}{2022}]{Mogensen:Hansen:2022}
\begin{barticle}[author]
\bauthor{\bsnm{Mogensen},~\bfnm{S{\o}ren~Wengel}\binits{S.~W.}} \AND
  \bauthor{\bsnm{Hansen},~\bfnm{Niels~Richard}\binits{N.~R.}}
(\byear{2022}).
\btitle{Graphical modeling of stochastic processes driven by correlated noise}.
\bjournal{Bernoulli}
\bvolume{28}
\bpages{3023–3050}.
\end{barticle}
\endbibitem

\bibitem[\protect\citeauthoryear{Pearl}{1994}]{PE94}
\begin{bbook}[author]
\bauthor{\bsnm{Pearl},~\bfnm{J.}\binits{J.}}
(\byear{1994}).
\btitle{Probabilistic Reasoning in Intelligent Systems: Networks of Plausible
  Inference},
\bedition{3.} ed.
\bpublisher{Morgan Kaufmann}, \baddress{San Francisco}.
\end{bbook}
\endbibitem

\bibitem[\protect\citeauthoryear{Petrovic and Dimitrijevic}{2012}]{PE12}
\begin{barticle}[author]
\bauthor{\bsnm{Petrovic},~\bfnm{Ljiljana}\binits{L.}} \AND
  \bauthor{\bsnm{Dimitrijevic},~\bfnm{Sladjana}\binits{S.}}
(\byear{2012}).
\btitle{Causality with finite horizon of the past in continuous time}.
\bjournal{Stat. Probab. Lett.}
\bvolume{82}
\bpages{1219-1223}.
\end{barticle}
\endbibitem

\bibitem[\protect\citeauthoryear{Priestley}{1981}]{Priestley}
\begin{bbook}[author]
\bauthor{\bsnm{Priestley},~\bfnm{M.~B.}\binits{M.~B.}}
(\byear{1981}).
\btitle{Spectral Analysis and Time Series, Volume II}.
\bpublisher{Academic Press}, \baddress{London}.
\end{bbook}
\endbibitem

\bibitem[\protect\citeauthoryear{Renault and Szafarz}{1991}]{Renault:Szafarz}
\begin{btechreport}[author]
\bauthor{\bsnm{Renault},~\bfnm{E.}\binits{E.}} \AND
  \bauthor{\bsnm{Szafarz},~\bfnm{A.}\binits{A.}}
(\byear{1991}).
\btitle{True versus spurious instantaneous causality.}
\btype{Technical Report},
\bpublisher{Universite Libre de Bruxelles},
\baddress{Bruxelles}.
\end{btechreport}
\endbibitem

\bibitem[\protect\citeauthoryear{Richardson}{2003}]{RI03}
\begin{barticle}[author]
\bauthor{\bsnm{Richardson},~\bfnm{Thomas}\binits{T.}}
(\byear{2003}).
\btitle{Markov properties for acyclic directed mixed graphs}.
\bjournal{Scand. J. Stat.}
\bvolume{30}
\bpages{145-157}.
\end{barticle}
\endbibitem

\bibitem[\protect\citeauthoryear{Rozanov}{1967}]{RO67}
\begin{bbook}[author]
\bauthor{\bsnm{Rozanov},~\bfnm{J.~A.}\binits{J.~A.}}
(\byear{1967}).
\btitle{Stationary Random Processes}.
\bpublisher{Holden-Day}, \baddress{San Francisco}.
\end{bbook}
\endbibitem

\bibitem[\protect\citeauthoryear{Røysland
  et~al.}{2024}]{roysland2024graphical}
\begin{barticle}[author]
\bauthor{\bsnm{Røysland},~\bfnm{Kjetil}\binits{K.}},
  \bauthor{\bsnm{Ryalen},~\bfnm{Pål}\binits{P.}},
  \bauthor{\bsnm{Nygård},~\bfnm{Mari}\binits{M.}} \AND
  \bauthor{\bsnm{Didelez},~\bfnm{Vanessa}\binits{V.}}
(\byear{2024}).
\btitle{Graphical criteria for the identification of marginal causal effects in
  continuous-time survival and event-history analyses}.
\bjournal{arXiv:2202.02311}.
\end{barticle}
\endbibitem

\bibitem[\protect\citeauthoryear{Schlemm and Stelzer}{2012}]{SC122}
\begin{barticle}[author]
\bauthor{\bsnm{Schlemm},~\bfnm{Eckhard}\binits{E.}} \AND
  \bauthor{\bsnm{Stelzer},~\bfnm{Robert}\binits{R.}}
(\byear{2012}).
\btitle{Quasi maximum likelihood estimation for strongly mixing state space
  models and multivariate Lévy-driven {CARMA} processes}.
\bjournal{Electron. J. Stat.}
\bvolume{6}
\bpages{2185-2234}.
\end{barticle}
\endbibitem

\bibitem[\protect\citeauthoryear{Schweder}{1970}]{Schweder}
\begin{barticle}[author]
\bauthor{\bsnm{Schweder},~\bfnm{Tore}\binits{T.}}
(\byear{1970}).
\btitle{Composable {M}arkov processes}.
\bjournal{J. Appl. Probability}
\bvolume{7}
\bpages{400-410}.
\end{barticle}
\endbibitem

\bibitem[\protect\citeauthoryear{Shojaie and Fox}{2022}]{Shojaie:Fox}
\begin{barticle}[author]
\bauthor{\bsnm{Shojaie},~\bfnm{Ali}\binits{A.}} \AND
  \bauthor{\bsnm{Fox},~\bfnm{Emily~B.}\binits{E.~B.}}
(\byear{2022}).
\btitle{Granger Causality: A Review and Recent Advances}.
\bjournal{Annu. Rev. Stat. Appl.}
\bvolume{9}
\bpages{289-319}.
\end{barticle}
\endbibitem

\bibitem[\protect\citeauthoryear{Sims}{1972}]{SI72}
\begin{barticle}[author]
\bauthor{\bsnm{Sims},~\bfnm{Christopher~A.}\binits{C.~A.}}
(\byear{1972}).
\btitle{Money, income, and causality}.
\bjournal{Am. Econ. Rev.}
\bvolume{62}
\bpages{540-552}.
\end{barticle}
\endbibitem

\bibitem[\protect\citeauthoryear{Spirtes et~al.}{1998}]{SP98}
\begin{barticle}[author]
\bauthor{\bsnm{Spirtes},~\bfnm{P.}\binits{P.}},
  \bauthor{\bsnm{Richardson},~\bfnm{T.}\binits{T.}},
  \bauthor{\bsnm{Meek},~\bfnm{C.}\binits{C.}},
  \bauthor{\bsnm{Scheines},~\bfnm{R.}\binits{R.}} \AND
  \bauthor{\bsnm{Glymour},~\bfnm{C.}\binits{C.}}
(\byear{1998}).
\btitle{Using Path Diagrams as a Structural Equation Modeling Tool}.
\bjournal{Sociol. Methods Res.}
\bvolume{27}
\bpages{182-225}.
\end{barticle}
\endbibitem

\bibitem[\protect\citeauthoryear{Weidmann}{1980}]{WE80}
\begin{bbook}[author]
\bauthor{\bsnm{Weidmann},~\bfnm{J.}\binits{J.}}
(\byear{1980}).
\btitle{Linear Operators in Hilbert Spaces}.
\bpublisher{Springer}, \baddress{New York}.
\end{bbook}
\endbibitem

\bibitem[\protect\citeauthoryear{Whittaker}{2008}]{WI08}
\begin{bbook}[author]
\bauthor{\bsnm{Whittaker},~\bfnm{Joe}\binits{J.}}
(\byear{2008}).
\btitle{Graphical Models in Applied Multivariate Statistics},
\bedition{2.} ed.
\bpublisher{Wiley}, \baddress{Chichester}.
\end{bbook}
\endbibitem

\bibitem[\protect\citeauthoryear{Wright}{1921}]{Wright1921}
\begin{barticle}[author]
\bauthor{\bsnm{Wright},~\bfnm{Sewall}\binits{S.}}
(\byear{1921}).
\btitle{Correlation and causation}.
\bjournal{J. Agric. Res.}
\bvolume{20}
\bpages{557-585}.
\end{barticle}
\endbibitem

\bibitem[\protect\citeauthoryear{Wright}{1934}]{Wright1934}
\begin{barticle}[author]
\bauthor{\bsnm{Wright},~\bfnm{Sewall}\binits{S.}}
(\byear{1934}).
\btitle{The method of path coefficients}.
\bjournal{Ann. Math. Statist.}
\bvolume{5}
\bpages{161-215}.
\end{barticle}
\endbibitem

\bibitem[\protect\citeauthoryear{Zadnik et~al.}{2000}]{ZA00}
\begin{barticle}[author]
\bauthor{\bsnm{Zadnik},~\bfnm{K.}\binits{K.}},
  \bauthor{\bsnm{Jones},~\bfnm{L.}\binits{L.}},
  \bauthor{\bsnm{Irvin},~\bfnm{B.}\binits{B.}},
  \bauthor{\bsnm{Kleinstein},~\bfnm{R.~N.}\binits{R.~N.}},
  \bauthor{\bsnm{Manny},~\bfnm{R.~E.}\binits{R.~E.}},
  \bauthor{\bsnm{Shin},~\bfnm{J.~A.}\binits{J.~A.}} \AND
  \bauthor{\bsnm{Mutti},~\bfnm{D.~O.}\binits{D.~O.}}
(\byear{2000}).
\btitle{Myopia and ambient night-time lighting}.
\bjournal{Nature}
\bvolume{404}
\bpages{143–144}.
\end{barticle}
\endbibitem

\end{thebibliography}

%\unappendix

%\renewcommand\appendixname{Supplementary Material}

\newpage

\thispagestyle{empty}
\setcounter{page}{1}

\begin{center}

{\normalfont {\normalsize{S}\footnotesize{UPPLEMENTARY} \normalsize{M}\footnotesize{ATERIAL FOR}}} \\ [6mm]

{\normalfont \bfseries{\large MIXED ORTHOGONALITY GRAPHS FOR \vspace*{0.2cm}  \\ \vspace*{0.2cm} CONTINUOUS-TIME STATIONARY PROCESSES}}\\[6mm]

{\normalfont  {\normalsize{B}\footnotesize{Y} \normalsize{V}\footnotesize{ICKY} \normalsize{F}\footnotesize{ASEN}-\normalsize{H}\footnotesize{ARTMANN AND} \normalsize{L}\footnotesize{EA} \normalsize{S}\footnotesize{CHENK}}}
\end{center}

\bigskip

\section{Proofs of Sections \ref{sec:prelim} and \ref{sec:influence}} \label{Sec:Suppl}

\begin{proof}[Proof of \Cref{additivity in time domain}] $\mbox{}$\\
(a) \, First of all, $\mathcal{L}_{Y_A}(s)\subseteq \mathcal{L}_{Y_A}(t)$ and $\mathcal{L}_{Y_A}(s,t)\subseteq \mathcal{L}_{Y_A}(t)$ by definition of the linear spaces and hence, $\mathcal{L}_{Y_A}(s) + \mathcal{L}_{Y_A}(s,t) \subseteq \mathcal{L}_{Y_A}(t)$, since $\mathcal{L}_{Y_A}(t)$ is a linear space. As $\mathcal{L}_{Y_A}(t)$ is closed, the first direction $\mathcal{L}_{Y_A}(s) \vee \mathcal{L}_{Y_A}(s,t) \subseteq \mathcal{L}_{Y_A}(t)$ follows.

For the opposite subset relation, let $\HY^A \in \ell_{Y_A}(-\infty,t)$. Then there are coefficients $\gamma_{a,i}\in \C$ and time points $-\infty < t_1 \leq \ldots \leq t_n \leq t$, $n \in \N$, such that $\mathbb{P}$-a.s.
\begin{align*}
\HY^A
&=\sum_{i=1}^n \sum_{a\in A} \gamma_{a,i} Y_a(t_i)
=\sum_{t_i \leq s} \sum_{a\in A} \gamma_{a,i} Y_a(t_i) + \sum_{t_i > s} \sum_{a\in A} \gamma_{a,i} Y_a(t_i) \\
&\in \ell_{Y_A}(-\infty,s) +\ell_{Y_A}(s, t)
\subseteq \mathcal{L}_{Y_A}(s) \vee \mathcal{L}_{Y_A}(s,t).
\end{align*}
Thus, $\ell_{Y_A}(-\infty,t) \subseteq \mathcal{L}_{Y_A}(s) \vee \mathcal{L}_{Y_A}(s,t)$. Since the space $\mathcal{L}_{Y_A}(s) \vee \mathcal{L}_{Y_A}(s,t)$ is closed, $\mathcal{L}_{Y_A}(t) \subseteq \mathcal{L}_{Y_A}(s) \vee \mathcal{L}_{Y_A}(s,t) $ follows.\\
(b,c,d) \, The proofs are very similar to the proof of (a) and therefore skipped.
\end{proof}

\begin{proof}[Proof of \Cref{Charakterisisierung linear granger non-causal}] $\mbox{}$\\
(a) $\Rightarrow$ (b): Suppose that $\mathcal{L}_{Y_B}(t, t+1) \perp \mathcal{L}_{Y_A}(t) \: \vert \: \mathcal{L}_{Y_{S \setminus A}}(t)$ for $t \in \R$. \\
\textsl{Step 1.} Let $\HY^B \in \mathcal{L}_{Y_B}(t, t+1)$. Then we obtain due to (a) that for $\HY^A\in \mathcal{L}_{Y_A}(t)$
\begin{align*}
\BE \left[ \left( \HY^B -P_{\mathcal{L}_{Y_{S \setminus A}}(t)} \HY^B  \right) \overline{ \left( \HY^A-P_{\mathcal{L}_{Y_{S \setminus A}}(t)} \HY^A \right)} \right]=0.
\end{align*}
\textsl{Step 2.}  Let $\HY^B \in \mathcal{L}_{Y_B}(t)$. Then $\HY^B \in \mathcal{L}_{Y_{S \setminus A}}(t)$
%, since $B \subseteq S \setminus A$
and $P_{\mathcal{L}_{Y_{S \setminus A}}(t)} \HY^B =\HY^B$, such that for $\HY^A\in \mathcal{L}_{Y_A}(t)$
\begin{align*}
\BE \left[ \left( \HY^B -P_{\mathcal{L}_{Y_{S \setminus A}}(t)} \HY^B  \right) \overline{\left( \HY^A-P_{\mathcal{L}_{Y_{S \setminus A}}(t)} \HY^A \right)} \right] =0.
\end{align*}
\textsl{Step 3.} Let $\HY^B \in \mathcal{L}_{Y_B}(t+1)$. We receive $\mathcal{L}_{Y_B}(t+1) =\mathcal{L}_{Y_B}(t) \vee \mathcal{L}_{Y_B}(t, t+1)$ due to \Cref{additivity in time domain} . Then there exists a sequence $\HY^B_n \in  \mathcal{L}_{Y_B}(t) + \mathcal{L}_{Y_B}(t, t+1)$, $n \in \N$, such that $ \lim_{n \rightarrow\infty} \Vert \HY^B - \HY^B_n \Vert_{L^2}=  0$. \cite{BR91}, Proposition 2.3.2 (iv)  provide that %this convergence yields to
\begin{align*}
\lim_{n \rightarrow \infty} \Vert  P_{\mathcal{L}_{Y_{S \setminus A}}(t)} \HY^B - P_{\mathcal{L}_{Y_{S \setminus A}}(t)} \HY^B_n \Vert_{L^2} = 0.
\end{align*}
Therefore, due to \eqref{limit E} we get for $\HY^A\in \mathcal{L}_{Y_A}(t)$
\begin{comment}
the additivity of $L^2$ convergence implies
\begin{align*}
\lim_{n \rightarrow \infty} \Vert \HY^B  -  Y_n^B - P_{\mathcal{L}_{Y_{S \setminus A}}(t)} \HY^B + P_{\mathcal{L}_{Y_{S \setminus A}}(t)} \HY^B_n \Vert_{L^2} = 0.
\end{align*}
With this prior knowledge and using Cauchy-Schwarz inequality, we obtain
\begin{align*}
0
\leq
&\: \vert \langle \HY^B -P_{\mathcal{L}_{Y_{S \setminus A}}(t)} \HY^B , \HY^A-P_{\mathcal{L}_{Y_{S \setminus A}}(t)} \HY^A \rangle_{L^2} \\
&\: - \langle \HY^B_n - P_{\mathcal{L}_{Y_{S \setminus A}}(t)} \HY^B_n , \HY^A-P_{\mathcal{L}_{Y_{S \setminus A}}(t)} \HY^A \rangle_{L^2} \vert \\
= &\: \vert \langle \HY^B -  \HY^B_n  -P_{\mathcal{L}_{Y_{S \setminus A}}(t)} \HY^B + P_{\mathcal{L}_{Y_{S \setminus A}}(t)} \HY^B_n , \HY^A-P_{\mathcal{L}_{Y_{S \setminus A}}(t)} \HY^A \rangle_{L^2} \vert \\
\leq &\: \Vert \HY^B -  \HY^B_n  -P_{\mathcal{L}_{Y_{S \setminus A}}(t)} \HY^B + P_{\mathcal{L}_{Y_{S \setminus A}}(t)} \HY^B_n \Vert_{L^2} \Vert \HY^A-P_{\mathcal{L}_{Y_{S \setminus A}}(t)} \HY^A \Vert_{L^2} \\
\rightarrow &\: 0,
\end{align*}
as $n \rightarrow \infty$. In the limit transition we apply the $L^2$ convergence from above and the fact that $\Vert \HY^A-P_{\mathcal{L}_{Y_{S \setminus A}}(t)} \HY^A \Vert_{L^2} < \infty$. Written out we get the statement
\end{comment}
\begin{align*}
&  \BE \left[ \left( \HY^B -P_{\mathcal{L}_{Y_{S \setminus A}}(t)} \HY^B  \right) \overline{\left( \HY^A-P_{\mathcal{L}_{Y_{S \setminus A}}(t)} \HY^A \right)} \right] \\
&\qquad =  \lim_{n \rightarrow \infty}  \BE \left[ \left( \HY^B_n - P_{\mathcal{L}_{Y_{S \setminus A}}(t)} \HY^B_n  \right) \overline{\left( \HY^A-P_{\mathcal{L}_{Y_{S \setminus A}}(t)} \HY^A \right)} \right].
\end{align*}
Since $\HY^B_n \in  \mathcal{L}_{Y_B}(t) + \mathcal{L}_{Y_B}(t, t+1)$, $n \in \N$, and by Step 1 and Step 2, the right-hand side is zero, so the left-hand side is also zero. % and we get the assertion %for limit elements $\HY^B \in \mathcal{L}_{Y_B}(t) \vee \mathcal{L}_{Y_B}(t, t+1)$ also.
%In summary as claimed
%\begin{align*}
%\BE \left[ \left( \HY^B -P_{\mathcal{L}_{Y_{S \setminus A}}(t)} \HY^B  \right) \overline{\left( \HY^A-%P_{\mathcal{L}_{Y_{S \setminus A}}(t)} \HY^A \right) }\right]=0.
%\end{align*}
%for all $\HY^B \in \mathcal{L}_{Y_B}(t+1)$ and $\HY^A\in \mathcal{L}_{Y_A}(t)$, respectively.
Finally, $\mathcal{L}_{Y_B}(t+1) \perp \mathcal{L}_{Y_A}(t) \: \vert \: \mathcal{L}_{Y_{S \setminus A}}(t)$, $t \in \R$. \\
(b)  $\Rightarrow$ (a):    Suppose that $\mathcal{L}_{Y_B}(t+1) \perp \mathcal{L}_{Y_A}(t) \: \vert \: \mathcal{L}_{Y_{S \setminus A}}(t)$ for $t \in \R$. Since $\mathcal{L}_{Y_B}(t, t+1)\subseteq \mathcal{L}_{Y_B}(t+1)$ it follows that $\mathcal{L}_{Y_B}(t, t+1) \perp \mathcal{L}_{Y_A}(t) \: \vert \: \mathcal{L}_{Y_{S \setminus A}}(t)$ for $t \in \R$. \\
Similarly, we can conclude by subset arguments that (a) $\Rightarrow$ (c) and (c) $\Rightarrow$ (d) hold.\\
 (c) $\Rightarrow$ (a):
Suppose that $\ell_{Y_B}(t,t+1) \perp \ell_{Y_A}(-\infty, t) \: \vert \: \mathcal{L}_{Y_{S \setminus A}}(t)$.
%First of all, we show that $\mathcal{L}_{Y_B}(t,t+1) \perp \ell_{Y_A}(t) \: \vert \: \mathcal{L}_{Y_{S \setminus A}}(t)$.
Let $\HY^B \in \mathcal{L}_{Y_B}(t,t+1)$. Then there exists a sequence $\HY^B_n \in  \ell_{Y_B}(t,t+1)$, $n \in \N$, such that $\lim_{n \rightarrow \infty}\Vert \HY^B - \HY^B_n \Vert_{L^2}=0$.  For $\HY^A \in \ell_{Y_A}(-\infty, t)$ \eqref{limit E} yields
\begin{align*}
  & \BE \left[
\left( \HY^B - P_{\mathcal{L}_{Y_{S \setminus A}}(t)} \HY^B \right)
\overline{\left( \HY^A - P_{\mathcal{L}_{Y_{S \setminus A}}(t)} \HY^A \right)}
\right] \\
&\qquad =   \lim_{n \rightarrow \infty}   \BE \left[
\left( \HY_n^B - P_{\mathcal{L}_{Y_{S \setminus A}}(t)} \HY_n^B\right)
\overline{\left( \HY^A - P_{\mathcal{L}_{Y_{S \setminus A}}(t)} \HY^A \right)}
\right].
\end{align*}
We apply the assumption (c) %$\ell_{Y_B}(t,t+1) \perp \ell_{Y_A}(t) \: \vert \: \mathcal{L}_{Y_{S \setminus A}}(t)$
to obtain that the expression on the right-hand side is zero. In conclusion, $\mathcal{L}_{Y_B}(t,t+1) \perp \ell_{Y_A}(-\infty,t) \: \vert \: \mathcal{L}_{Y_{S \setminus A}}(t)$. 
In a second step, one can now show analogously that $\mathcal{L}_{Y_B}(t,t+1) \perp \mathcal{L}_{Y_A}(t) \: \vert \: \mathcal{L}_{Y_{S \setminus A}}(t)$. \\
(d) $\Rightarrow$ (c):
Suppose that $\ell_{Y_b}(s,s) \perp \ell_{Y_a}(s',s') \: \vert \: \mathcal{L}_{Y_{S \setminus A}}(t)$ for  $a \in A$, $b \in B$, $s \in [t,t+1]$, $s'\leq t$,  $t \in \R$. Let $\HY^B \in  \ell_{Y_B}(t,t+1)$. Then there are coefficients $\gamma_{b,i}\in \C$ and time points $t \leq t_1 \leq \cdots \leq t_n \leq t+1$, $n\in \N$, such that $\mathbb{P}$-a.s.
\begin{align*}
\HY^B = \sum_{i=1}^n \sum_{b\in B} \gamma_{b,i} Y_b(t_i).
\end{align*}
For $\HY^a \in \ell_{Y_a}(s',s')$ by linearity of the orthogonal projection and the expectation
\begin{align*}
& \: \BE \left[
\left(\HY^B - P_{\mathcal{L}_{Y_{S \setminus A}}(t)} \HY^B \right)
\overline{\left( \HY^a - P_{\mathcal{L}_{Y_{S \setminus A}}(t)} \HY^a \right)}\right] \\
%& \qquad =   \: \BE \left[
%\left( \sum_{i=1}^n \sum_{b\in B} \gamma_{b,i} Y_b(t_i) - P_{\mathcal{L}_{Y_{S \setminus A}}(t)} \sum_{i=1}^n \sum_{b\in B} \gamma_{b,i} Y_b(t_i) \right) \overline{\left( \HY^a - P_{\mathcal{L}_{Y_{S \setminus A}}(t)} \HY^a \right)}\right] \\
& \qquad =  \: \sum_{i=1}^n \sum_{b\in B} \gamma_{b,i} \: \BE \left[
\left( Y_b(t_i)  - P_{\mathcal{L}_{Y_{S \setminus A}}(t)} Y_b(t_i)  \right) \overline{\left( \HY^a - P_{\mathcal{L}_{Y_{S \setminus A}}(t)} \HY^a \right)}\right].
\end{align*}
Finally, we apply assumption (d) to obtain that the expectation on the right-hand side is zero. Thus, $\ell_{Y_B}(t,t+1) \perp \ell_{Y_a}(s',s') \: \vert \: \mathcal{L}_{Y_{S \setminus A}}(t)$
for  $a \in A$, $s'\leq t$,  $t \in \R$. In a second step, one can now show analogously that 
$\ell_{Y_B}(t,t+1) \perp \ell_{Y_A}(-\infty,t) \: \vert \: \mathcal{L}_{Y_{S \setminus A}}(t)$ \mbox{for $t\in \R$.}
\end{proof}

\section{Proofs of Section~\ref{sec:path_diagrams}} \label{suppl:proofs_section_5}
%Let us now discuss separability and conditional linear separation, taking \Cref{properties of conditional orthogonality} into account. We start with the separability.

%If we look back at \Cref{properties of conditional orthogonality} and in particular the property of intersection (C5), two main questions arise and motivate the following considerations. First of all, we have to know, if the generated linear subspaces are separable. Second, are there conditions which are easy to verify, such that $\mathcal{L}_{Y_A}(t)$ and $\mathcal{L}_{Y_B}(t)$ conditionally linearly separated by $\mathcal{L}_{Y_C}(t)$? We start with the first question and discuss separability. Therefore we require the continuity in mean square, as we have announced before. However, we point out that it is sufficient to assume that the right and left limits in mean square exist for the mean square integrable process $\CY_V$ (\cite{CR61}, Lemma 1). As usual, we limit the proof to those spaces that we need later.

\begin{proof}[Proof of \Cref{Separabilität}]
 We refer to \cite{CR61}, Lemma 1, for the proof of $\ell_{Y_A}(-\infty,\infty)$ being separable. If $M_A$ is a countable dense subset of $\ell_{Y_A}(-\infty,\infty)$, it is also a countable dense subset of $\mathcal{L}_{Y_A}$, which can be explained as follows. Let $Y \in \mathcal{L}_{Y_A}$ be the limit in mean square of a sequence $Y_n \in \ell_{Y_A}(-\infty,\infty)$, $n \in \N$, and let $\varepsilon>0$. Then there exists a $n_0\in \N$ such that \mbox{$\Vert Y-Y_{n}\Vert_{L^2} < \frac{\varepsilon}{2}$} for $n\geq n_0$. Furthermore, we can chose $m_v \in M_A$ such that \mbox{$\Vert Y_{n_0} - m_v \Vert_{L^2} < \frac{\varepsilon}{2}$,} since $M_A$ is dense in $\ell_{Y_A}(-\infty,\infty)$. Then
\begin{align*}
\Vert Y - m_v \Vert_{L^2} \leq \Vert Y -Y_{n_0} \Vert_{L^2} + \Vert Y_{n_0} - m_v \Vert_{L^2} < \varepsilon,
\end{align*}
and thus, $M_A$ is a countable dense subset of $\mathcal{L}_{Y_A}$, and $\mathcal{L}_{Y_A}$ is separable. Similarly, we obtain that $\mathcal{L}_{Y_A}(t)$ and $\mathcal{L}_{Y_A}(s,t)$ are separable using, e.g., $P_{\mathcal{L}_{Y_A}(t)}M_A$ and $P_{\mathcal{L}_{Y_A}(s,t)}M_A$ as countable dense subsets of
$\mathcal{L}_{Y_A}(t)$ and $\mathcal{L}_{Y_A}(s,t)$, respectively.
\end{proof}

%We continue with the concept of two linear spaces being conditionally linearly separated by a third, as in \Cref{definition linear separation} and introduce a first sufficient criterion.

\begin{comment}
\begin{lemma}\label{Lemma lineare Separabilität}
%Let $\CY_V=(Y_V(t))_{t\in \R}$ be a $k$-dimensional process that satisfies \Cref{Assumption 1 und 2}.
Suppose that
\begin{align*}
\mathcal{L}_{Y_{A}}(t) \cap \mathcal{L}_{Y_{B }}(t) =  \{0\}
\quad \text{ and } \quad
\mathcal{L}_{Y_{A}}(t) + \mathcal{L}_{Y_{B}}(t) = \mathcal{L}_{Y_{A}}(t) \vee \mathcal{L}_{Y_{B}}(t),
\end{align*}
$\mathbb{P}$-a.s., for all $t\in \R$ and for all disjoint subsets $A,B \subseteq V$. Then
\begin{align*}
\mathcal{L}_{Y_{A\cup C}}(t) \cap \mathcal{L}_{Y_{B \cup C}}(t) =  \mathcal{L}_{Y_C}(t),
\end{align*}
$\mathbb{P}$-a.s. for all $t\in \R$ and disjoint subsets $A,B,C \subseteq V$.
\end{lemma}
\end{comment}

\begin{proof}[Proof of \Cref{Lemma lineare Separabilität}]
Let $t\in \R$ and $A, B, C \subseteq V$ be disjoint. Then $\mathcal{L}_{Y_C}(t)\subseteq \mathcal{L}_{Y_{A\cup C}}(t) \cap \mathcal{L}_{Y_{B \cup C}}(t)$ follows immediately. For the relation $ \mathcal{L}_{Y_{A\cup C}}(t) \cap \mathcal{L}_{Y_{B \cup C}}(t)\subseteq\mathcal{L}_{Y_C}(t)$, suppose \mbox{$Y \in \mathcal{L}_{Y_{A\cup C}}(t) \cap \mathcal{L}_{Y_{B \cup C}}(t)$.} Then by assumption
\begin{align*}
Y \in \mathcal{L}_{Y_{A\cup C}}(t) = \mathcal{L}_{Y_{A}}(t) +\mathcal{L}_{Y_{C}}(t) \quad \text{and} \quad
Y \in \mathcal{L}_{Y_{B\cup C}}(t) = \mathcal{L}_{Y_{B}}(t) +\mathcal{L}_{Y_{C}}(t).
\end{align*}
 Therefore, $\HY=\HY^A+\HY^C = \HZ^B+\HZ^C$ $\mathbb{P}$-a.s., where $\HY^A \in \mathcal{L}_{Y_{A}}(t)$, $\HZ^B\in \mathcal{L}_{Y_{B}}(t)$ and \mbox{$\HY^C, \HZ^C \in \mathcal{L}_{Y_{C}}(t)$.} This yields to
\begin{align*}
\HY^A- \HZ^B = \HZ^C - \HY^C \in \mathcal{L}_{Y_{A \cup B}}(t) \cap \mathcal{L}_{Y_{C}}(t),
\end{align*}
where $\mathcal{L}_{Y_{A \cup B}}(t) \cap \mathcal{L}_{Y_{C}}(t)= \{0\}$ $ \mathbb{P}$-a.s. by assumption. Finally,
\begin{align*}
\HY^A = \HZ^B \in \mathcal{L}_{Y_{A }}(t) \cap \mathcal{L}_{Y_{B}}(t)=\{0\} \quad \mathbb{P}\text{-a.s.},
\end{align*}
where we used again the assumption, and as claimed $\HY=\HY^C \in \mathcal{L}_{Y_C}(t)$ $\mathbb{P}$-a.s.
\end{proof}

\begin{proof}[Proof of \Cref{Eigenschaft für Hilfslemma B3}]
Let $t\in \R$ and $A \subseteq V$. Obviously, the relation $\supseteq$ holds.
%$\mathcal{L}_{Y_{V\setminus A}}(t) \subseteq \mathcal{L}_{Y_A}(t-k) \vee \mathcal{L}_{Y_{V\setminus A}}(t)$ for $k\in \N$. Hence the first direction
%\begin{align*}
%\mathcal{L}_{Y_{V\setminus A}}(t) \subseteq \bigcap_{k \in \N} \left( \mathcal{L}_{Y_A}(t-k) \vee %\mathcal{L}_{Y_{V\setminus A}}(t)\right).
%\end{align*}
For $\subseteq$
%$ \bigcap_{k \in \N} (\mathcal{L}_{Y_A}(t-k) \vee \mathcal{L}_{Y_{V\setminus A}}(t))\subseteq \mathcal{L}_{Y_{V\setminus A}}(t)$
suppose that
\begin{align*}
\HY \in \bigcap_{k \in \N} \left( \mathcal{L}_{Y_A}(t-k) \vee \mathcal{L}_{Y_{V\setminus A}}(t)\right).
\end{align*}
Then, $\HY \in  \mathcal{L}_{Y_A}(t-k) \vee \mathcal{L}_{Y_{V\setminus A}}(t)=\mathcal{L}_{Y_A}(t-k) + \mathcal{L}_{Y_{V\setminus A}}(t)$ for $k \in \N$  due to \Cref{Assumption an Dichte} respectively \Cref{Eigenschaften der linearen Räume}. Hence, there exist $\HY^A_{t-k} \in \mathcal{L}_{Y_A}(t-k)$ and $\HY^{V\setminus A}_{t-k}\in \mathcal{L}_{Y_{V\setminus A}}(t)$, such that $\HY = \HY^A_{t-k} + \HY^{V\setminus A}_{t-k}$ $\mathbb{P}$-a.s. for $k\in\N$, and
\begin{align*}
\HY^A_{t-1} - \HY^A_{t-k}
= \HY^{V\setminus A}_{t-k} -\HY^{V\setminus A}_{t-1}
&\in \mathcal{L}_{Y_A}(t-1) \cap \mathcal{L}_{Y_{V\setminus A}}(t-1) = \{ 0\} \quad \mathbb{P}\text{-a.s.}
\end{align*}
 due to \Cref{Eigenschaften der linearen Räume} again. Therefore, % and since the linear spaces increase in the index set,
\begin{align*}
\HY^A_{t-1} = \HY^A_{t-k} \in \mathcal{L}_{Y_A}(t-1) \cap \mathcal{L}_{Y_A}(t-k) \subseteq \mathcal{L}_{Y_V}(t-1) \cap \mathcal{L}_{Y_V}(t-k) \quad \mathbb{P}\text{-a.s.}
\end{align*}
Since $k\in\N$ is arbitrary and due to \Cref{Assumption purely nondeterministic of full rank},
\begin{align*}
\HY^A_{t-1} \in \bigcap_{k\in \N} \mathcal{L}_{Y_V}(t-k) =   \mathcal{L}_{Y_V}(-\infty) = \{0\} \quad \mathbb{P}\text{-a.s.}
\end{align*}
But then  $\HY = \HY^{V\setminus A}_{t-1} \in \mathcal{L}_{Y_{V\setminus A}}(t)$ $\mathbb{P}$-a.s. as claimed.
\end{proof}

\section{Proofs of Section~\ref{sec:CGMCAR}} \label{suppl:section 6}
\subsection{Proof of Proposition~\ref{Dichteannahme erfüllt}}\label{subsec:proofdensity}
Let us start with the simple \Cref{Assumption purely nondeterministic of full rank}.

\begin{proof}[Proof of \Cref{Assumption purely nondeterministic of full rank}]
According to \Cref{Projektion für MCAR für S=V} we obtain for $v\in V$ and $t\in \R$   that
\begin{align*}
\Vert P_{\mathcal{L}_{Y_V}(t)}Y_v(t+h) \Vert^2_{L^2}
= \Vert e_v^\top  \BFC e^{\BA h} X(t)  \Vert^2_{L^2}
=  e_v^\top  \BFC e^{\BA h} c_{X X}(0) e^{\BA^\top  h} \BFC^\top  e_v
\rightarrow 0,
\end{align*}
as $h\rightarrow \infty$, since $\sigma(\BA) \subseteq (-\infty, 0) + i \R$. Then \cite{RO67}, III, eq. (2.1) and Theorem 2.1 conclude that $\CY_V$ is purely non-deterministic and hence, \Cref{Assumption purely nondeterministic of full rank} holds.
\end{proof}

For  Assumption \ref{Assumption an Dichte} first note that $f_{Y_VY_V}(\cdot)$ has the representation as given in \Cref{Lemma 5.2} (d) and since $\BS_L>0$ we have
$f_{Y_VY_V}(\cdot)>0$ as well.
Now, to the second part of \Cref{Assumption an Dichte}, where we claim that there exists $0<\varepsilon<1$, such that
\begin{align*}
f_{Y_AY_A}(\lambda)^{-1/2}f_{Y_AY_B}(\lambda)f_{Y_BY_B}(\lambda)^{-1}f_{Y_BY_A}(\lambda)f_{Y_AY_A}(\lambda)^{-1/2} \leq_L
(1-\varepsilon)I_{\alpha},
\end{align*}
for (almost) all $\lambda \in \R$ and for all disjoint subsets $A,B\subseteq V$, $\# A=\alpha$. To prove this, we require several auxiliary lemmata.

\begin{lemma} \label{Schranke Intervall}
Let $\CY_V$ be a causal %$k$-dimensional
MCAR$(p)$ process with \mbox{$\BS_L>0$.}
%and $\BS_L$ be non singular.
Further, let $A,B\subseteq V$, $A\cap B=\emptyset$, and $\# A=\alpha$.
Then for each compact interval $K\subset \R$ there exists an $0<\varepsilon_K<1$, such that
\begin{align*}
f_{Y_AY_A}(\lambda)^{-1/2}f_{Y_AY_B}(\lambda)f_{Y_BY_B}(\lambda)^{-1}f_{Y_BY_A}(\lambda)f_{Y_AY_A}(\lambda)^{-1/2} \leq_L (1-\varepsilon_K) I_\alpha \quad \forall \: \lambda \in K.
\end{align*}
\end{lemma}

\begin{proof}
%Let $A,B\subseteq V$, $A\cap B=\emptyset$, $\# A=\alpha$, $\# B =\beta$, and $K\subset \R$ be a compact interval.
As $\det(P(i \lambda))$ has no zeros due to $\mathcal{N}(P) \subseteq (-\infty,0)+i\R$, the spectral density matrix
$f_{Y_VY_V}(\lambda)=1/(2\pi) P(i\lambda)^{-1} \BS_L (P(-i\lambda)^{-1})^\top ,$
$\lambda\in \R,$
is continuous. Then \cite{BH97}, Corollary VI.1.6, states that there exist continuous functions $\sigma_1(\lambda),\ldots,\sigma_k(\lambda)$ which are the eigenvalues of $f_{Y_VY_V}(\lambda)$. Since $f_{Y_VY_V}(\lambda)$ is hermitian and positive definite, these eigenvalues are in $(0,\infty)$ and in particular, they can be ordered as $0<\sigma_1(\lambda) \leq \ldots \leq \sigma_k(\lambda)$ for $\lambda \in \R$, see \cite{BH97}, p.~154.
%Since $f_{Y_VY_V}(\lambda)>0$ for all $\lambda \in \R$ it follows that
%\begin{align*}
%0<\sigma_1(\lambda) \leq \ldots \leq \sigma_k(\lambda),
%\end{align*}
%for all $\lambda \in \R$.
Furthermore, \cite{BE09}, Lemma 8.4.1, provides
$\sigma_1(\lambda) I_k \leq_L f_{Y_VY_V}(\lambda) \leq_L \sigma_k(\lambda) I_k,$
and due to \cite{BE09}, Proposition 8.1.2, we obtain
\begin{align*}
\sigma_1(\lambda) I_{\alpha+\beta} \leq_L f_{Y_{A \cup B}Y_{A \cup B}}(\lambda)  \quad \text{ and } \quad
f_{Y_{A}Y_{A}}(\lambda) \leq_L \sigma_k(\lambda) I_\alpha \quad \forall\,\lambda\in \R.
\end{align*}
Let $ \lambda\in\R$. Using \cite{BE09}, Proposition 8.1.2, again gives
\begin{align*}
\left(f_{Y_{A \cup B}Y_{A \cup B}}(\lambda)\right)^{-1} \leq_L \frac{1}{\sigma_1(\lambda)} I_{\alpha+\beta},
\end{align*}
and together with \cite{BE09}, Proposition 8.2.5, we receive
\begin{align*}
\left(f_{Y_A Y_A}(\lambda) - f_{Y_A Y_B}(\lambda) f_{Y_B Y_B}(\lambda)^{-1} f_{Y_B Y_A}(\lambda) \right)^{-1} \leq_L \frac{1}{\sigma_1(\lambda)} I_{\alpha}.
\end{align*}
Now \cite{BE09}, Proposition 8.1.2, yields
\begin{align*}
\sigma_1(\lambda) I_{\alpha} \leq_L f_{Y_A Y_A}(\lambda) - f_{Y_A Y_B}(\lambda) f_{Y_B Y_B}(\lambda)^{-1} f_{Y_B Y_A}(\lambda).
\end{align*}
 If we combine this result with $f_{Y_{A}Y_{A}}(\lambda) \leq_L \sigma_k(\lambda) I_\alpha$ from above, we obtain
\begin{align*}
\frac{\sigma_1(\lambda)}{\sigma_k(\lambda)}f_{Y_{A}Y_{A}}(\lambda) \leq \sigma_1(\lambda) I_{\alpha} \leq_L f_{Y_A Y_A}(\lambda) - f_{Y_A Y_B}(\lambda) f_{Y_B Y_B}(\lambda)^{-1} f_{Y_B Y_A}(\lambda).
\end{align*}
Thus,
\begin{align*}
f_{Y_A Y_B}(\lambda) f_{Y_B Y_B}(\lambda)^{-1} f_{Y_B Y_A}(\lambda) \leq_L \left( 1- \frac{\sigma_1(\lambda)}{\sigma_k(\lambda)}\right)f_{Y_{A}Y_{A}}(\lambda),
\end{align*}
and \cite{BE09}, Proposition 8.1.2, finally provides
\begin{align*}
f_{Y_{A}Y_{A}}(\lambda)^{-1/2} f_{Y_A Y_B}(\lambda) f_{Y_B Y_B}(\lambda)^{-1} f_{Y_B Y_A}(\lambda)f_{Y_{A}Y_{A}}(\lambda)^{-1/2}  \leq_L \left( 1- \frac{\sigma_1(\lambda)}{\sigma_k(\lambda)}\right)I_\alpha.
\end{align*}
 We now differentiate two cases to prove the assertion. First, let $\sigma_1(\lambda)/\sigma_k(\lambda)=1$ for all $\lambda \in K$. Then
\begin{align*}
f_{Y_{A}Y_{A}}(\lambda)^{-1/2} f_{Y_A Y_B}(\lambda) f_{Y_B Y_B}(\lambda)^{-1} f_{Y_B Y_A}(\lambda)f_{Y_{A}Y_{A}}(\lambda)^{-1/2}  \leq_L 0_\alpha \quad \forall\,\lambda\in K,
\end{align*}
 and the assertion holds with any $0<\varepsilon_K<1$. W.l.o.g. we set $\varepsilon_K=1/2$. In the second case, let $\sigma_1(\lambda)/\sigma_k(\lambda)<1$ for at least one $\lambda \in K$.  Since the continuous function $\sigma_1(\lambda)/\sigma_k(\lambda)$ achieves its minimum on the compact set $K$, we define
\begin{align*}
\varepsilon_K = \min_{\lambda \in K} \frac{\sigma_1(\lambda)}{\sigma_k(\lambda)}
\end{align*}
and obtain that $0<\varepsilon_K<1$ as well as the upper bound
\begin{align*}
f_{Y_{A}Y_{A}}(\lambda)^{-1/2} f_{Y_A Y_B}(\lambda) f_{Y_B Y_B}(\lambda)^{-1} f_{Y_B Y_A}(\lambda)f_{Y_{A}Y_{A}}(\lambda)^{-1/2}  \leq_L (1- \varepsilon_K) I_\alpha  \quad \forall\,\lambda\in K.
\end{align*}
Since $\sigma_1(\lambda)/\sigma_k(\lambda)\leq 1$ for all $\lambda \in K$, these are all possible cases; the assertion holds.
\end{proof}

We now establish a relationship between the convergence of matrices in norm and the Loewner order, which we could not find in the literature. The result is similar to the epsilon definition of the convergence of sequences.

\begin{lemma}\label{Konvergenz also Schranke}
Let $F(\lambda)\in M_{\alpha}(\R)$, $\lambda \in \R$, and $M\in M_{\alpha}(\R)$ such that
$\lim_{\vert \lambda \vert \rightarrow \infty}  \Vert F(\lambda) - M \Vert = 0.$
Then for any $\varepsilon^*>0$ there exists a $\lambda^* \in \R$ such that
\begin{align*}
F(\lambda) \leq_L M + \varepsilon^* I_\alpha \quad \forall \, \vert \lambda \vert \geq \lambda^*.
\end{align*}
\end{lemma}

\begin{proof}
Let $\varepsilon^*>0$. Due to $\lim_{\vert\lambda\vert \rightarrow \infty} \Vert F(\lambda) - M \Vert = 0$ it obviously holds that %in any norm, it holds that
%\begin{align*}
%\lim_{\vert\lambda\vert \rightarrow \infty} \sum _{i=1}^{\alpha}\sum _{j=1}^{\alpha} \vert \left(F(\lambda) - M\right)_{ij} \vert =0.
%\end{align*}
%Thus,
$\lim_{\vert\lambda\vert \rightarrow \infty}  \vert \left(F(\lambda) - M\right)_{ij} \vert =0$
for  $i,j =1,\ldots,\alpha$. It follows that for $\varepsilon^*>0$, $k\geq \alpha$, there exists a $\lambda^* \in \R$ such that
\begin{align*}
\vert \left(F(\lambda) - M\right)_{ij} \vert \leq \frac{\varepsilon^*}{k},
\end{align*}
for all $\vert \lambda \vert \geq \lambda^*$, $i,j =1,\ldots,\alpha$. Now, for any $x\in \R^\alpha$ and $\vert \lambda \vert \geq \lambda^*$ we receive that
\begin{align*}
x^\top  \left(F(\lambda) - M \right)x
%&= \sum _{i=1}^{\alpha}\sum _{j=1}^{\alpha} x_i \left(F(\lambda) - M\right)_{ij} x_j \\
&= \frac{1}{4}\sum _{i=1}^{\alpha}\sum _{j=1}^{\alpha} \left( (x_i+x_j)^2 \left(F(\lambda) - M\right)_{ij} - (x_i-x_j)^2 \left(F(\lambda) - M\right)_{ij} \right)\\
%&\leq \frac{1}{4}\sum _{i=1}^{\alpha}\sum _{j=1}^{\alpha} \left((x_i+x_j)^2 \frac{\varepsilon^*}{k} + (x_i-x_j)^2 \frac{\varepsilon^*}{k}\right)\\
&\leq  \frac{\varepsilon^*}{2k}\sum _{i=1}^{\alpha}\sum _{j=1}^{\alpha} \left(x_i^2+x_j^2\right) \\
%&= \frac{\varepsilon^*}{2k}\sum _{i=1}^{\alpha} \left( \alpha  x_i^2 + x^\top x \right) \\
%&= \frac{\varepsilon^*}{2k} \left( \alpha x^\top x  + \alpha x^\top x \right) \\
&= \frac{\varepsilon^*\alpha }{k} x^\top  x.
%&\leq \varepsilon^*x^\top  x.
\end{align*}
Thus, since $k\geq \alpha$, %for each $\varepsilon^*>0$ there is a $\lambda^* \in \R$ such that
%\begin{align*}
%x^\top  \left(F(\lambda) - M \right)x  \leq \varepsilon^* x^\top  x, \quad \forall \: x\in \R^\alpha, \vert \lambda \vert \geq \lambda^*.
%\end{align*}
%That is
$F(\lambda) - M \leq_L \varepsilon^* I_\alpha$  and $F(\lambda) \leq_L M  + \varepsilon^* I_\alpha$ for  $\vert \lambda \vert \geq \lambda^*.$
\end{proof}

\begin{lemma}\label{Grenzübergang}
Let $\CY_V$ be a causal %$k$-dimensional
MCAR$(p)$ process with \mbox{$\BS_L>0$.} Further, let $A,B\subseteq V$, $A\cap B=\emptyset$, and $\# A=\alpha$. Define
\begin{align*}
F(\lambda) &= f_{Y_{A}Y_{A}}(\lambda)^{-1/2} f_{Y_A Y_B}(\lambda) f_{Y_B Y_B}(\lambda)^{-1} f_{Y_B Y_A}(\lambda)f_{Y_{A}Y_{A}}(\lambda)^{-1/2}, \\
M &= H_{AA}^{-1/2} H_{AB} H_{BB}^{-1} H_{BA} H_{AA}^{-1/2},
\end{align*}
where for $S,S_1,S_2\subseteq V$,
\begin{align*}
H_{S_1S_2} = E_{S_1}^\top  \BFC \: \BB \BS_L \BB^\top  \BFC^\top  E_{S_2}
\quad \text{ and } \quad
[E_S]_{ij} =
\begin{cases}
1 & i=j \in S,\\
0 & else.
\end{cases}.
\end{align*}
Then for $\varepsilon^*>0$ there exists a $\lambda^* >0$ such that
\begin{align*}
F(\lambda) \leq_L M + \varepsilon^* I_\alpha \quad \forall \, \vert \lambda \vert \geq \lambda^*.
\end{align*}
\end{lemma}

\begin{proof}
%Let $A,B\subseteq V$, $A\cap B=\emptyset$, $\# A = \alpha$. %Recall that
%\begin{align*}
%f_{Y_VY_V} (\lambda)= \frac{1}{2\pi} \BFC \left(i\lambda I_{kp}- \BA \right)^{-1} \BB \BS_L \BB^\top   \left(-i\lambda I_{kp}-\BA^\top  \right)^{-1} \BFC^\top .
%\end{align*}
%Furthermore
\cite{BE09}, (4.4.23), states that
\begin{align*}
\left( i\lambda I_{kp}-\BA \right)^{-1} = \sum_{n=0}^{kp-1} \frac{ \left(i\lambda \right)^n}{\chi_\BA(i\lambda)} \Delta_n,
\end{align*}
where $\Delta_n \in \R^{kp \times kp}$, $\Delta_{kp-1}=I_{kp}$, and
\begin{align*}
\chi_\BA(z) = z^{kp} + \gamma_{kp-1}z^{kp-1}+ \cdots + \gamma_1 z + \gamma_0, \quad z\in\C,
\end{align*}
 is the characteristic polynomial of $\BA$ with $\gamma_{1},\ldots,\gamma_{kp-1} \in \R$, see \cite{BE09}, (4.4.3). % Thus, additionally
%\begin{align*}
%\left(-i\lambda I_{kp}-\BA^\top \right)^{-1}
%= \overline{\left(i\lambda I_{kp}-\BA \right)^{-1}}^\top 
%= \sum_{n=0}^{kp-1} \frac{\left(- i\lambda \right)^n}{\chi_\BA(-i\lambda)} \Delta_n^\top .
%\end{align*}
Inserting this representation in the spectral density given in \Cref{Lemma 5.2} yields
\begin{align*}
f_{Y_VY_V}(\lambda)
%&=\frac{1}{2\pi} \BFC \left(i\lambda I_{kp}- \BA \right)^{-1} \BB \BS_L \BB^\top   \left(-i\lambda I_{kp}-\BA^\top  \right)^{-1} \BFC^\top \\
= \frac{1}{2\pi}  \sum_{m=0}^{kp-1} \sum_{n=0}^{kp-1}
\frac{ \left(i\lambda \right)^m}{\chi_\BA(i\lambda)} \frac{ \left(- i\lambda \right)^n}{\chi_\BA(-i\lambda)} \BFC \Delta_m \BB \BS_L \BB^\top   \Delta_n^\top  \BFC^\top .
%&= \frac{1}{2\pi \chi_\BA(i\lambda) \chi_\BA(-i\lambda)}  \sum_{m=0}^{kp-1} \sum_{n=0}^{kp-1}
%\left(i\lambda \right)^m \left(- i\lambda \right)^n \BFC \Delta_m \BB \BS_L \BB^\top   \Delta_n^\top  \BFC^\top  \\
%&= \frac{1}{2\pi \chi_\BA(i\lambda) \chi_\BA(-i\lambda)}  \sum_{m=0}^{kp-1} \sum_{n=0}^{kp-1}
%\left(i\lambda \right)^{m+n} \left(- 1\right)^n \BFC \Delta_m \BB \BS_L \BB^\top   \Delta_n^\top  \BFC^\top ,
\end{align*}
 In particular, we have, for $a,b\in V$,
 
\begin{align*}
f_{Y_aY_b}(\lambda)= \frac{1}{2\pi \chi_\BA(i\lambda) \chi_\BA(-i\lambda)}  \sum_{m=0}^{kp-1} \sum_{n=0}^{kp-1}
\left(i\lambda \right)^{m+n} \left(- 1\right)^n e_a^\top  \BFC \Delta_m \BB \BS_L \BB^\top   \Delta_n^\top  \BFC^\top e_b.
\end{align*}
From this rational function, we can specify the asymptotic behaviour. The numerator contains a complex polynomial of maximal degree $2kp-2$ with leading coefficient
\begin{align*}
 e_a^\top  \BFC \Delta_{kp-1} \BB \BS_L \BB^\top   \Delta_{kp-1}^\top  \BFC^\top e_b
=  e_a^\top  \BFC \BB \BS_L \BB^\top   \BFC^\top e_b,
\end{align*}
that may be zero. The denominator is a complex polynomial of degree $2kp$ with leading coefficient $2\pi$. Combining both gives
%Differentiating between $e_a^\top  \BFC \BB \BS_L \BB^\top   \BFC^\top e_b=0$ and $e_a^\top  \BFC \BB \BS_L \BB^\top   \BFC^\top e_b\neq 0$ yields
\begin{align*}
\lim_{\vert \lambda \vert \rightarrow \infty} \vert 2 \pi \lambda^2 f_{Y_aY_b}(\lambda) - e_a^\top  \BFC \BB \BS_L \BB^\top   \BFC^\top e_b \vert =0.
\end{align*}
%for all $a,b\in V$. Now we obtain for $A,B\subseteq V$
%\begin{align*}
%\lim_{\vert \lambda \vert \rightarrow \infty} \sum_{a \in A} \sum_{b \in B} \vert 2 \pi \lambda^2 f_{Y_aY_b}(\lambda) - e_a^\top  \BFC \BB \BS_L \BB^\top   \BFC^\top  e_b \vert = 0,
%\end{align*}
%respectively due to the equivalence of matrixnorms (\cite{BU18}, Theorem 1.2.5.)
Finally, for $S_1,S_2\subseteq V$ we receive
\begin{align} \label{eqC1}
\lim_{\vert \lambda \vert \rightarrow \infty} \left\| 2 \pi \lambda^2 f_{Y_{S_1}Y_{S_2}}(\lambda) - H_{S_1S_2} \right\|
%= \lim_{\vert \lambda \vert \rightarrow \infty} \left\| 2 \pi \lambda^2 f_{Y_{S_1}Y_{S_2}}(\lambda) - E_{S_1}^\top  \BFC \BB \BS_L \BB^\top   \BFC^\top  E_{S_2} \right\|
= 0.
\end{align}
%Analogous one obtains
%\begin{align*}
%\lim_{\vert \lambda \vert \rightarrow \infty} \Vert 2 \pi \lambda^2 f_{Y_AY_A}(\lambda) - H_{AA} \Vert
%= \lim_{\vert \lambda \vert \rightarrow \infty} \Vert 2 \pi \lambda^2 f_{Y_AY_A}(\lambda) - E_A^\top  \BFC \BB \BS_L \BB^\top   \BFC^\top  E_A \Vert
%= 0, \\
%\lim_{\vert \lambda \vert \rightarrow \infty} \Vert 2 \pi \lambda^2 f_{Y_BY_B}(\lambda) - H_{BB} \Vert
%= \lim_{\vert \lambda \vert \rightarrow \infty} \Vert 2 \pi \lambda^2 f_{Y_BY_B}(\lambda) - E_B^\top  \BFC \BB \BS_L \BB^\top   \BFC^\top  E_B \Vert = 0.
%\end{align*}
Since $2 \pi \lambda^2 f_{Y_BY_B}(\lambda)>0$ for $\lambda\neq 0$ as well as $E_B^\top  \BFC \BB \BS_L \BB^\top   \BFC^\top  E_B>0$, \cite{BU18}, Corollary 1.5.7(ii), provide the continuity of the formation of the inverse and it follows
\begin{align} \label{eqC2}
\lim_{\vert \lambda \vert \rightarrow \infty} \left\| \frac{1}{2 \pi \lambda^2} f_{Y_BY_B}(\lambda)^{-1} - H_{BB}^{-1} \right\| = 0.
\end{align}
In addition, \cite{BH97}, Theorem X.1.1 and  equation (X.2), respectively provide the following inequality for induced matrix norms and $\lambda \neq 0$,
\begin{align*}
\left\| \sqrt{2 \pi} \vert \lambda \vert  f_{Y_AY_A}(\lambda)^{1/2} - H_{AA}^{1/2} \right\|_{ind}
\leq \left\| 2 \pi \lambda^2 f_{Y_AY_A}(\lambda) - H_{AA} \right\|_{ind}^{1/2}.
\end{align*}
Due to the equivalence of matrix norms and since the right side of the inequality converges to zero, we obtain
\begin{align*}
\lim_{\vert \lambda \vert \rightarrow \infty} \left\| \sqrt{2 \pi} \vert \lambda \vert  f_{Y_AY_A}(\lambda)^{1/2} - H_{AA}^{1/2} \right\|=0.
\end{align*}
Using the positive definiteness of the positive square root and \cite{BU18}, Corollary 1.5.7(ii) again, it follows
\begin{align} \label{eqC3}
\lim_{\vert \lambda \vert \rightarrow \infty} \left\| \frac{1}{\sqrt{2 \pi} \vert \lambda \vert}  f_{Y_AY_A}(\lambda)^{-1/2} - H_{AA}^{-1/2} \right\|=0.
\end{align}
An application of \eqref{eqC1}, \eqref{eqC2}, \eqref{eqC3}, and the submultiplicativity of the induced matrix norm result in
\begin{align*}
& \lim_{\vert \lambda \vert \rightarrow \infty}
 \left\| f_{Y_AY_A}(\lambda)^{-1/2} f_{Y_AY_B}(\lambda) f_{Y_BY_B}(\lambda)^{-1} f_{Y_BY_A}(\lambda) f_{Y_AY_A}(\lambda)^{-1/2} \right.  \\
& %\quad\quad \phantom{\lim_{\vert \lambda \vert \rightarrow \infty}}
\left.
\quad\quad\quad - H_{AA}^{-1/2} H_{AB} H_{BB}^{-1} H_{BA} H_{AA}^{-1/2} \right\|_{ind}
%= & \lim_{\vert \lambda \vert \rightarrow \infty}
% \Vert \frac{1}{\sqrt{2 \pi} \vert \lambda \vert}  f_{Y_AY_A}(\lambda)^{-1/2} 2 \pi \lambda^2 f_{Y_AY_B}(\lambda)
% \frac{1}{2 \pi \lambda^2} f_{Y_BY_B}(\lambda)^{-1} \\
% & \phantom{\lim_{\vert \lambda \vert \rightarrow \infty}} 2 \pi \lambda^2 f_{Y_BY_A}(\lambda) \frac{1}{\sqrt{2 \pi} \vert \lambda \vert}  f_{Y_AY_A}(\lambda)^{-1/2} - H_{AA}^{-1/2} H_{AB} H_{BB}^{-1} H_{BA} H_{AA}^{-1/2}
   =0.
\end{align*}
Therefore,
$
\lim_{\vert \lambda \vert \rightarrow \infty}  \Vert F(\lambda) - M \Vert = 0.
$
Finally, \Cref{Konvergenz also Schranke} provides that for each $\varepsilon^*>0$ there exists a $\lambda^* \in \R$, such that
\begin{equation*}
F(\lambda) \leq_L M + \varepsilon^* I_\alpha \quad \forall \, \vert \lambda \vert \geq \lambda^*. \qedhere
\end{equation*}
\end{proof}

\begin{lemma}\label{Schranke Grenzmatrix}
Let $\CY_V$ be a causal %$k$-dimensional
MCAR$(p)$ process with \mbox{$\BS_L>0$.} Further, let $A,B\subseteq V$, $A\cap B=\emptyset$, and $\#A = \alpha$.
Then there exists an $0<\varepsilon_M<1$, such that
\begin{align*}
M \leq_L (1-\varepsilon_M)I_\alpha,
\end{align*}
where $M$ is defined as in \Cref{Grenzübergang}.
\end{lemma}

\begin{proof}
 First of all one obtains analogous to the proof of \Cref{Schranke Intervall} that
\begin{align*}
M= H_{AA}^{-1/2} H_{AB} H_{BB}^{-1} H_{BA} H_{AA}^{-1/2}  \leq_L \left( 1- \frac{\sigma_1}{\sigma_k}\right)I_\alpha,
\end{align*}
where $\sigma_1$ is the smallest eigenvalue and $\sigma_k$ is the biggest eigenvalue of $\BFC \:\BB \BS_L \BB^\top  \BFC^\top $. Note that the matrix $\BFC \: \BB \BS_L \BB^\top  \BFC^\top $ is positive definite due to $\BS_L>0$, and $\BFC$ %being of full rank by Assumption %\ref{Assumption Q}
and $\BB$ being of full rank. Thus, the eigenvalues $\sigma_1$ and $\sigma_k$ are positive. Here again, we distinguish between two cases. In the first case, let $\sigma_1/\sigma_k=1$. Then
\begin{align*}
M= H_{AA}^{-1/2} H_{AB} H_{BB}^{-1} H_{BA} H_{AA}^{-1/2} \leq_L 0_\alpha,
\end{align*}
and the assertion holds with any $0<\varepsilon_M<1$, we set $\varepsilon_M=1/2$. In the second case, let $\sigma_1/\sigma_k<1$. Then we set $\varepsilon_M = \sigma_1 / \sigma_k$ and obtain that $0<\varepsilon_M<1$ as well as
\begin{align*}
M = H_{AA}^{-1/2} H_{AB} H_{BB}^{-1} H_{BA} H_{AA}^{-1/2}  \leq_L (1- \varepsilon_M) I_\alpha.
\end{align*}
Since $\sigma_1/\sigma_k\leq 1$, these are all cases that may occur and the assertion follows.
\end{proof}

\begin{proof}[Proof of \Cref{Assumption an Dichte}]
With the  notation of \Cref{Grenzübergang} and  \Cref{Schranke Grenzmatrix} we  choose  $0 < \varepsilon^* < \varepsilon_M$. Now, \Cref{Grenzübergang} provides that there exists a $\lambda^*\in \R$ such that
\begin{align*}
F(\lambda) \leq_L M + \varepsilon^* I_\alpha \quad \forall \:\vert \lambda \vert \geq \lambda^*.
\end{align*}
Furthermore, \Cref{Schranke Grenzmatrix} yields
\begin{align*}
F(\lambda) \leq_L M + \varepsilon^* I_\alpha \leq_L (1-\varepsilon_M)I_\alpha + \varepsilon^* I_\alpha
= (1-(\varepsilon_M - \varepsilon^*)) I_\alpha.
\end{align*}
For $\vert \lambda \vert \geq \lambda^*$ we thus find the boundary matrix $(1-(\varepsilon_M - \varepsilon^*)) I_\alpha$, where $0<\varepsilon_M - \varepsilon^*< 1$ due to the choice of $\varepsilon^*$. On the compact interval $K= [-\lambda^*, \lambda^*]$, \Cref{Schranke Intervall} states that there exists an $0<\varepsilon_K <1$, such that
%\begin{align*}
$F(\lambda) \leq_L (1-\varepsilon_K) I_\alpha.$
%\end{align*}
We set $\varepsilon_{AB}= \min\{\varepsilon_K, \varepsilon_M - \varepsilon^* \}$, then $F(\lambda) \leq_L (1-\varepsilon_{AB}) I_\alpha$
for all $\lambda \in \R$. However, $\varepsilon_{AB}$ still depends on $A$ and $B$. Since there are only finitely many such index sets, we set $\varepsilon= \min\{\varepsilon_{AB}: A, B\subseteq V, A\cap B = \emptyset\}$ and obtain that
$0<\varepsilon<1$ and
\begin{align*}
F(\lambda) \leq_L (1-\varepsilon) I_\alpha,
\end{align*}
holds for all $\lambda \in \R$ and for all disjoint subsets $A,B\subseteq V$.
\end{proof}

\subsection{Proof of (\ref{eq:VAR_differences})}

\begin{proof}[Proof of \eqref{eq:VAR_differences}] \label{eq:VAR:differences}
By induction, one can show that
\begin{align*}
Z_V(t+1-n) = \sum_{j=1}^{n} \binom{n-1}{j-1} (-1)^{j-1} \texttt{D}^{(j-1)} Z_V(t),
\end{align*}
for $n=1,\ldots,p$, $t\in \Z$. Then we obtain the representation of the VAR$(p)$ process
\begin{align*}
Z_V(t+1)
 &= \sum_{n=1}^p \sum_{j=1}^{n}  \binom{n-1}{j-1} (-1)^{j-1} \Phi_n \texttt{D}^{(j-1)} Z_V(t) + \varepsilon(t+1) \\
 &= \sum_{j=1}^{p} \sum_{n=j}^p \binom{n-1}{j-1} (-1)^{j-1} \Phi_n \texttt{D}^{(j-1)} Z_V(t) + \varepsilon(t+1).
\end{align*}
%by inserting $Z_V(t+1-n)$ and swapping the order of summation.
Accordingly, we receive the representation of the $b$-th component
\begin{equation*}
Z_b(t+1) = \sum_{j=1}^{p} \sum_{n=j}^p   \binom{n-1}{j-1} (-1)^{j-1} e_b^\top  \Phi_n \texttt{D}^{(j-1)} Z_V(t) + e_b^\top  \varepsilon(t+1), \quad t\in\Z. \qedhere
\end{equation*}
\end{proof}

\end{document}